\documentclass{amsart}

\usepackage{amsfonts,amssymb,mathtools,mathabx}
\usepackage[usenames,dvipsnames]{xcolor}

\usepackage{tikz}
\usetikzlibrary{arrows, positioning, calc, intersections}
\usetikzlibrary{decorations.pathreplacing, decorations.markings}
\usepackage{relsize}					%% 	for scaling math symbols
\usepackage{exscale}					%%	ditto

\usepackage{float}
\usepackage[section]{placeins}

\usepackage{hyperref}
\hypersetup{
    bookmarks=true,         	% show bookmarks bar?
    unicode=false,          		% non-Latin characters in Acrobat’s bookmarks
    pdftoolbar=true,        		% show Acrobat’s toolbar?
    pdfmenubar=true,        	% show Acrobat’s menu?
    pdffitwindow=false,     		% window fit to page when opened
    pdfstartview={FitH},    		% fits the width of the page to the window
   pdftitle={Soliton resolution for the Derivative Nonlinear Schrodinger Equation},    					% title
   	pdfauthor={Robert Jenkins, Jiaqi Liu, Peter Perry, Catherine Sulem},     	% author
    linktocpage=true,			% link to page not section title in TOC
    pdfnewwindow=true,      	% links in new PDF window
    colorlinks=true,       			% false: boxed links; true: colored links
    linkcolor=blue,          		% color of internal links (change box color with linkbordercolor)
    citecolor=PineGreen,    	% color of links to bibliography
    filecolor=magenta,      		% color of file links
    urlcolor=cyan           		% color of external links
}

\usepackage{graphics}

\theoremstyle{plain}
\newtheorem{theorem}{Theorem}[section]
\newtheorem{corollary}[theorem]{Corollary}
\newtheorem{lemma}[theorem]{Lemma}
\newtheorem{proposition}[theorem]{Proposition}

\theoremstyle{remark}
\newtheorem{remark}[theorem]{Remark}

\theoremstyle{definition}
\newtheorem{problem}[theorem]{Problem}
\newtheorem{RHP}[theorem]{Riemann-Hilbert Problem}
\newtheorem{DBAR}[theorem]{$\dbar$-Problem}
\newtheorem{DBARHP}{\revisedtext{$\dbar$/Riemann-Hilbert Problem}}
\newtheorem{definition}[theorem]{Definition}

\DeclareMathOperator{\ad}{ad}
\DeclareMathOperator{\real}{Re}
\DeclareMathOperator{\dist}{dist}
\DeclareMathOperator{\imag}{Im}
\DeclareMathOperator{\sgn}{sgn}

\DeclareMathOperator{\Res}{Res}
\DeclareMathOperator{\res}{res}

\allowdisplaybreaks

\begin{document}

%%%%%%%%%%%%%%%%%%%%%%%
%
%		FOR REVISION
%		
%		Mark out revisions for referee
%		and co-authors
%
%%%%%%%%%%%%%%%%%%%%%%%
%%%%%%%%%%%%%%%%%%%%%%%
%
%		In this version these markup 
%		commands don't do anything
%
%%%%%%%%%%%%%%%%%%%%%%%

%\newcommand{\revisedtext}[1]{{\color{purple} #1}}
\newcommand{\revisedtext}[1]{#1}
\newcommand{\sidenote}[1]{ }

\title[Soliton Resolution for DNLS]{Soliton Resolution for the Derivative Nonlinear Schr\"{o}dinger Equation}
\author{Robert Jenkins}
\author{Jiaqi Liu}
\author{Peter A. Perry}
\author{Catherine Sulem}
\address[Jenkins]{Department of Mathematics, University of Arizona, Tucson, AZ 85721-0089 USA}
\address[Liu]{Department of Mathematics, University of Toronto, Toronto, Ontario M5S 2E4, Canada}
\address[Perry]{ Department of Mathematics,  University of Kentucky, Lexington, Kentucky 40506--0027}
\address[Sulem]{Department of Mathematics, University of Toronto, Toronto, Ontario M5S 2E4, Canada }
\thanks{Peter Perry supported in part by NSF Grant DMS-1208778 and by Simons Foundation Research and Travel Grant 359431}
\thanks {C. Sulem supported in part by NSERC Grant 46179-13}
\date{\today}

\begin{abstract}
We study the derivative nonlinear Schr\"odinger equation for generic initial data  in a weighted 
Sobolev space that can support bright solitons (but exclude spectral singularities).  Drawing on previous well-posedness results, we give a full description of  the long-time behavior of the solutions in the form of a finite sum of 
localized solitons and a dispersive component.  At leading order and in space-time
cones,  the solution has the form of a multi-soliton whose parameters  are slightly modified  from their initial 
values by  soliton-soliton and soliton-radiation interactions.  Our analysis  provides an explicit expression 
for the correction dispersive term.  We use  the nonlinear steepest descent method  of Deift and Zhou \cite{DZ03}  
revisited by  the $\overline{\partial}$-analysis of {McLaughlin-Miller \cite{MM08} and Dieng-McLaughlin \cite{DM08}}, and complemented by the recent 
work of Borghese-Jenkins-McLaughlin \cite{BJM16} on soliton resolution for the focusing 
nonlinear Schr\"{o}dinger equation. 
%An importance consequence  of our result is that it provides a proof of asymptotic stability of  N-soliton solutions.
Our results imply that $N$-soliton solutions of the derivative nonlinear Schr\"{o}dinger equation are asymptotically stable.
\end{abstract}
\maketitle

\tableofcontents

%\input{macros}	   			%% 	\newcommands
%%%%%%%%%%%%%%%%%%%%%%%%%%%%%%%%%%%%%%%
%
%		MACROS
%		File dnls-imrn-arXiv-macros
%		Created 4.7.2017
%
%%%%%%%%%%%%%%%%%%%%%%%%%%%%%%%%%%%%%%%

%%%%%%%%%%%%%%%%%%%%%%%
%
%   Miscellaneous Shorthand
%
%%%%%%%%%%%%%%%%%%%%%%%

\newcommand{\rarr}{\rightarrow}
\newcommand{\darr}{\downarrow}
\newcommand{\uarr}{\uparrow}

\newcommand{\sig}{\sigma_3}

\newcommand{\dotarg}{\, \cdot \, }

\newcommand{\eps}{\varepsilon}
\newcommand{\lam}{\lambda}
\newcommand{\Lam}{\Lambda}

\newcommand{\dee}{\partial}
\newcommand{\dbar}{\overline{\partial}}

\newcommand{\vect}[1]{\boldsymbol{\mathbf{#1}}}

\newcommand{\dint}{\displaystyle{\int}}

\newcommand{\restrict}[2]{\left. {#1} \right|_{#2}}

%%%%%%%%%%%%%%%%%%%%%%%%%%%%%%%%%%%%%%%%%%
%
%		Text shorthands
%
%%%%%%%%%%%%%%%%%%%%%%%%%%%%%%%%%%%%%%%%%

\newcommand{\resp}{resp.\@}
\newcommand{\ifff}{if and only if }
\newcommand{\ie}{i.e.}
%q\newcommand{\eg}{e.g.}
%\newcommand{\cf}{cf.}
%\newcommand{\etal}{et al.}
%\newcommand{\pdf}{p.d.f.}
%\newcommand{\cdf}{c.d.f.}
%\newcommand{\iid}{i.i.d.}

%%%%%%%%%%%%%%%%%%%%%%%%%%%%%%%%%%%%%%%%%%
%
%		Delimiter Shorthands
%
%%%%%%%%%%%%%%%%%%%%%%%%%%%%%%%%%%%%%%%%%

\newcommand{\lp}{\left(}
\newcommand{\rp}{\right)}
\newcommand{\lb}{\left[}
\newcommand{\rb}{\right]}
\newcommand{\lw}{\left<}
\newcommand{\rw}{\right>}

\newcommand{\one}{\bm{1}}

%%%%%%%%%%%%%%%%%%%%%%%%%%%%%%%%%%%%%%%%%%
%
%		Font Shorthands
%
%%%%%%%%%%%%%%%%%%%%%%%%%%%%%%%%%%%%%%%%%

%% blackboard bold

\newcommand{\C}{\mathbb{C}}
\newcommand{\R}{\mathbb{R}}
\newcommand{\N}{\mathbb{N}}

%% calligraphic letters

\newcommand{\calB}{\mathcal{B}}
\newcommand{\calC}{\mathcal{C}}
\newcommand{\calD}{\mathcal{D}}
\newcommand{\calE}{\mathcal{E}}
\newcommand{\calF}{\mathcal{F}}
\newcommand{\calG}{\mathcal{G}}
\newcommand{\calI}{\mathcal{I}}
\newcommand{\calK}{\mathcal{K}}
\newcommand{\calL}{\mathcal{L}}
\newcommand{\calM}{\mathcal{M}}
\newcommand{\calN}{\mathcal{N}}
\newcommand{\calR}{\mathcal{R}}
\newcommand{\calS}{\mathcal{S}}
\newcommand{\calU}{\mathcal{U}}
\newcommand{\calZ}{\mathcal{Z}}

%% breves

\newcommand{\ba}{\breve{a}}
\newcommand{\bb}{\breve{b}}

\newcommand{\br}{\breve{r}}
\newcommand{\bs}{\breve{s}}

\newcommand{\balpha}{\breve{\alpha}}
\newcommand{\bbeta}{\breve{\beta}}
\newcommand{\brho}{\breve{\rho}}
\newcommand{\bgamma}{\breve{\gamma}}

\newcommand{\bC}{\breve{C}}
\newcommand{\bN}{\breve{N}}

%% boldface

\newcommand{\bfe}{\mathbf{e}}
\newcommand{\bff}{\mathbf{f}}
\newcommand{\bfN}{\mathbf{N}}
\newcommand{\bfM}{\mathbf{M}}

%% bars (complex conjugates) - roman

\newcommand{\qbar}{\overline{q}}
\newcommand{\rbar}{\overline{r}}
\newcommand{\wbar}{\overline{w}}
\newcommand{\zbar}{\overline{z}}

%% bars (complex conjugates) - greek

\newcommand{\etabar}{\overline{\eta}}
\newcommand{\lambar}{\overline{\lambda}}
\newcommand{\lambdabar}{\overline{\lambda}}
\newcommand{\rhobar}{\overline{\rho}}
\newcommand{\btheta}{\overline{\vartheta}}
\newcommand{\zetabar}{\overline{\zeta}}

%% hats

\newcommand{\hatphi}{\widehat{\phi}}

%%%%%%%%%%%%%%%%%%%%%%%%%%%%%
%
%		Order and metric notation	-	
%		Celebrating diversity by using different authors'
%		conventions for the same thing in some cases
%
%		PAP modified definition of littleo to match its
%		use in the one place where it occurs
%
%%%%%%%%%%%%%%%%%%%%%%%%%%%%%

\newcommand{\bigo}[1]{\mathcal{O} \left( #1 \right)}
\newcommand{\littleo}[2][ ]{ {o}_{#1} \left( #2 \right) }

\newcommand{\bigO}[2][ ]{\mathcal{O}_{#1} \left( {#2} \right)}
\newcommand{\norm}[2][ ]{\left\| {#2} \right\|_{#1}}

%%%%%%%%%%%%%%%%%%%%%%%%
%
%		Hack for medium sized union symbol
%
%%%%%%%%%%%%%%%%%%%%%%%%

\newcommand{\medcup}{{\mathsmaller{\bigcup}}}
%\newcommand{\bigsum}{{\, \mathlarger{\sum} \,}}

%%%%%%%%%%%%%%%%%%%%%%%%
%
%		Matrix Shorthand
%
%%%%%%%%%%%%%%%%%%%%%%%%%

\newcommand{\upmat}[1]
{
	\begin{pmatrix}	
	0	&	#1	\\
	0	&	0
	\end{pmatrix}
}
\newcommand{\lowmat}[1]
{
	\begin{pmatrix}
	0	&	0	\\
	#1	&	0
	\end{pmatrix}
}
\newcommand{\upunitmat}[1]
{
	\begin{pmatrix}
	1	&	#1	\\
	0	&	1
	\end{pmatrix}
}
\newcommand{\lowunitmat}[1]
{
	\begin{pmatrix}
	1	&	0	\\
	#1	&	1
	\end{pmatrix}
}

%%%%%%%%%%%%%%%%%%%%%%%%%%%%%
%
%		Shorthand for introduction and sections on
%		long-time behavior
%
%%%%%%%%%%%%%%%%%%%%%%%%%%%%%

% Lax Pair Variables
\newcommand{\La}{ {\mathcal{L} } }
\newcommand{\B}{ {\mathcal{B } } }

% Soliton solutions
\newcommand{\qsol}{q_{\mathrm{sol}}}
\newcommand{\usol}{u_{\mathrm{sol}}}
\newcommand{\tqsol}{\widetilde{q}_{\mathrm{sol}}}

% scattering data quantities
\newcommand{\poles}{\Lambda}
\newcommand{\indicator}{\chi_{_\poles}}
\newcommand{\poledist}{ {\mathrm{d}_{\poles} } }
\newcommand{\coeff}{\mathcal{C}}
\newcommand{\data}{\sigma_d}

\newcommand{\NN}{ {N} }
\newcommand{\nn}{ {n} }
\newcommand{\nk}[2][]{ {{n}_{#1}^{(#2)}} }

\newcommand{\mk}[1]{ m^{(#1)} }
\newcommand{\vk}[1]{ {v^{(#1)}} }
\newcommand{\Sk}[1]{ {\Sigma^{(#1)} } }
\newcommand{\Wk}[2][]{ {W_{#1}^{(#2)} } }

% intervals and sets for jumps
\newcommand{\sgnt}{\eta}
\newcommand{\pospoles}{\Delta_{\xi,\sgnt}^+}
\newcommand{\negpoles}{\Delta_{\xi,\sgnt}^-}
\newcommand{\posnegpoles}{\Delta_{\xi,\sgnt}^\pm}
\newcommand{\posint}{ I^{+}_{\xi,\sgnt}} 
\newcommand{\negint}{ I^{-}_{\xi,\sgnt}} 						%% PAP changed

% parametrices
\newcommand{\Nrhp}[1][]{ {{\mathcal{N}}_{#1}^{\mathsc{rhp}} } }
%% PAP changed for CMP - eliminate \!\!
%\newcommand{\Nsol}[2][]{ {{\mathcal{N}}_{\!\!#1}^{\mathrm{sol},#2} } }
\newcommand{\Nsol}[2][]{ {{\mathcal{N}}_{#1}^{\,\,\mathrm{sol},#2} } }
%% PAP changed for CMP - eliminate \!\!
%\newcommand{\Nout}[1][]{ {{\mathcal{N}}_{\!\!#1}^{\mathrm{out}} } }
\newcommand{\Nout}[1][]{ {{\mathcal{N}}_{#1}^{\,\mathrm{out}} } }
\newcommand{\NsolAlt}[1][]{ {{\mathcal{N}}_{#1}^{\negpoles} } }
\newcommand{\NPC}[1][]{ {{\mathcal{N}}_{#1}^{\mathsc{pc}} } }

\newcommand{\Uxi}{ {\mathcal{U}_\xi} }

\newcommand{\error}{ {{\mathcal{E}}} }

\newcommand{\Ske}[1][]{ { \Sigma^{(\error)}_{#1} }  }
\newcommand{\vke}{ \vk{\error} }

\newcommand{\mout}{M^{(\textrm{out})}}
\newcommand{\mxi}{M^{(\xi)}}
\newcommand{\mPC}{M^{(\textsc{pc})}}

\newcommand{\sol}{\mathrm{sol}}

%%%%%%%%%%%%%%%%%%%%%%
%
% 		Spacing for dbar definitions
%
%%%%%%%%%%%%%%%%%%%%%%
%
\newcommand{\spacing}[2]
{ 
{	
	\mathrlap{#1}
	\hphantom{#2} 
}
}
%
%%%%%%%%%%%%%%%%%%%%%%%%%%%%%%%%%%%%%%%%%%
%
%		Derivative commands
%
%%%%%%%%%%%%%%%%%%%%%%%%%%%%%%%%%%%%%%%%%%

\newcommand{\pd}[3][ ]{\frac{\partial^{#1} #2}{\partial #3^{#1} } }
\newcommand{\od}[3][ ]{\frac{\mathrm{d}^{#1} #2}{\mathrm{d} #3^{#1} } }
\newcommand{\vd}[3][ ]{\frac{ \delta^{#1} #2}{ \delta #3^{#1}} } 

%%%%%%%%%%%%%%%%%%%%%%%%%%%%%%%%%%%%%%%%%%
%
%		Complex variable notation
%
%%%%%%%%%%%%%%%%%%%%%%%%%%%%%%%%%%%%%%%%%%

\renewcommand{\Re}{\mathop{ \mathrm{Re}}\nolimits}
\renewcommand{\Im}{\mathop{ \mathrm{Im} }\nolimits}
\newcommand{\im}{\mathrm{i} }

%%%%%%%%%%%%%%%%%%%%%%%%%%%%%%%%%%%%%%%%%%
%
%		Vectors and Matrices
%
%%%%%%%%%%%%%%%%%%%%%%%%%%%%%%%%%%%%%%%%%

\newcommand{\triu}[2][1]{\begin{pmatrix} #1 & #2 \\ 0 & #1 \end{pmatrix}}
\newcommand{\tril}[2][1]{\begin{pmatrix} #1 & 0 \\ #2 & #1 \end{pmatrix}}
\newcommand{\diag}[2]{\begin{pmatrix} #1 & 0 \\ 0 & #2 \end{pmatrix}}
\newcommand{\offdiag}[2]{\begin{pmatrix} 0 & #1 \\ #2 & 0 \end{pmatrix}}

\newcommand{\striu}[2][1]{\begin{psmallmatrix} #1 & #2 \\ 0 & #1 \end{psmallmatrix}}
\newcommand{\stril}[2][1]{\begin{psmallmatrix} #1 & 0 \\ #2 & #1 \end{psmallmatrix}}
\newcommand{\sdiag}[2]{\begin{psmallmatrix*}[c] #1 & 0 \\ 0 & #2 \end{psmallmatrix*}}
\newcommand{\soffdiag}[2]{\begin{psmallmatrix*}[r] 0 & #1 \\ #2 & 0 \end{psmallmatrix*}}
\newcommand{\stwomat}[4]{\begin{psmallmatrix*} #1 & #2 \\ #3 & #4 \end{psmallmatrix*}}

%%%%%%%%%%%%%%%%%%%%%%%%%%%%%%%%%%%%%%%%%%
%
%		\mathsc font
%
%%%%%%%%%%%%%%%%%%%%%%%%%%%%%%%%%%%%%%%%%

\newcommand{\mathsc}[1]{ {\text{\normalfont\scshape#1}} }

%%%%%%%%%%%%%%%%%%%%%%%%%%%%%
%
%		Old versions of \newcommand's for legacy code!
%
%%%%%%%%%%%%%%%%%%%%%%%%%%%%%

\newcommand{\oldnorm}[2][]
{
	 {\left\|  #2 \right\|_{#1} } 
}

\newcommand{\twomat}[4]
{
\begin{pmatrix}
	#1	&&	#2	\\
	#3	&&	#4
\end{pmatrix}
}

\newcommand{\Twomat}[4]
{
	\left(
		\begin{array}{ccc}
			#1	&&	#2	\\
			\\
			#3	 &&	#4
		\end{array}
	\right)
}

\newcommand{\twovec}[2]
{
	\left(
		\begin{array}{c}
			#1		\\
			#2
		\end{array}
	\right)
}

\newcommand{\Twovec}[2]
{
	\left(
		\begin{array}{c}
			#1		\\
			\\
			#2
		\end{array}
	\right)
}

%%%%%%%%%%%%%%%%%%%%%%%%
%\input{tikz}		   			%% 	Bob's tikz figures

%%%%%%%%%%%%%%%%%%%%%%%%%%%%%%%%%%%%%%%
%
%		TIKZ Figures - Bob
%		File dnls-imrn-arXiv-tikz.tex
%		Created 4.7.2017
%
%%%%%%%%%%%%%%%%%%%%%%%%%%%%%%%%%%%%%%%

%% arrow in middle of line ->- 
\tikzset{->-/.style={decoration={
  markings,
  mark=at position .55 with {\arrow{triangle 45}} },postaction={decorate}}
%  mark=at position .55 with {\arrow{Latex[length=2.5mm]}} },postaction={decorate}}
}

%%%%%%%%%%%%%%%%%%%%%%%%%%%%
%
%	Phase figure for sgn t= +1
%	Colors changed by PAP
%
%%%%%%%%%%%%%%%%%%%%%%%%%%%%
\newcommand{\FigPhaseA}[1][0.7]{
\begin{tikzpicture}[scale={#1}]
%% \Omega_+
\path[fill=gray!20,opacity=0.5]	(0,0) rectangle(4,4);
\path[fill=gray!20,opacity=0.5]    (-4,-4) rectangle(0,0);
%% \Omega_-
\path[fill=gray,opacity=0.5]		(-4,0) rectangle (0,4);
\path[fill=gray,opacity=0.5]		(0,-4) rectangle (4,0);
%% origin
\draw[fill] (0,0) circle[radius=0.075];
%% oriented contour
\draw[thick,->,>=stealth] 	(0,0) -- (2,0);
\draw	[thick]    	(2,0) -- (4,0);
\draw[thick,->,>=stealth]	(-4,0) -- (-2,0);
\draw[thick]		(-2,0) -- (0,0);
\draw[thin,dashed]	(0,4) -- (0,-4);
%% labels
\node at (2,2) 	{$e^{2it\theta} \gg 1 $};
\node at (-2,-2)  {$e^{2it\theta} \gg 1 $};
\node at (-2,2)	{$e^{2it\theta} \ll 1 $};
\node at (2,-2)	{$e^{2it\theta} \ll 1 $};
\node[above] at (0,4)	{$ \sgnt = +1 $};
\node[below] at (2,-0.05)		{$\posint$};
\node[below] at (-2,-0.05)		{$\negint$};
\node[below right] at (0,0)		{$\xi$};
\end{tikzpicture}
}

%%%%%%%%%%%%%%%%%%%%%%%%%%%%
%
%	Phase figure for sgn t = -1
%
%%%%%%%%%%%%%%%%%%%%%%%%%%%%

\newcommand{\FigPhaseB}[1][0.7]{
\begin{tikzpicture}[scale={#1}]
%% \Omega_+
\path[fill=gray,opacity=0.5]	(0,0) rectangle(4,4);
\path[fill=gray,opacity=0.5]    (-4,-4) rectangle(0,0);
%% \Omega_-
\path[fill=gray!20,opacity=0.5]		(-4,0) rectangle (0,4);
\path[fill=gray!20,opacity=0.5]		(0,-4) rectangle (4,0);
%% origin
\draw[fill] (0,0) circle[radius=0.075];
%% oriented contour
\draw[thick,->,>=stealth] 	(0,0) -- (2,0);
\draw	[thick]    	(2,0) -- (4,0);
\draw[thick,->,>=stealth]	(-4,0) -- (-2,0);
\draw[thick]		(-2,0) -- (0,0);
\draw[thin,dashed]	(0,4) -- (0,-4);
%% labels
\node at (2,2) 	{$e^{2it\theta} \ll 1 $};
\node at (-2,-2)  {$e^{2it\theta} \ll 1 $};
\node at (-2,2)	{$e^{2it\theta} \gg 1 $};
\node at (2,-2)	{$e^{2it\theta} \gg 1 $};
\node[above] at (0,4)	{$ \sgnt = -1 $};
\node[below] at (-2,-0.05)		{$\posint$};
\node[below] at (2,-0.05)		{$\negint$};
\node[below right] at (0,0)		{$\xi$};
\end{tikzpicture}
}

%%%%%%%%%%%%%%%%%%%%%%%%%%%%
%
%	Contours and Regions for dbar extension w/ sgn t = +1
%
%%%%%%%%%%%%%%%%%%%%%%%%%%%%

\newcommand{\posDBARcontours}{
\resizebox{0.45\textwidth}{!}{
\begin{tikzpicture}
% DBAR REGIONS
\path [fill=gray!20] (0,0) -- (-4,4) -- (-4.5,4) -- (-4.5,-4) -- (-4,-4) -- (0,0);
\path [fill=gray!20] (0,0) -- (4,4) -- (4.5,4) -- (4.5,-4) -- (4,-4) -- (0,0);
%
% Contours
\draw [help lines] (-4.5,0) -- (4.5,0);
\draw [thick][->-] (-4,4) -- (0,0);
\draw [thick][->-] (0,0) -- (4,4);
\draw [thick][->-] (0,0) -- (4,-4);
%
% holes in the dbar support
\foreach \pos in { 
		(-1.5,1.4), (-1.5,-1.4), (-3.5,2.5), 
		(-3.5,-2.5), (3,2), (3,-2)}
\draw[color=white, fill=white] \pos circle [radius=.2];
%
% Contour Labels
\node[left] at (3,3) {$\Sigma_{1}\,$};
\node[right] at (-3,3) {$\Sigma_{2}$};
\node[right] at (-3,-3) {$\,\Sigma_{3}$};
\node[left] at (3,-3) {$\Sigma_{4}\,$};
\draw[fill] (0,0) circle [radius=0.025];
\node[below] at (0,0) {$\xi$};
\node[above] at (0,4) {$\sgnt = +1$};
%
% Region Labels
\node at (1,.4) {$\Omega_{1}$};
\node at (0,1.08) {$\Omega_{2}$};
\node at (-1,.4) {$\Omega_{3}$};
\node at (-1,-.4) {$\Omega_{4}$};
\node at (0,-1.08) {$\Omega_{5}$};
\node at (1,-.4) {$\Omega_{6}$};
%
%DBAR derivatives
\node[left] at (4.5,0.8) {$\tril{-R_1 e^{-2it \theta}} $ };
\node[right] at (-4.5,0.8) {$\triu{ -R_3 e^{2it \theta}} $ };
\node[right] at (-4.5,-0.8) {$\tril{ R_4 e^{-2it \theta}} $ };
\node[left] at (4.5,-0.8) {$\triu{ R_6 e^{2it \theta }} $};
\node at (0,2.5) {$\diag{\,\, 1\,\,}{\,\,1 \,\,}$};
\node at (0,-2.5) {$\diag{\,\, 1\,\,}{\,\,1 \,\,}$};
\end{tikzpicture}
}
}

%%%%%%%%%%%%%%%%%%%%%%%%%%%%
%
%	Contours and Regions for dbar extension w/ sgn t = -1
%
%%%%%%%%%%%%%%%%%%%%%%%%%%%%

\newcommand{\negDBARcontours}{
\resizebox{0.45\textwidth}{!}{
\begin{tikzpicture}
% DBAR REGIONS
\path [fill=gray!20] (0,0) -- (-4,4) -- (-4.5,4) -- (-4.5,-4) -- (-4,-4) -- (0,0);
\path [fill=gray!20] (0,0) -- (4,4) -- (4.5,4) -- (4.5,-4) -- (4,-4) -- (0,0);
%
% Contours
\draw [help lines] (-4.5,0) -- (4.5,0);
\draw [thick][->-] (-4,4) -- (0,0) ;
\draw [thick][->-](-4,-4)--(0,0);
\draw [thick][->-] (0,0) -- (4,4);
\draw [thick][->-] (0,0) -- (4,-4);
%
% holes in the dbar support
\foreach \pos in { 
		(-1.5,1.4), (-1.5,-1.4), (-3.5,2.5), 
		(-3.5,-2.5), (3,2), (3,-2)}
\draw[color=white, fill=white] \pos circle [radius=.2];
%
% Contour Labels
\node[left] at (3,3) {$\Sigma_{2}\,$};
\node[right] at (-3,3) {$\Sigma_{1}$};
\node[right] at (-3,-3) {$\,\Sigma_{4}$};
\node[left] at (3,-3) {$\Sigma_{3}\,$};
\draw[fill] (0,0) circle [radius=0.025];
\node[below] at (0,0) {$\xi$};
\node[above] at (0,4) {$\sgnt = -1$};
%
% Region Labels
\node at (1,.4) {$\Omega_{3}$};
\node at (0,1.08) {$\Omega_{2}$};
\node at (-1,.4) {$\Omega_{1}$};
\node at (-1,-.4) {$\Omega_{6}$};
\node at (0,-1.08) {$\Omega_{5}$};
\node at (1,-.4) {$\Omega_{4}$};
%
%DBAR derivatives
\node[left] at (4.5,0.8) {$\triu {-R_3 e^{2it \theta}} $ };
\node[right] at (-4.5,0.8) {$\tril{ -R_1 e^{-2it \theta}} $ };
\node[right] at (-4.5,-0.8) {$\triu{ R_6 e^{2it \theta}} $ };
\node[left] at (4.5,-0.8) {$\tril{ R_4 e^{-2it \theta }} $};
\node at (0,2.5) {$\diag{\,\, 1\,\,}{\,\,1 \,\,}$};
\node at (0,-2.5) {$\diag{\,\,1 \,\,}{\,\,1 \,\,}$};
\end{tikzpicture}
}
}

%%%%%%%%%%%%%%%%%%%%%%%%%%%%
%
%	Parabolic Cylinder model jumps w/ sgn t = +1
%
%%%%%%%%%%%%%%%%%%%%%%%%%%%%

\newcommand{\FigPCjumps}{
%%\resizebox{0.5\textwidth}{!}{
%% PAP
\resizebox{0.45\textwidth}{!}{
\begin{tikzpicture}
\draw [help lines] (-4,0) -- (4,0);

\draw[->-] (-3,3) -- (0,0);
\draw[->-] (-3,-3) -- (0,0);
\draw[->-] (0,0) -- (3,3);
\draw[->-] (0,0) -- (3,-3);

\node[left] at (2,2) {$\Sigma_{1}\,$};
\node[right] at (-2,2) {$\Sigma_{2}$};
\node[right] at (-2,-2) {$\,\Sigma_{3}$};
\node[left] at (2,-2) {$\Sigma_{4}\,$};
\draw[fill] (0,0) circle [radius=0.025];
\node[below] at (0,0) {$\xi$};
\node[above] at (0,3) {$\sgnt = 1$};
%
% Region Labels
\node at (1,.4) {$\Omega_{1}$};
\node at (0,1.08) {$\Omega_{2}$};
\node at (-1,.4) {$\Omega_{3}$};
\node at (-1,-.4) {$\Omega_{4}$};
\node at (0,-1.08) {$\Omega_{5}$};
\node at (1,-.4) {$\Omega_{6}$};
%
%DBAR derivatives
\node[right] at (2.0,1.5) {$\tril[1] { s_\xi }$}; 
% \zeta^{-2i\kappa} e^{i\zeta^2/2} } $ };
\node[right] at (2.0,-1.5) {$\triu[1]{ r_\xi } $}; 
%\zeta^{2i\kappa} e^{-i\zeta^2/2} } $ };
\node[left] at (-2.0,1.5) {$\triu[1]{ \dfrac{r_\xi}{1+r_\xi s_\xi} } $}; 
%\zeta^{2i\kappa} e^{-i\zeta^2/2} } $ };
\node[left] at (-2.0,-1.5) {$\tril[1] { \dfrac{s_\xi}{1+r_\xi s_\xi} } $}; 
%\zeta^{-2i\kappa} e^{i\zeta^2/2} } $ };
%
\end{tikzpicture}
}}

%%%%%%%%%%%%%%%%%%%%%%%%%%%%
%
%	Parabolic Cylinder model jumps w/ sgn t = -1  
%
%%%%%%%%%%%%%%%%%%%%%%%%%%%%
\newcommand{\FigPCjumpsB}{
%%\resizebox{0.5\textwidth}{!}{
%% PAP
\resizebox{0.45\textwidth}{!}{
\begin{tikzpicture}
\draw [help lines] (-4,0) -- (4,0);

\draw[->-] (-3,3) -- (0,0)  ;
\draw[->-] (-3,-3) -- (0,0);
\draw[->-] (0,0) -- (3,3);
\draw[->-] (0,0) -- (3,-3);

\node[left] at (2,2) {$\Sigma_{2}\,$};
\node[right] at (-2,2) {$\Sigma_{1}$};
\node[right] at (-2,-2) {$\,\Sigma_{4}$};
\node[left] at (2,-2) {$\Sigma_{3}\,$};
\draw[fill] (0,0) circle [radius=0.025];
\node[below] at (0,0) {$\xi$};
\node[above] at (0,3) {$\sgnt = -1$};
%
% Region Labels
\node at (1,.4) {$\Omega_{3}$};
\node at (0,1.08) {$\Omega_{2}$};
\node at (-1,.4) {$\Omega_{1}$};
\node at (-1,-.4) {$\Omega_{6}$};
\node at (0,-1.08) {$\Omega_{5}$};
\node at (1,-.4) {$\Omega_{4}$};
%
%jump matrices
\node[right] at (2.0,1.5) {$\triu[1]{ \dfrac{r_\xi}{1+r_\xi s_\xi} }$}; 
% \zeta^{2i\kappa} e^{-i\zeta^2/2} } $ };
\node[right] at (2.0,-1.5) {$\tril[1] { \dfrac{s_\xi}{1+r_\xi s_\xi}} $}; 
%\zeta^{-2i\kappa} e^{i\zeta^2/2} } $ };
\node[left] at (-2.0,1.5) {$\tril[1] { s_\xi }$}; 
% \zeta^{-2i\kappa} e^{i\zeta^2/2} } $ };
\node[left] at (-2.0,-1.5) {$\triu[1]{ r_\xi }$}; 
% \zeta^{2i\kappa} e^{-i\zeta^2/2} } $ };
%
\end{tikzpicture}
}}

%%%%%%%%%%%%%%%%%%%%%%%%%%%%
%
%	Spectral strip and light cone for soliton separation
%
%%%%%%%%%%%%%%%%%%%%%%%%%%%%

%%%%%%%%%%%%%%%%%%%%%%%%%%%%
%
%		PAP changed scale from 0.75 to 0.7
%		to make things fit
%
%%%%%%%%%%%%%%%%%%%%%%%%%%%%

\newcommand{\solfig}{
\hspace*{\stretch{1}}
%% altered by PAP
%\begin{tikzpicture}[scale=0.7]
\begin{tikzpicture}[scale=0.5]						
\coordinate (shift) at (-40:1.5);
\coordinate (x1) at (0.7,0);
\coordinate (x2) at (2.1,0);
\coordinate (topleft) at ($(x1)+(95:6)$);
\coordinate (bottomleft) at ($(x1)+(-120:6)$);
\coordinate (topright) at ($(x2)+(60:6)$);
\coordinate (bottomright) at ($(x2)+(-85:6)$);
\coordinate (C) at ($ (x1)!.5!(x2)+(77.5:.5)$);

\begin{scope}
  \clip (-4,-4) rectangle (4,4);
  \path[name path=top] (-4,4) -- (4,4);
  \path[name path=bottom] (-4,-4) -- (4,-4);
  \path[name path=right] (4,4) -- (4,-4);  	
  \path [fill=gray!15] (x1) -- (topleft) -- (topright) -- 
    (x2) -- (bottomright) -- (bottomleft) -- (x1);
  \draw[thick,name path=leftcone] (bottomleft) -- (x1) -- (topleft);
  \draw[thick,name path=rightcone] (bottomright) -- (x2) -- (topright);
    \draw [help lines, name path=axis][->] (-4,0) -- (3.3,0)
    	node[label=0:$x$] {};
    \draw [help lines][->] (-0.5,-4) -- (-0.5,3.0)
    	node[label=90:$t$] {};
    \path [name intersections={of= leftcone and top, by=TL}];
    \path [name intersections={of= rightcone and right, by=TR}];
    \path [name intersections={of= leftcone and bottom, by=BL}];
    \path [name intersections={of= rightcone and bottom, by=BR}];
\end{scope}
    \node [fill=black, inner sep=1.5pt, circle, label=-70:$x_2$] at (x2) {};
    \node [fill=black, inner sep=1.5pt, circle, label=-177:$x_1$] at (x1) {};
	\node [label=180:${x-v_1 t = x_1}$] at (TL) {};
	\node [label=90:${x-v_2 t = x_2}$] at (TR) {};
    \node [label=170:${x-v_2 t = x_1}$] at (BL) {};
    \node [label=10:${x-v_1 t = x_2}$] at (BR) {};
    \node at (C) {$\mathcal{S}$};
\end{tikzpicture}
\hspace*{\stretch{1}}
%% PAP again
\begin{tikzpicture}[scale=0.5]			%% PAP changed 0.7 -> 0.5
	\coordinate (v1) at (0.75,0);
	\coordinate (v2) at (-2,0);
	\path [fill=gray!15] ($(v1)+(0,-1)$) -- ($(v1)+(0,7)$) 
	--($(v2)+(0,7)$) -- ($(v2)+(0,-1)$) --  ($(v1)+(0,-1)$);
	\draw[thick] ($(v1)+(0,-.4)$) -- ($(v1)+(0,7)$);
	\draw[thick] ($(v2)+(0,-.4)$) -- ($(v2)+(0,7)$);
	\draw[help lines][->] (-4,0) -- (4,0) 
		node[label=0:$\Re \lam$] {};
	\node [label=${-v_1/4}$] at (.75,-1.2) {};
	\node [label=${-v_2/4}$] at (-2,-1.2) {};
%%%  Hard-coded nodes (alas) to avoid subscript troubles
\node [fill=black, inner sep = 1pt, circle, label=-90:$\lambda_1$] at (3,6) 			{};
\node [fill=black, inner sep = 1pt, circle, label=-90:$\lambda_2$] at (1.5,5) 		{};
\node [fill=black, inner sep = 1pt, circle, label=-90:$\lambda_3$] at (0,5.5) 		{};
\node [fill=black, inner sep = 1pt, circle, label=-90:$\lambda_5$] at (-3.3,4.4) 	{};
\node [fill=black, inner sep = 1pt, circle, label=-90:$\lambda_8$] at (-0.8,0.8) 	{};
\node [fill=black, inner sep = 1pt, circle, label=-90:$\lambda_4$] at (-2.5,6.5) 	{};
\node [fill=black, inner sep = 1pt, circle, label=-90:$\lambda_6$] at (-1.6,3) 		{};
\node [fill=black, inner sep = 1pt, circle, label=-90:$\lambda_9$] at (2,1) 			{};
\node [fill=black, inner sep = 1pt, circle, label=-90:$\lambda_7$] at (-3.8,1.1) 	{};
\node [fill=black, inner sep = 1pt, circle, label=-90:$\lambda_{10}$] at (3.8,2.3) 	{};
\end{tikzpicture}
\hspace*{\stretch{1}}
}

%\input{intro}			   		%% 	Introduction

%%%%%%%%%%%%%%%%%%%%%%%%%
%
%		intro.tex 
%
%%%%%%%%%%%%%%%%%%%%%%%%

\section{Introduction}

In this paper, we prove soliton resolution for the derivative nonlinear Schr\"{o}dinger equation (DNLS)
\begin{subequations}
\begin{align}
\label{DNLS1}
i u_t + u_{xx} - i \varepsilon (|u|^2 u)_x &=0\\
\label{data1}
u(x,t=0) &= u_0
\end{align}
\end{subequations}
\sidenote{2}
for initial data in a dense and open subset  of the \revisedtext{weighted Sobolev  space} $H^{2,2}(\R)$ which contains $0$ as well as initial data of arbitrarily large $L^2$-norm.  Here 
$\eps = \pm 1$ and 
$H^{2,2}(\R)$ is the completion of $C_0^\infty(\R)$ in the norm
$$ \norm[H^{2,2}]{u} = \left( \norm[L^2]{(1+(\dotarg))^2 u(\dotarg)}^2 + \norm[L^2]{u''}^2 \right)^{1/2}. $$

Our work builds on three previous papers \cite{LPS15,LPS16,JLPS17b} where we respectively considered global well-posedness in the soliton-free sector, large-time asymptotics in the soliton-free sector, and global well-posedness on the dense and open subset described in what follows. We will refer to these as Papers I, II, and III 
\sidenote{55}{for the remainder of the introduction.}
%in what follows.  
A more detailed presentation of the material in Paper III and the current paper may be found in the preprint \cite{JLPS17b}.  

Soliton resolution refers to the property that the solution decomposes into the sum of a finite number of separated
\sidenote{3}  solitons and a radiative part as $|t| \to \infty$.  The \revisedtext{limiting soliton parameters} are slightly modulated, due to the soliton-soliton and soliton-radiation interactions. We fully describe the dispersive part which contains two components,  one coming from the continuous  spectrum and another one from the interaction of the discrete and continuous spectrum.

This decomposition is a central feature
in nonlinear wave dynamics and has been the object of many theoretical and numerical studies. It has been established in many {\em{perturbative}} contexts,   that is when the initial condition is
close to a soliton or a multi-soliton. In non-perturbative cases,  this property was proved rigorously for KdV  \cite{ES83},  mKdV \cite{DZ93,S86}  and  for the focussing NLS equation \cite{BJM16} using the inverse scattering approach.
The last result has been conjectured for a long time \cite{ZS72} but rigorously proved only recently.
A  direct consequence  of this result is that   N-soliton solutions are asymptotically stable.

The soliton resolution conjecture is at the heart of current studies in nonlinear waves and extends to solutions that blow up in finite time. 
 In the context of  non-integrable equations, Tao \cite{Tao08} considered the NLS equation with potential in high dimension ($d\ge 11$) and proved the existence of a global  attractor, assuming radial symmetry.
A recent work by Duykaerts, Jia, Kenig and Merle \cite{DJKM16}  concerns  the focusing energy critical wave equation for which they prove that any bounded solution  asymptotically behaves like  a finite sum of modulated solitons, a regular component in the finite time blow-up case or a free radiation in the global case, plus a residue term that vanishes asymptotically in the energy space as time approaches the maximal time of existence  (see  also \cite{DKM13,DKM15} for other cases,   radial and non-radial, in various dimensions).

\subsection{The Inverse Scattering Method}

In this subsection we briefly review the global well-posedness 
result  of \cite{JLPS17a} and describe the spectrally defined, dense open subset of initial data $H^{2,2}(\R)$ for which we will prove soliton resolution. 
Equation \eqref{DNLS1} is gauge-equivalent to the equation
\begin{subequations}
\begin{align}
\label{DNLS2}
iq_t + q_{xx} + i \eps q^2 \bar q_x + \frac{1}{2} |q|^4 q = 0, \\
\label{data}
q(x,t=0) = q_0(x)
\end{align}
\end{subequations}
via the gauge transformation\footnote{  In papers I and II, we use {{a}} slightly different gauge transformation, namely $\exp\left(- i\eps \int_{-\infty}^x |u(y)|^2 \, dy \right) u(x)$. Both
transformations are clearly equivalent, up to the constant phase factor $\exp\left( i\eps \int_{-\infty}^{\infty} |u_0(y)|^2 \, dy \right)$. The current transformation has the advantage
of slightly  simplifying some formulae in the {analysis of long time behavior.}} 
\begin{equation}
\label{G}
q(x)  :=\calG (u)(x) = \exp\left( i\eps \int_x^{\infty} |u(y)|^2 \, dy \right) u(x).
\end{equation}
This nonlinear, invertible mapping is an isometry of $L^2(\R)$, maps soliton solutions to soliton solutions, and maps dense open sets to dense open sets in weighted Sobolev spaces.
Consequently,  global well-posedness for \eqref{DNLS2} on an open and dense set $U$ in $H^{2,2}(\R)$ containing data of arbitrary $L^2$-norm implies 
global well-posedness of \eqref{DNLS1} on a subset $\calG^{-1}(U)$ of $H^{2,2}(\R)$ with the same properties. In \cite{JLPS17b},  global well-posedness is established for \eqref{DNLS2}.

Our analysis exploits the discovery of Kaup and Newell \cite{KN78} that \eqref{DNLS2} has the Lax representation
\begin{equation*}
%\label{Lax}
\begin{aligned}
L	&=	-i\lam \sigma_3 + Q_\lam - \frac{i}{2} \sigma_3 Q^2\\
A	&=	2 \lam L + i \lam (Q_\lam)_x + \frac{1}{2}\left[ Q_x , Q \right] + \frac{i}{4}\sigma_3 Q^4
\end{aligned}
\end{equation*}
where
\begin{equation*}
%\label{Q.sig} 
\sigma_3 = \begin{pmatrix} 1 & 0 \\ 0 & -1 \end{pmatrix}, \quad
Q(x) =  
\begin{pmatrix} 0 & q(x) \\  \eps  \overline{q(x)} & 0 \end{pmatrix},
\quad
Q_\lam(x) = 
\begin{pmatrix} 0 & q(x) \\  \eps \lam \overline{q(x)} & 0 \end{pmatrix}.
\end{equation*}
That is, a
smooth function $q(x,t)$ solves \eqref{DNLS2} if and only if the operator identity
$$ L_t - A_x + [L,A] = 0 $$ 
\sidenote{4}holds for the \revisedtext{ matrices} above with $q=q(x,t)$ (so that both 
%operators 
\revisedtext{matrices}
depend on $t$).

To exploit the Lax representation, we consider the 
%linear 
spectral problem
\begin{equation}
\label{LS}
\Psi_x = L\Psi
\end{equation} 
for  $\lam \in \C$ and $2 \times 2$ matrix-valued solutions $\Psi(x;\lam)$.

As shown in \cite{JLPS17b},  \eqref{LS} defines a map $\calR$ from $q \in H^{2,2}(\R)$ to spectral data, and has an inverse $\calI$ defined by a Riemann-Hilbert problem (\revisedtext{Riemann-Hilbert Problem} \ref{RHP2} below) which recovers the potential $q(x)$. Moreover, the spectral data for a solution 
 $q = q(x,t)$ of \eqref{DNLS2} obey a linear law of evolution. Thus the solution %operator 
$\calM$ for the Cauchy problem \eqref{DNLS2}-\eqref{data} is given by
\sidenote{8}
\begin{equation}
\label{sol.op}
\revisedtext{\calM (t) q_0} = \left( \calI \circ\Phi_t \circ \calR \right) q_0 
\end{equation}
where $\Phi_t$  represents  the linear evolution on spectral data. To state the results of Paper III that we will use, we  describe the set $U$ and the maps $\calR$,  $\Phi_t$, and $\calI$ in greater detail. For any $q \in H^{2,2}(\R)$ and $\lam \in \R$; there exist unique 
\emph{Jost solutions} $\Psi^\pm(x,\lam)$ of $\Psi_x = L\Psi$ with
respective asymptotics
$\lim_{x \rarr \pm \infty} \Psi^\pm(x;\lam) e^{i\lam x \sigma_3} =I$, where $I$ denotes the $2\times 2$ identity matrix. The scattering data
are defined via the relation
\begin{equation}
\label{Jost.T}
\Psi^+(x;\lam) = \Psi^-(x;\lam) T(\lam), \quad
T(\lam) \coloneqq
\begin{pmatrix} 
\alpha(\lam) 	& \beta(\lam) \\
\bbeta(\lam) 	&	\balpha(\lam)
\end{pmatrix}
\end{equation}
where, since $\det \Psi^\pm = 1$,
\begin{equation*}
%\label{T.det}
\alpha \balpha - \beta \bbeta = 1.
\end{equation*}
Symmetries of \eqref{LS} imply that, also
\begin{equation}
\label{alpha.beta.sym}
\balpha(\lam) = \overline{\alpha(\lam)}, \quad \bbeta(\lam) = \eps \lam \overline{\beta(\lam)}.
\end{equation}
We introduce the reflection coefficient
\begin{equation} \label{scatt-rho}
  \rho(\lam) = \beta(\lam)/\alpha(\lam) 
 \end{equation}
and note that 
$ 1-\eps \lam |\rho(\lam)|^2 = \dfrac{1}{|\alpha(\lam)|^2} > 0. $
We then define  the set
\begin{equation}
\label{S}
P = \left\{ \rho\in H^{2,2}(\R): 1- \eps \lam |\rho(\lam)|^2 > 0 \right\} .
\end{equation}

\sidenote{11}\revisedtext{The function $\balpha$ may be computed via the
`Wronskian formula'  
$$ \balpha(\lam) = 
		\begin{vmatrix}
			\psi_{11}^-(x,\lam)	&	\psi_{12}^+(x,\lam)	\\
			\psi_{21}^-(x,\lam)	&	\psi_{22}^+(x,\lam)
		\end{vmatrix};
$$
a standard analysis of \eqref{LS} shows that if $q \in L^1(\R) \cap L^2(\R)$,
the vector-valued functions $\psi^-_{(1)}=(\psi_{11}^-,\psi_{21}^-)$ and $\psi^+_{(2)}=(\psi_{12}^+,\psi_{22}^+)$ have analytic continuations to $\imag \lam> 0$ for each $x$,
and decay exponentially fast respectively as $x \rarr \pm \infty$.
Thus $\balpha$ has an analytic continuation to $\C^+$. Zeros $\lam_j$ of $\balpha$ \sidenote{11} are
eigenvalues of the spectral problem  \eqref{LS} and signal the presence of square-integrable solutions. 
\sidenote{5}
We have $\psi_{(1)}^{-}(x,\lam_j)= B_j \psi_{(2)}^{+}(x,\lam_j)$ for nonzero $B_j$ and, if the zero $\lam_j$ is simple, we define the associated \emph{norming constant} $C_j$ as $B_j/\balpha'(\lam_j)$. 
If $q_0 \in H^{2,2}(\R)$,  zeros of  $\balpha$ in $\C^+$  correspond to  \emph{bright solitons}, as our results on large-time asymptotics show. 
}

Zeros of $\balpha$ on $\R$ 
are known to occur and %for algebraic solitons \cite{KN78}.
 they correspond to spectral singularities \revisedtext{ \cite{Zhou89-b}}. They are excluded from our analysis in the present paper. Their presence  will be addressed in a forthcoming article.

\begin{definition}
\label{def:U}
We denote by $U$ the subset of $H^{2,2}(\R)$ consisting of  functions $q$ for which $\balpha$ has no zeros on $\R$ and at most finitely many simple zeros in $\C^+$.
\end{definition}
Clearly $U = \bigcup_{N=0}^\infty \, U_N$ where $U_N$ consists of those $q$ for which $\balpha$ has exactly $N$ zeros in $\C^+$. If $N \neq 0$, we denote by $\lam_1, \ldots, \lam_N$ the simple zeros of $\balpha$ in $\C^+$.

Associated to each $\lam_j$ is a \emph{norming constant} $C_j \in \C \setminus \{0\} := \C^\times$ which can be described in terms of the Jost solutions. Thus, the spectral problem \eqref{LS} associates to each $q \in U_N$ the spectral data
$$ 
\calD(q) = 
\left( 
	\rho, \{ (\lam_j, C_j )\}_{j=1}^N 
\right) \in 
P \times \left(\C^+ \times \C^\times\right)^N.
$$
The map $q \mapsto \calD(q)$ is called the \emph{direct scattering map}. To describe its range, let 
\begin{equation}
\label{V}
V=\bigcup_{N=0}^\infty V_N, \quad V_N = P \times \left( \C^+ \times \C^\times \right)^N.
\end{equation} 
For an element $\calD$ of $V$, write 
\begin{equation*}
%\label{Lambda}
\Lambda = \Lambda^+ \cup \overline{\Lambda^+}, \quad \Lambda^+ = \left\{ \lam_1, \ldots , \lam_N \right\}
\end{equation*}
and define
\begin{equation}
\label{distances}
d_\Lambda = \inf_{\lam, \mu \in \Lam, \, \lam \neq \mu} |\lam - \mu|. 
\end{equation}
Note that for any $\lambda_j \in \Lambda, |\imag \lambda_j|  \geq (1/2) d_\Lambda$.
We say that a subset $V'$ of $V_N$ is \emph{bounded} if:
\begin{itemize}
\item[(i)]		$1 -\varepsilon \lam |\rho(\lam)|^2 \geq c_1>0$ for a fixed strictly positive constant $c_1$,
\item[(ii)]	$\sup_j \left(|C_j| + |\lam_j|\right) \leq C$ for a fixed strictly positive constant $C$,
\item[(iii)]   $d_\Lambda \geq c_2 > 0$ for a fixed strictly positive constant $c_2$.
\end{itemize} 
We say that a subset $U'$ of $U$ is bounded if the set $\{\calD(q): q \in U\}$ is a bounded subset of $V_N'$ and if, also, $U'$ is a bounded subset of $H^{2,2}(\R)$.  

If $q=q(x,t)$ evolves according to \eqref{DNLS2}, the corresponding scattering data evolves linearly:
\begin{equation}
\label{sd.ev}
\dot{\rho}(\lam,t) = -4i \lam^2 \rho(\lam,t), \quad	\dot{\lam_j} = 0, \quad \dot{C_j} = -4i \lam_j^2 C_j.
\end{equation}
We denote by $\Phi_t$ the linear evolution on scattering data
$\left\{ \rho, \{ \lam_j, C_j \}_{j=1}^N \right\}$ induced by \eqref{sd.ev}.

Let us introduce the phase function 
\begin{equation}
\label{phase.lambda} 
\theta(x,t,\lam) = -\left(\frac{x}{t} \lam + 2 \lam^2\right).
\end{equation}
%% added
\sidenote{7}\revisedtext{Here and in what follows we write
$$ \ad(\sigma_3) A = \left[ \sigma_3, A \right] $$
for a $2 \times 2$ matrix $A$, so that
$$ e^{i\theta \ad\sigma_3} A = e^{i\theta \sigma_3} A e^{-i\theta \sigma_3}. $$
}
%% end addition
The (time-dependent)  inverse spectral problem is defined by a Riemann-Hilbert problem (RHP) and a reconstruction formula.

\begin{RHP}
\label{RHP2}
Given $x,t \in \R$, $\rho \in P$ and {$\{ (\lam_j, C_j) \}_{j=1}^N \subset \C^{+} \times \C^\times$}, find a matrix-valued function 
$$N(\lam;x,t) : \C \setminus (\R \cup \Lam %\cup \overline{\Lam}
) \rarr SL(2,\C)$$ 
with the following properties:
\medskip
\begin{itemize}
\item[(i)]		$N_{22}(\lam;x,t) = \overline{N_{11}(\lambar;x,t)}$, 
$N_{21}(\lam;x,t) = \eps \lam \overline{N_{12}(\lambar;x,t)}$, 
\smallskip
\item[(ii)]	\sidenote{6}\revisedtext{There exists $p^* \in \C$ so that} $N(\lam;x,t) = \stril{p^*} + \bigO{\dfrac{1}{\lam}}$ as $|\lam| \rarr \infty$, 
\smallskip
\item[(iii)]	$N$ has continuous boundary values $N_\pm$ on $\R$ and
				$$ N_+(\lam;x,t) = N_-(\lam;x,t) e^{it\theta \ad \sigma_3} v(\lam),
				\quad 
				v(\lam) = \begin{pmatrix}
							1-{\eps} \lam |\rho(\lam)|^2 	&	\rho(\lam)	\\
								-\eps \lam \overline{\rho(\lam)}	&	1
							\end{pmatrix}
				$$
\smallskip
\item[(iv)] 	$N(\lam;x,t)$ has simple poles at each point $p \in \Lam$:  
				$$
				\Res_{\lam = p} N(\lam;x,t) 
				= \lim_{\lam \rarr p} N(\lam;x,t) e^{it\theta \ad \sigma_3} v(p) 
				$$ 
				where for each $\lam_j \in \Lam^+$
				$$ v(\lam_j)  = \twomat{0}{0} {\lam_j C_j}{0}, \quad 
				v(\overline{\lam_j}) = \twomat{0}{ \eps  \overline{C_j} }{0}{0}. 
				$$
\end{itemize}
\end{RHP}
Given the solution $N(\lam;x,t)$ of \revisedtext{Riemann-Hilbert} Problem \ref{RHP2}, we can recover the solution
$q(x,t)$ from the reconstruction formula
\begin{equation}
\label{q.lam}
	q(x,t)= \lim_{\lam \rarr \infty}  2i \lam  N_{12}(\lam;x,t).
\end{equation}
We denote by $\calI$ the map from scattering data $\left\{ \rho, \{ \lam_j,C_j\}\right\}$ to $q(x,0)$, so that formally the solution operator for \eqref{DNLS2} is given by 
$$ \calM(t) = \calI \circ \Phi_t \circ \calR. $$

In Paper III, we proved:

\begin{theorem}
%\label{thm:RIP}
For $\eps=\pm 1$, the set $U$ is open and dense in $H^{2,2}(\R)$. Moreover, the direct scattering map 
\begin{equation*}
 \calR: U \rarr V 
\end{equation*}
maps bounded subsets of $U_N$ into bounded subsets of $V_N$ for each $N$,  and is
uniformly Lipschitz continuous on bounded subsets of $U_N$. The inverse scattering map $\calI$ has analogous mapping and Lipschitz properties. Finally, the evolution $\calM$ defines a solution operator for \eqref{DNLS2}
in the sense that 
\begin{itemize}
\sidenote{9}
\item[(i)]		$\revisedtext{q_0} \mapsto \calM(t) q_0$ is locally Lipschitz continuous on $U_N$,
\item[(ii)]	 $\revisedtext{t} \mapsto \revisedtext{\calM}(t)q_0$ defines a continuous curve in $H^{2,2}(\R)$, and
\item[(iii)]	For each $q_0 \in \calS(\R)$, $q(x,t) = \calM(t) q_0$ is a classical solution of \eqref{DNLS2}.
\end{itemize}
\end{theorem}

\sidenote{55}
As discussed in Paper III, \revisedtext{RHP} \ref{RHP2} is equivalent to a technically somewhat simpler RHP involving row vector-valued functions.

\begin{RHP}
\label{RHP2.row}
Given $x,t \in \R$, $\rho\in P$, and $\{ (\lam_j, C_j)\}_{j=1}^N$ in $(\C^{+} \times \C^\times)^N$, find a row vector-valued function $n(x,t,z): \C \setminus (\R \cup\Lam)  \rarr \C^2$ with the following properties:
\begin{itemize}
\item[(i)]		$n(\lam;x,t) = \begin{pmatrix} 1 & 0 \end{pmatrix} + \bigO{\dfrac{1}{\lam}}$ as $|z| \rarr \infty$,
\item[(ii)]  $n$ has continuous boundary values $n_\pm$ on $\R$ and
				$$ n_+(\lam;x,t) = n_-(\lam;x,t) e^{it\theta \ad (\sigma_3)} v(\lam),
				\quad 
				v(\lam) = \begin{pmatrix}
								1-\eps\lam |\rho(\lam)|^2 	&	 \rho(\lam)	\\
								-\eps \lam \overline{\rho(\lam)}	&	1
							\end{pmatrix},
				$$
\item[(iii)]	For each $\lam \in \Lam$, 
				$$ \Res_{z = \lam} n(\lam;x,t) = \lim_{z \rarr \lam} n(z;x,t) e^{it\theta \ad (\sigma_3)} V(\lam) $$
				where for each $\lam \in \Lam^+$
				$$ v(\lam)  = \lowmat{\lam C_\lam}, \quad v( \overline{\lam}) = \upmat{\eps \overline{C_\lam}}. $$
\end{itemize} 
\end{RHP}

\subsection{Soliton Resolution}

To deduce large-time asymptotics and soliton resolution for \eqref{DNLS1}, we will first obtain large-time asymptotics for \eqref{DNLS2} and then compute the asymptotics of the phase in \eqref{G} in terms of spectral data. These two results together will yield large-time asymptotics and  a modified form of soliton resolution for \eqref{DNLS1}.

Long-time asymptotic behavior and solution resolution for \eqref{DNLS2}  are established for  initial data $q_0 \in U$ using the Deift-Zhou method of steepest descent applied to \revisedtext{RHP} \ref{RHP2}. Note that the phase function $\theta$ in \eqref{phase.lambda}   has a single critical point at $\xi = -x/4t$.
From the evolution \eqref{sd.ev}  and the symmetry properties of 
\revisedtext{RHP} \ref{RHP2}, %and Remark \ref{rem:RHP2.unique},
it suffices to solve \sidenote{10}\revisedtext{Riemann-Hilbert Problem \ref{RHP2.row}}
%the Riemann-Hilbert problem 
for the row vector-valued function $n(\lam;x,t) = (N_{11}(\lam;x,t),N_{12}(\lam;x,t))$. 

We choose intervals $[v_1,v_2]$ of velocities and $[x_1,x_2]$ of initial positions and compute the asymptotic behavior of $q(x,t)$ in space-time regions of the form
\begin{equation}
\label{S.cone}
\calS(v_1,v_2,x_1,x_2)	=	
\left\{ (x,t) : x=x_0 + vt \text{ for } v \in [v_1,v_2], \, \, x_0 \in [x_1,x_2] \right\}.
\end{equation}
The explicit formula for one-soliton solutions given in \eqref{1sol}   shows that a soliton associated to eigenvalue $\lambda$ moves with velocity
($-4 {\rm Re}  \lambda$).

In what follows, we set
\sidenote{12}
\begin{equation}
\label{LamI}
\Lambda(I) = \{ \lam \in \Lambda: \real (\lam) \in I \} \revisedtext{= \Lambda^+(I) \cup \overline{\Lambda^+(I)}}
\end{equation}
and \sidenote{12}
\revisedtext{
\begin{equation}
\label{NI}
N(I) = |\Lambda^+(I)|. 
\end{equation}
}
Solitons in $\Lambda([-v_2/4, -v_1/4])$ should be `visible' {within $\calS(v_1,v_2,x_1,x_2)$}, but remaining solitons will 
move either too slowly or too fast to be seen in the moving window.

To state our result, we need the following notation. For $\xi \in \R$ and $\eta \in \{ -1 , +1 \}$ where  $\eta = \sgn t$, let
\begin{equation}
\begin{aligned}
\label{posint}
I_{\xi,\eta}^-	&=	\left\{ \lam \in \C:		
  \imag \lam = 0, \quad -\infty < \eta \Re \lam \leq \eta \xi \right\}, \\
I_{\xi,\eta}^+	&=	\left\{ \lam \in \C:		
  \imag \lam = 0, \quad \eta \xi < \eta \Re \lam < \infty \right\}.
\end{aligned}  
\end{equation}
%For $x_1 \leq x_2$ and $v_1 \leq v_2$, let $\calS(v_1,v_2,x_1,x_2)$ be the subset of $\R^2$ given by

\begin{figure}[H]
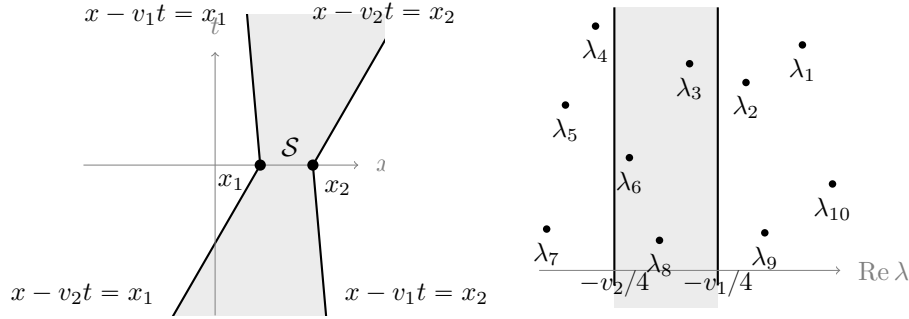

\solfig
\caption{
Given initial data $q_0(x)$ which generates scattering data 
$\left\{ \rho, \{ \lam_k, C_k \}_{k=1}^N \right\}$, then, asymptotically as $|t| \to \infty$ inside the space-time cone $\mathcal{S}(v_1,v_2,x_1,x_2)$ (shaded on left) the solution $q(x,t)$ of \eqref{DNLS2} approaches an $N(I)$-soliton $\qsol(x,t)$ corresponding to the discrete spectra in $\poles(I)$ 
(shaded region on right).  The connection coefficients $\widehat{C}_k$ for $\qsol$ are modulated by the soliton-soliton and soliton-radiation interactions as described in Theorem~\ref{thm:long-time}.
%\label{fig:light.cone}
}
\end{figure}
We will prove:

\begin{theorem}
\label{thm:long-time}
Suppose that $q_0 \in U_N$ with scattering data $\left( \rho, \{ (\lam_k,C_k) \}_{k=1}^N \right)$. Fix $x_1, x_2, v_1, v_2 \in \R$ with $v_1 < v_2$, let $I = [-v_2/4,-v_1/4]$
and $ \xi = - x/(4t)$.
Denote by  $\qsol(x,t;\mathcal{D}_I)$ the soliton solution of \eqref{DNLS2} with 
{modulating reflectionless scattering data} \sidenote{12}
$$ \calD_I = \left\{ \rho_I \equiv 0, \{(\lam_k, \widehat{C_k})\}_{\lam_k \in \revisedtext{\Lambda(I^+)}} \right\}$$
where
$$
\widehat{C_k} =  C_k \  
	\prod_{\mathclap{\Re \lam_j \in I_{\xi,\eta}^- \setminus I}}
	\quad \ \left(	\frac{\lam_k - \lam_j}{\lam_k - \overline{\lam_j}} \right)^2 
	\exp \left(  
					\frac{i}{\pi} \int_{I_{\xi,\eta}^-} 
						\frac{\log\left( 1 - \eps \lam |\rho(\lam)|^2 \right)}{\lam-\lam_k}  \, d\lam
			\right).
$$
Then, as $|t| \rarr \infty$ in the cone $\calS(v_1,v_2,x_1,x_2)$, we have
\sidenote{13}
\begin{equation}
\label{q.long-time}
q(x,t) \revisedtext{\, = \, }
\qsol(x,t;\mathcal{D}_I) + t^{-1/2} f(x,t) + \bigO{t^{-3/4}}. 
\end{equation}
%%%%%%%%%%%%%%%%%%%%%%%%%%%%%%%%%%%%%%%%%%%%%
%%
%% 		PAP changed here to avoid line overflow - see changes below 
%%		and after this equation
%%
%%%%%%%%%%%%%%%%%%%%%%%%%%%%%%%%%%%%%%%%%%%%%
Here $f(x,t)$ is given by
\begin{equation}\label{phase.phi}
\begin{aligned}
f(x,t) &=  2^{-1/2} 
	\left[ 
	  A_{12}(\xi,\eta) % \Nsol[11]{\emptyset}
%	  {\mathcal N}^{\mathrm sol}_{11}(\xi;x,t\, \vert \, \calD_I)^2 e^{i\phi(\xi)} 
      \left({\mathcal N}^{\mathrm sol}_{11}\right)^2 e^{i\phi(\xi)} 
	  + \eps \xi \overline{A_{12}(\xi,\eta)}  
%	{\mathcal N}^{\mathrm sol}_{12}(\xi;x,t\,\vert\, \calD_I)^2 e^{-i\phi(\xi)} 
	  \left({\mathcal N}^{\mathrm sol}_{12}\right)^2 e^{-i\phi(\xi)} 
	\right]
\\
\phi(\xi) &:= {-4} \sum_{\mathclap{\Re \lam_k \in \negint \setminus I}} \arg(\lam_k - \xi)
\end{aligned}
\end{equation}
%where $A_{12}(\xi,\eta)$ is given by \eqref{Axi} and ${\mathcal N}^{\mathrm sol}$ solves RHP~\ref{RHP2} with scattering data $\calD_I$.
where $A_{12}(\xi,\eta)$ is given by \eqref{Axi} and ${\mathcal N}^{\mathrm sol} \coloneqq {\mathcal N}^{\mathrm sol}(\xi;x,t \vert \, \calD_I)$ solves RHP~\ref{RHP2} with scattering data $\calD_I$.
\end{theorem}

Equation \eqref{q.long-time} expresses soliton resolution in the following sense. First, as described below, the function $q_\sol(x,t)$ is generically asymptotic to a superposition of one-soliton solutions.
Second, the term at order $t^{-1/2}$ represents a dispersive contribution; in the no-soliton case, i.e., if $v_1,v_2$ are chosen in Theorem~\ref{thm:long-time} such that $N(I) = 0$, 
$\mathcal{N}^{\mathrm{sol}} \equiv I$, and $\qsol \equiv 0$,  using \eqref{Axi}, the asymptotic behavior of the  solution reduces to
%%%%%%%%%%%%%%%%%%%%%%%%%%%%%
%%
%%		PAP changed second line of this formula to 
%%		avoid line overflow
%%
%%%%%%%%%%%%%%%%%%%%%%%%%%%%%
\begin{align*}
q(x,t) 		&= \frac{1}{\sqrt{2|t|}} \frac{\kappa(\xi)}{ \xi} e^{i\alpha_\pm(\xi) } 
	   					e^{i x^2/(4t) \mp \kappa(\xi) \log|8t| }
	   					+ \bigo{t^{-3/4}}, \quad t \to \pm \infty \\
\intertext{where}
\kappa(s)&= -\frac{1}{2\pi} \ln (1 -\varepsilon s|\rho(s)|^2),  \\
\alpha_\pm(\xi) &= \frac{\pi}{4} - \arg(-\eps \xi \overline{\rho(\xi)})
	  					 \pm \arg \Gamma(i\kappa(\xi)) 
%	  					 \pm \frac{1}{\pi} \int_{\mp \infty}^\xi \log|\xi - \lam| 
						\mp 2  \int_{\mp \infty}^\xi \log|\xi - \lam| 
%	     					 \mathrm{d}_\lam \log(1-\eps \lam |\rho(\lam)|^2),
							\, d\kappa(\lam)
\end{align*}
which agrees with the dispersive asymptotics obtained,  for example, in \cite{LPS16}. 

Kitaev and Vartanian computed similar asymptotics in the 
no-soliton sector \cite{KV97} and the finite-soliton sector \cite{KV99}, making more stringent assumptions on regularity together with a smallness assumption on the reflection coefficient that we do not require.

Theorem~\ref{thm:long-time} implies, as a special case, the asymptotic separation of the solution $q(x,t)$ into a sum of one-soliton solutions whenever the $\lam_k \in \poles^+$ have distinct real parts. 
 If the initial data $q_0$ generates scattering data $\{ \rho, \{(\lam_k, C_k)\}_{k=1}^N \}$ with $\lam_k = \eta_k + i \tau_k$ then
applying Theorem~\ref{thm:long-time} repeatedly to sets $\mathcal{S}_k$ each of which contains a single soliton speed $v_k=-4\eta_k$ one finds that the solution of \eqref{DNLS2}-\eqref{data} satisfies 
\begin{equation}\label{sol.res.generic}
	q(x,t) = \sum_{k=1}^N 
		\mathcal{Q}_{\textrm{sol}}(x,t; \lam_k, x_k^\pm, \varphi_k^\pm) + \bigo{|t|^{-1/2}}
		\qquad
		t \to \pm \infty	
\end{equation}		
where, setting $u = \real \lambda$,\sidenote{14}
\begin{multline*}
	\mathcal{Q}_{\textrm{sol}}(x,t;\lam,x_0,\varphi_0) =  
	\mathcal{Q}(x - x_0 +4 u t , \lam)  \\
	\times
		\exp i \left\{ 4|\lam|^2 t - 2u(x + 4u t) 
		  - \frac{\eps}{4} \int_{-\infty}^{x- x_0 + 4 u t} \mathcal{Q}(\eta\revisedtext{,\lam})^2 d\eta
		  - \varphi_0 \right\}
\end{multline*}
is the form of a general one-soliton solution of \eqref{DNLS2}. 
Here $\varphi_0$, $x_0$, and $\mathcal{Q}$ are given as in \eqref{1sol}.
The asymptotic phases
are 
%%%%%%%%%%%%%%%%%%%%%%%%%%%%%
%%
%%		PAP substituted \kappa(s) for the log to save
%%		space in the following two equations
%%
%%%%%%%%%%%%%%%%%%%%%%%%%%%%%
\begin{gather}
\label{x phases}
	x^\pm_k = 
		\frac{1}{4\tau_k} \log \left| \frac{\lam_k C_k^2}{4 \tau_k^2} \right|
		+ \frac{1}{2\tau_k} 
		\sum_{\substack{\lam_j \in \poles^+ \\ \mathclap{\pm (\eta_k - \eta_j) > 0}}} 
		  \log \left| \frac{ \lam_k - \lam_j}{\lam_k - \overline{ \lam_j} } \right|
%		  \pm \frac{1}{2\pi} \int_{\eta_k}^{\mp\infty}
			\mp \int_{\eta_k}^{\mp\infty}
%		    \frac{ \log(1-\eps s | \rho(s)|^2) }{(s-\eta_k)^2+\tau_k^2} ds \\
			\frac{\kappa(s)}{(s-\eta_k)^2+\tau_k^2} ds \\
\label{alpha phases}
	\varphi^\pm_k = 
		\arg \lp i \lam_k C_k \rp 
		+ \sum_{\substack{\lam_j \in \poles^+ \\ \mathclap{\pm (\eta_k - \eta_j) > 0}}}
		    \arg \lp \frac{ \lam_k - \lam_j}{\lam_k - \overline{ \lam_j} } \rp
%		\mp \frac{1}{\pi} \int_{\eta_k}^{\mp \infty} 
		  \pm 2 \int_{\eta_k}^{\mp \infty} 
%		    \frac{ (s - \eta_k) \log(1-\eps s | \rho(s)|^2) }{(s-\eta_k)^2+\tau_k^2} ds
		\frac{ (s - \eta_k) \kappa(s) }{(s-\eta_k)^2+\tau_k^2} ds
		 \mod{2\pi}
\end{gather}
so that the total phase shifts of the $k$th soliton, as it interacts both with the other solitons and the radiation component, are
%%%%%%%%%%%%%%%%%%%%%%%%%%%%%
%%
%%		PAP changed log to kappa(s) to save space
%%		in the following two equations
%%
%%%%%%%%%%%%%%%%%%%%%%%%%%%%%
\sidenote{15}
\begin{align*}
	x_k^+ - x_k^- &= 
		\frac{1}{2\tau_k} 
		\sum_{\mathclap{j \neq k }} \sgn(\eta_k - \eta_j) 
		  \log \left| \frac{ \lam_k - \lam_j}{\lam_k - \overline{ \lam_j} } \right|
%		  + \frac{1}{2\pi} \int_{-\infty}^{\infty}
		 -  \int_{-\infty}^{\infty}
%		    \frac{ \sgn(s-\eta_k) \log(1-\eps s | \rho(s)|^2) }{(s-\eta_k)^2+\tau_k^2} ds \\
			\frac{ \sgn(s-\eta_k) \revisedtext{\kappa(s) }}{(s-\eta_k)^2+\tau_k^2} ds, \\
	\varphi_k^+ - \varphi_k^- &= 
	\lb
		\sum_{\mathclap{j \neq k}}
		    \sgn(\eta_k - \eta_j) 
		    \arg \lp \frac{ \lam_k - \lam_j}{\lam_k - \overline{ \lam_j} } \rp
%		- \frac{1}{\pi} \int_{-\infty}^{\infty}    
	+ 2 \int_{-\infty}^\infty
%		    \frac{ |s - \eta_k| \log(1-\eps s | \rho(s)|^2) }{(s-\eta_k)^2+\tau_k^2} ds
			 \frac{ |s - \eta_k| \kappa(s) }{(s-\eta_k)^2+\tau_k^2} ds
	\rb\\
&
	 \mod{2\pi}	~.    
\end{align*} 
In the non-generic case in which two or more $\lam_j \in \poles^+$ have the same real part one still observes a form of soliton resolution akin to \eqref{sol.res.generic}. In this case, the one-solitons in \eqref{sol.res.generic} corresponding to spectral values with the same real part coalesce to form higher-order solitons called breathers; these breathers are 
%localized traveling waves whose amplitudes exhibit quasi-periodic oscillations in the frame of the traveling wave. 
\sidenote{16}\revisedtext{spatially localized and temporally quasiperiodic when viewed in the moving frame with constant velocity $-4 \real \lam$.}

To obtain a similar asymptotic formula for \eqref{DNLS1}, we use the gauge transformation \eqref{G} to write 
\sidenote{8}
$$
u(x,t) = 
	\left[ \calG^{-1}	
			\circ
				\revisedtext{\calM(t)}
			\circ\calG
	\right]	(u_0)  (x)
$$ 
where \sidenote{8}\revisedtext{$\calM(t)$} is given by \eqref{sol.op}.
We derive an asymptotic formula for $u(x,t)$ in terms of spectral data for $q_0 = \calG (u_0)$ in Proposition \ref{prop:u.gauge.expansion}, which plays a key result and introduces some complications in the asymptotic formulas for small $\xi$.   Recall Definition \ref{def:U}, the fact that $\calG$ maps dense open subsets to dense open sets in $H^{2,2}(\R)$,
and the space-time region  defined in \eqref{S.cone}. If $u_0 \in \calG^{-1}(U)$, then $q_0 = \calG(u_0)$ has 
no spectral singularities and the scattering coefficient $\balpha$ for $q_0$ has at most finitely many zeros. 
%%%%%%%%%%%%%%%%%%%%%%%%%%%%%
%%
%%		PAP added a bit of notation 
%%
%%%%%%%%%%%%%%%%%%%%%%%%%%%%%
In what follows, set
\begin{equation}
\label{usol}
\usol(x,t;\calD_I) := \calG^{-1}(\qsol(\cdot,t;\calD_I)).
\end{equation}

\begin{theorem}
\label{thm:long-time-gauge}
Suppose that $u_0 \in \calG^{-1}(U)$,  let $q_0 = \calG (u_0)$, and  
$$\calR(q_0) = \left\{ \rho, \{ (\lam_k,C_k) \}_{k=1}^N \right\}.$$ Fix $v_1,v_2,x_1,x_2$ as in Theorem \ref{thm:long-time}, let $I=[-v_2/4,-v_1/4]$,   $\xi = -x/(4t)$, and  $\eta = \pm 1$ for $\pm t > 0$. 
%Let $u_\sol(x,t) = \calG^{-1} \left( q_\sol(\cdot,t) \right)(x)$. 
Fix $M > 0$.
The solution $u(x,t)$ of \eqref{DNLS1} has the following asymptotics as $|t| \rarr \infty$ in the cone
$\calS(v_1,v_2,x_1,x_2)$.
\begin{enumerate}
\item[(i)]	 For $|\xi| \geq M t^{-1/8}$,
$$ 
    u(x,t) =  
      u_\sol(x,t;\calD_I) e^{i\alpha_0(\xi,\eta)}
      \left[ 1 + t^{-1/2} g(x,t) \right] + \bigO{t^{-3/4}} 	 
$$
\item[(ii)]	For $|\xi| \leq M t^{-1/8}$,
$$ 
    u(x,t) =u_\sol(x,t;\calD_I) e^{i\alpha_0(\xi,\eta)} F(\xi,t,\eta) 
    \left[ 1 + t^{-1/2}\widetilde{g}(x,t) \right] + \bigO{t^{-3/4}}. 
$$
\end{enumerate}
%%
%%	PAP		The next line gets modified because of the added notation...
%%
%In the above formulas, $\usol(x,t;\calD_I) := \calG^{-1}(\qsol(\cdot,t;\calD_I))$ 
In the above formulas,
%%%%%%%%%%%%%%%%%%%%%%%%%%%%%
%%	
%%	PAP substituted \kappa(s) for log
%%	and used abbreviated notation for $\calN^{sol}$
%%
%%%%%%%%%%%%%%%%%%%%%%%%%%%%%
denoting
${\mathcal{N}}^{\mathrm sol} \coloneqq {\mathcal{N}}^{\mathrm sol}(\xi'; x,t \vert {\mathcal D}_I),$
\begin{gather*}
 \begin{aligned}
	\alpha_0(\xi,\eta)	&= 	
%	  \frac{1}{\pi} \int_{\negint} \frac{1}{\lam} 
%	     \log\left(1-\eps \lam|\rho(\lam)|^2 \right) \, d \lam
		-2 \int_{\negint} \frac{\kappa(\lam)}{\lam} \, d\lam
	  + 4 \sum_{\mathclap{\Re \lam_k \in \negint \setminus I}} \arg \lam_k , 
	\\
	g(x,t) &= 
	  \frac{ f(x,t) }{\qsol(x,t; \calD_I)}
	  + 2^{1/2}\eps \Re 
	  \lb
	    A_{12}(\xi,\sgnt) 
%	    {\mathcal N}^{\mathrm sol}_{11}(\xi;x,t  \vert  \calD_I)
%		  \overline{ 
%		  {\mathcal N}^{\mathrm sol}_{12}(\xi;x,t  \vert  \calD_I) }
		  {\mathcal N}^{\mathrm sol}_{11}
		  \overline{{\mathcal N}^{\mathrm sol}_{12}}
		  e^{i \phi(\xi) }
	  \rb  ,    
	\\
	\widetilde{g}(x,t) &= g(x,t) \\
	&\quad
	  + (1-G(\xi,t,\sgnt)) \overline{A_{12}(\xi,\sgnt)} 
	  \exp 
	    \left( 
	      %4i \sum_{\mathclap{\Re \lam_k \in \negint \setminus I}} \arg \lam_k 
	      4i \quad \sum_{\mathclap{\Re \lam_k \in \negint \setminus I}} \arg \lam_k 
	    \right) 
	  \int_x^\infty \usol(y,t;\calD_I) dy ,
  \end{aligned}
\intertext{where $\phi(\xi)$ is given by \eqref{phase.phi} and, setting $p=e^{i\pi/4} |8t\xi^2|^{1/2}$,}
	F(\xi,t,\eta) =	
	  \left[ e^{p^2/4} p^{-i\eta \kappa(\xi)} D_{i\eta \kappa(\xi)}(p) \right]^{-2} , 
	  \qquad
	G(\xi,t,\eta)	=	
	  \frac{p D_{i\eta\kappa(\xi)-1}(p)}{D_{i\eta\kappa(\xi)}(p)}.
\end{gather*}
and
$D_a(z)$ denote the parabolic cylinder functions whose properties are given in \cite{DLMF}.
\end{theorem}

\begin{remark}
The presence of the phase $e^{i\alpha_0(\xi,\sgnt)}$ in the above formulas represents the
mismatch in the phase of $u(x,t) = \calG^{-1}(q)$ and $\usol(x,t;\calD_I)$ as defined in \eqref{usol} caused by the cumulative interaction of $\usol(y,t;\calD_I)$ with the radiation and soliton components of the full system which are traveling faster than our chosen reference frame $\calS$. Because the velocities are proportional to $- \Re \lam$ (recall that $v = - 4 \Re \lam_k$ for solitons), faster velocities correspond to the part of the spectrum $\negint$ which lies to the left (resp. right) of $\xi$ as $t\to \infty$ (resp. $t \to-\infty$).
\end{remark}	

An important consequence of our analysis is that it provides a proof of asymptotic stability of  N-soliton solutions. % in $L^\infty$-norm. 
Until now, the only known result about stability of DNLS soliton is due to Colin and Ohta \cite{CO06} who proved orbital stability of 1-soliton solutions in $H^1$.
\sidenote{17}\revisedtext{Our result gives a detailed description of the long-time behavior of perturbed $N$-solitons in the form of the sum of  $N$ $1$-soliton solutions with parameters
 close to those of the unperturbed $N$-soliton. It}
%Our result  
is obtained by combining the Lipschitz continuity of the forward and inverse scattering maps described by Theorems 1.3 and 1.8 of Paper III with the long time result of Theorem~\ref{thm:long-time}. % immediately gives the following asymptotic stability result for $N$-soliton solutions of \eqref{DNLS2}. 

\begin{theorem}
Given an $N$-soliton $q_\mathrm{sol}(x,t;\mathcal{D}^\mathrm{sol})$ solution of \eqref{DNLS2}    with initial data in $U_N$ such that $\Re \lam_k \neq \Re \lam_j$, $j\neq k$, with scattering data $\mathcal{D}^\mathrm{sol} = \{ 0, \{ \lam_k^\mathrm{sol}, C_k^{\mathrm{sol}} \}_{k=1}^N \}$, there exist positive constants $\eta_0=\eta_0(q_{\mathrm{sol}})$, $T=T(q_{\mathrm{sol}})$, and $K= K(q_{\mathrm{sol}})$ such that any initial data $q_0 \in H^{2,2}(\R)$ with
\[
	\eta_1 := \| q_0 - q_{\mathrm{sol}}(\, \cdot\,,0;\mathcal{D}^\mathrm{sol})  \|_{H^{2,2}(\R)} 	\leq \eta_0
\]
also lies in $U_N$, with scattering data 
$\mathcal{D} = \mathcal{R}(q_0) = \{ \rho, \{ \lam_k, C_k \}_{k=1}^N \}$ 
satisfying
\begin{equation}\label{near.soliton.data}
\|\rho\|_{H^{2,2}(\R)} +	\sum_{k=1}^N | \lam_k - \lam_k^\mathrm{sol}| + |C_k - C_k^\mathrm{sol}| \leq K \eta_1
\end{equation}
and the solution of the Cauchy problem \eqref{DNLS2}-\eqref{data} asymptotically separates into a sum of N 1-solitons 
\[
	\sup_{x \in \R} 
	\,\,
	\left| 
	\,
	q(x,t) - \sum_{k=1}^N  \mathcal{Q}_\mathrm{sol}(x,t;\lam_k, x_k^\pm, \alpha_k^\pm) 
	\,
	\right|
	\leq K \eta_1 |t|^{-1/2}, \qquad  |t| > T
\]	
where the 1-solitons $\mathcal{Q}_\mathrm{sol}$ are given by \eqref{sol.res.generic}-\eqref{alpha phases}.
\end{theorem}

To state the corresponding result for \eqref{DNLS1} let 
$u_\mathrm{sol}(x,t;\mathcal{D}^\mathrm{sol})$ denote the $N$-soliton solution of \eqref{DNLS1} such that $\mathcal{R}\lp \mathcal{G}(u_\mathrm{sol}) \rp = \mathcal{D}^{\mathrm{sol}} := \left\{ 0, \{(\lam_k^\mathrm{sol}, C_k^\mathrm{sol} )\}_{k=1}^N \right\}$, i.e. 
\begin{gather*}
	\mathcal{G} \lp u_\mathrm{sol}(\cdot,t;\mathcal{D}^\mathrm{sol}) \rp = 
	q_\mathrm{sol}(x,t;\mathcal{D}^\mathrm{sol}),
\shortintertext{and similarly let }
	\mathcal{U}_\mathrm{sol}(x,t;\lam,x_k,\varphi_k) 
	:= \mathcal{G}^{-1} \lp \mathcal{Q}_\mathrm{sol}(\cdot,t;\lam,x_k,\varphi_k)\rp
\end{gather*}
be the inverse gauge transformation of the 1-soliton solutions of \eqref{DNLS2} defined by \eqref{sol.res.generic}-\eqref{alpha phases}. Then applying \eqref{thm:long-time-gauge} to the previous result gives  

\begin{theorem}
Given an $N$-soliton of \eqref{DNLS1} $u_\mathrm{sol}(x,t;\mathcal{D}^\mathrm{sol}) \in \mathcal{G}^{-1}(U_N)$ as defined above with $\Re \lam_k \neq \Re \lam_j$, $j\neq k$ there exist positive constants $\eta_0=\eta_0(q_{\mathrm{sol}})$, $T=T(q_{\mathrm{sol}})$, and $K= K(q_{\mathrm{sol}})$ such that any initial data $u_0 \in H^{2,2}(\R)$ with
\[
	\eta_1 := \| u_0 - u_{\mathrm{sol}}(\, \cdot\,,0;\mathcal{D}^\mathrm{sol}) \|_{H^{2,2}(\R)} 	\leq \eta_0
\]
also lies in $\mathcal{G}^{-1}(U_N)$ with scattering data 
$\mathcal{D} := \mathcal{R}(\mathcal{G}(u_0)) = \{ \rho, \{ \lam_k, C_k \}_{k=1}^N \}$ 
satisfying \eqref{near.soliton.data}. Moreover the solution of the Cauchy problem \eqref{DNLS1}-\eqref{data1} asymptotically separates into a sum of N 1-solitons 
\[
	\sup_{x \in \R} 
	\,\,
	\left| 
	\,
	u(x,t) - \sum_{k=1}^N  \mathcal{U}_\mathrm{sol}(x,t;\lam_k, x_k^\pm, \varphi_k^\pm) 
	e^{ i \alpha_0(\Re \lam_k, \pm)}
	\,
	\right|
	\leq K \eta_1 |t|^{-1/2}, 
	\quad   |t| > T
\]	
where the phase corrections $\alpha_0(\xi,\eta)$ are defined in Theorem~\ref{thm:long-time-gauge}.
\end{theorem}

\begin{remark}
	In both of the above theorems, the condition that the discrete spectral points have distinct real parts is generic are made so that the stability is easier to state. If one considers multi-solitons with spectral points with equal real part (breathers) then the above results still hold for (generic) perturbations which separate the real parts, but the time scale on which the solitons separate will depend upon the particular perturbation. 
\end{remark}
%\begin{definition}
%For each nonnegative integer $N$, $t \in \R$, $\rho \in P$, and $\{ \lam_j, C_j \}_{j=1}^N \subset \C^+ \times \C^\times$, the linear evolution $\Phi_t: V_N \rarr V_N$ is given by
%$$ \Phi_t \left(\rho, \{ (\lam_j, C_j) \}_{j=1}^N \right) = \left( e^{-4i(\dotarg)^2 t}  \rho(\dotarg), \{ ( \lam_j, C_j e^{-4i\lam_j^2 t } )\}_{j=1}^N \right). $$
%\end{definition}
%It is easy to see that the map $\Phi_t$ preserves $V_N$ for each $N$ and is jointly continuous in $t$ and the data $\left(\rho, \{ (\lam_j, C_j) \}_{j=1}^N \right)$.

\medskip
We close this introduction by sketching the content of this paper. 
The proofs of  
Theorems \ref{thm:long-time} and \ref{thm:long-time-gauge} are given 
in Sections \ref{sec:deform},  \ref{sec:models}, and \ref{sec:largetime}, using the steepest descent method of Deift and Zhou \cite{DZ93}, the later approach of  McLaughlin-Miller \cite{MM08} and  Dieng-McLaughlin \cite{DM08}, and the  recent work of Borghese, Jenkins and McLaughlin \cite{BJM16} on the focusing cubic NLS which shows how to treat a problem with discrete as well as continuous spectral data. Following \cite{BJM16}, we reduce RHP \ref{RHP2} to an `outer' model which describes the asymptotic behavior of solitons, and an `inner' model which computes the contributions due to the interactions of solitons and radiation. 

In Section \ref{sec:deform}, we begin with the row-vector valued function solution of  RHP  \ref{RHP2.row} %(see Remark \ref{RHP2}) 
and deform it to an RHP on a new contour $\Sigma^{(2)}$ whose jump matrices approach the identity exponentially fast away from the critical point $\xi$
 (see Fig. \ref{fig:n2def}). In Section \ref{sec:conj},
%we  conjugate the row-vector RHP for $n=(n_1,n_2)$  to a form suitable for lensing, i.e., deforming the contour $\R$ about the critical point $\xi$ so that the jump matrix of the new Riemann-Hilbert problem approaches the identity exponentially fast away from the critical point $\xi$. The solution of the conjugated RHP, Problem \ref{rhp.n1}, is  the unknown $n^{(1)}=n \delta^{-\sigma_3}$ where $\delta$ solves a scalar RHP with contour $(-\infty,\xi)$ and has asymptotics $1 + \bigO{1/z}$ as $|z| \rarr \infty$.  The jump matrix for Problem \ref{rhp.n1} has a `good factorization for lensing.
we conjugate the row-vector RHP for $n = (n_1, n_2)$, by defining a new unknown $n^{(1)} = n \delta^{-\sigma_3}$, where $\delta$ is the unique function satisfying Lemma \ref{lem:T}. The new unknown $n^{(1)}$ is amenable to lensing, i.e., deforming the contour $\mathbb{R}$ about the critical point $\xi$ so that the jump matrix of the deformed Riemann-Hilbert problem approaches the identity exponentially fast away from the critical point $\xi$. 
 In order to deform the contour, we extend the scattering data in the jump matrix of \revisedtext{RHP} \ref{rhp.n1} into the complex plane to define a new unknown $n^{(2)}$ (see eq. \eqref{n2 def}). The new unknown solves a mixed $\dbar$-Riemann-Hilbert problem, \revisedtext{$\dbar$/Riemann-Hilbert Problem \ref{rhp.n2}}, described in Section \ref{sec:extensions}). The extension introduces non-analyticity of $n^{(2)}$ which is solved away at a later step. The solution $n^{(2)}$ coincides with $n^{(1)}$ in the sectors $\Omega_2$ and $\Omega_5$ (see figure \ref{fig:n2def}) and is piecewise analytic in the sectors $\Omega_1$--$\Omega_6$ with no jumps across the real axis. 

In Section \ref{sec:models}, we construct a solution $\calN^{\mathrm{RHP}}$ 
of 
%the Riemann-Hilbert problem, Problem \ref{rhp.Nrhp}, 
\revisedtext{RHP \ref{rhp.Nrhp},}
%corresponding to 
\revisedtext{determined by the asymptotic and jump conditions in $\dbar$/Riemann-Hilbert Problem \ref{rhp.n2}}. Thus the function $$n^{(3)} = n^{(2)}\left(\calN^{\mathrm{RHP}}\right)^{-1}$$ obeys a pure $\dbar$-problem, 
\sidenote{55}
\revisedtext{$\dbar$-}Problem \ref{dbar.n3}. Because the original RHP contains both `continuous' and `discrete' data, the solution $\calN^{\mathrm{RHP}}$ consists of
an `outer' model for the soliton components (see Section \ref{sec:outer model}, 
%Problem 
\sidenote{34}
\revisedtext{RHP} 
\ref{outermodel} and Proposition \ref{outer.soliton}) and an `inner' model for the stationary phase point (see Section \ref{sec:local model},  \sidenote{55}\revisedtext{RHP} \ref{rhp.localmodel} and Proposition \ref{prop:PCmodel.est}).  The outer and inner models are used to build a
\sidenote{55} 
parametrix for \revisedtext{RHP}
\ref{rhp.Nrhp} in Section \ref{sec:RHP.exist}. The `gluing' of parametrices is carried out by solving a small-norm Riemann-Hilbert problem, 
%Problem \ref{rhp.E}.  
\sidenote{55}\revisedtext{Riemann-Hilbert Problem \ref{rhp.E}.}
The $\dbar$ problem for $n^{(3)}$ is solved in Section \ref{sec:dbar} and is shown to have asymptotic behavior 
\sidenote{18}
$$
n^{(3)}(z) = 
%\begin{pmatrix} 1 &&&& 0 \\ 0 &&&& 1 \end{pmatrix}
\revisedtext{I
+ \bigO{t^{-3/4}}, \quad I = \begin{pmatrix} 1 &&&& 0 \\ 0 &&&& 1 \end{pmatrix}
}.
$$ 
Thus, (suppressing the $(x,t)$ dependence for brevity)
$$ n(z) = n^{(3)}(z) \mathcal{N}^{\mathrm{RHP}}(z) \mathcal{R}^{(2)}(z)^{-1} \delta(z)^{\sigma_3}.$$
The leading contribution to $q(x,t)$ in \eqref{q.lam} comes from the explicitly computable model factor $\calN^{\mathrm{RHP}}$ owing to the asymptotics of $n^{(3)}$, the fact that $\mathcal{R}^{(2)}$ is the identity in sectors $\Omega_2$ and $\Omega_5$ (and we can take $z \rarr \infty$ in either sector), and the diagonal matrix $\delta^{-\sigma_3}$ does not change $n_2$ at order $\bigO{1/z}$.

\sidenote{55}
With the solution of \revisedtext{RHP} \ref{RHP2} in hand, we give the proof of Theorems \ref{thm:long-time} and 
\ref{thm:long-time-gauge} in Sections \ref{subsec:sol-q} and \ref{subsec:sol-u}, respectively.  To prove Theorem \ref{thm:long-time-gauge}, we establish an asymptotic formula for the phase factor in the gauge transformation \eqref{G} in terms of spectral data, Proposition \ref{prop:u.gauge.expansion}. 
This in turn relies on a weak Plancherel formula, Lemma \ref{lem.Plancherel}, proved in Appendix \ref{app:Weak}.
We construct $N$-soliton solutions $q_\sol$ for \eqref{DNLS2} in Appendix \ref{app:solitons}.

%\input{reduction}    		%% 	Regroups conjugation and extension sections

%%%%%%%%%%%%%%%%%%%%%%%%%%%%%%%%%%%%%%%
%
%		File dnls-imrn-arXiv-reduction-longtime.tex
%
%      This file regroups the 2 files 
%
%			dnls-imrn-arXiv-conjugation.tex  
%			dnls-imrn-arXiv-extensions.tex
%
%		Created 4.27.2017
%
%%%%%%%%%%%%%%%%%%%%%%%%%%%%%%%%%%%%%%%
 
\section{Deformation to a mixed \texorpdfstring{$\bar{\partial}$}{DBAR}-Riemann-Hilbert Problem}
\label{sec:deform}

This section is devoted to the two first transformations in the reduction of the original RHP \ref{RHP2} to a model that can be solved explicitly
and provides the precise behavior of the solution of the DNLS equation for long time up to small terms of order $O(|t|^{-3/4})$.
We present the analysis for $t>0$ and $t<0$ simultaneously by introducing the parameter $\eta= \sgn(t)$.

The first step  is the conjugation of the solution of RHP \ref{RHP2} by a scalar function  $\delta$ defined in \eqref{T}, which is itself the solution of a scalar
RHP (Section \ref{sec:conj}). This operation is standard  and its effect is described in detail in \cite{DZ03}, see also \cite{BJM16} and Paper 2.
The second step (Section \ref{sec:extensions}) is the deformation of contours from the real axis to the contour $\Sigma^{(2)}$ shown in Figure \ref{fig:n2def}. 
Our presentation follows 
Paper 2 with the addition of the  treatment of the discrete data associated to the  residue conditions \cite{BJM16}.

%%%%%%%%%%%%%%%%%%%%%%%%%%%%%%%%%%%%%%%
%
%		CONJUGATION OF RHP
%		File dnls-imrn-arXiv-conjugation.tex
%		Created 4.7.2017
%		Some modifications by PAP on 4.12.2017
%
%%%%%%%%%%%%%%%%%%%%%%%%%%%%%%%%%%%%%%%

\subsection{Conjugation}
\label{sec:conj}
The long-time asymptotic analysis of RHP~\ref{RHP2} is determined by the growth and decay of the exponential function $e^{2it\theta}$ 
appearing in both the jump relation (Problem \ref{RHP2}(iii))  and the residue conditions (Problem \ref{RHP2}(iv)).
Let
%%\begin{equation}\label{parameters}
%%	\xi = -\frac{x}{4t}  
%%%%		\qquad \qquad
%%%	\sgnt = \sgn t ~.
%\end{equation}
$\xi=-\frac{x}{2t}$ be the  (unique) critical point of the phase $\theta$ defined in \eqref{phase.lambda}. For $|t| \gg 1$, $|e^{2it \theta}| \ll 1$ whenever $\sgnt \Re(\lambda - \xi) < -c < 0$, and $|e^{2it \theta}| \gg 1$ whenever $\sgnt \Re(\lambda - \xi) > c > 0$. RHP~\ref{RHP2} is formulated from the scattering data in such a way that its solution has identity asymptotics as $x \to +\infty$ with $t$ fixed.  We are interested in the behavior of solutions when $|t| \to \infty$ with $x/t$ fixed in some interval. It is necessary to renormalize the RHP so
that it is well behaved as $t \to \infty$ with $x/t$ in the interval of interest. 

\begin{figure}[htb]
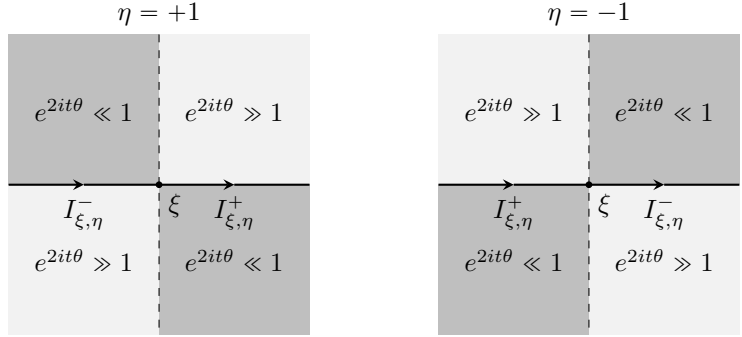

\centering
\hspace*{\stretch{1}}
\FigPhaseA[.5]
\hspace*{\stretch{1}}
\FigPhaseB[.5]
\hspace*{\stretch{1}}
\caption{The regions of growth and decay of the exponential factor $e^{2it\theta}$ in the $\lambda$-plane for either sign of $\sgnt = \sgn t$. 
%Those poles $\lambda_k$ which lie in the region $e^{2it\theta} << 1$ (resp. ${}\gg 1$) correspond to indices $k \in \negpoles$ (resp. $k \in \pospoles$) .
\label{fig:theta signs}
}
\end{figure}

Orient the intervals $I^\pm_{\xi,\sgnt}$ from left-to-right, see \eqref{posint} and Figure~\ref{fig:theta signs}.
Note that if the sign of $t$, and thus $\sgnt$, is changed with $\xi$ held fixed, the effect is simply to exchange intervals 
$I_{\xi,\eta}^\pm$.
Let
\begin{equation}\label{T}
	\begin{gathered}
	\delta(\lambda) = \delta(\lambda, \xi, \sgnt)  := 
	%\prod_{k \in \negpoles} \lp \frac{ \lambda - \lambda_k^*}{\lambda-\lambda_k}  \rp 
	\exp \lp i \int_{\negint}  \frac{\kappa(z) }{z-\lambda} dz \rp,
	\quad
	\kappa(z) = -\frac{1}{2\pi} \log(1 -\eps z |\rho(z)|^2).
	\end{gathered}
\end{equation}

\begin{lemma}\label{lem:T}
The function $\delta(\lambda)$ defined by \eqref{T} has the following properties:
\begin{enumerate}
	\item[(i)] $\delta$ is meromorphic in $\C \setminus \negint$. 
%	For each $k \in \negpoles$, 
%	$\delta(\lam)$ has a simple pole at $\lambda_k$ and a simple zero at $\lambda_k^*$.
	\item[(ii)] For $\lambda \in \C \setminus \negint$, $ \delta(\lambda) \overline{\delta(\overline{\lambda})} = 1$.
	Moreover,
	$\displaystyle
		e^{-\| \kappa \|_\infty/2} \leq 
		%\bigg| 
		\left| \delta(\lam) \right|
		%\prod_{k \in \negpoles} \lp \frac{\lam - \lam_k}{\lam - \lam_k^*} \rp 
		%\bigg|
		\leq e^{\| \kappa \|_\infty/2} ~.
	$
	\item[(iii)] For $\lambda \in \negint$, $\delta$'s boundary values  $\delta_\pm$, as $\lambda$ approaches the real axis from above and below, satisfy 
	\begin{equation}\label{Tjump}
		\delta_+(\lambda) / \delta_-(\lambda) = 1 - \eps \lambda |\rho(\lambda)|^2.
		%\quad \lambda \in \negint.
	\end{equation}
		\item[(iv)] As $ |\lambda| \to \infty$ with $|\arg(\eta \lambda)| \neq \pi$, 
	\begin{equation}\label{Texpand}
		{  \delta(\lambda)} = 1 + 
		%2 \sum_{k \in \negpoles } \Im z_k +
		\frac{\delta_1}{\lam} 
		+ \bigo{ \lambda^{-2} },
		\quad \delta_1 = -i \int_{\negint} \kappa(z) dz .
	\end{equation}
	\item[(v)] As $\lambda \to \xi$ along any ray $\xi + e^{i \phi} \R_+$ with 
	$| \arg (\sgnt(\lambda-\xi) )| < \pi$
	\begin{gather} 
	\nonumber
	%\label{Tbound}
		\left| \delta(\lambda, \xi,\sgnt) - \delta_0(\xi,\sgnt) (\sgnt(\lambda-\xi))^{i \sgnt \kappa(\xi)} \right| 
		\lesssim_{\rho,\phi} - |\lambda - \xi | \log |\lambda - \xi|.
\intertext{The implied constant depends on $\rho$ through its $H^{2,2}(\R)$-norm and is independent of $\xi$. Here $\delta_0(\xi, \sgnt) = \exp(i\beta(\xi,\xi))$ is a complex unit with}
		\nonumber
		\label{delta0 arg}
		\beta(z,\xi) = - \sgnt \kappa(\xi) \log(\sgnt(z-\xi+1))
		+ \int_{\negint} \frac{ \kappa(s) - \chi(s) \kappa(\xi)}{s-z} ds,
	\end{gather}
	and $\chi(s)$ is the characteristic function of the interval $\sgnt \xi - 1 < \sgnt s < \sgnt \xi$. 
In all of the above formulas, we choose the principal branch of power and logarithm functions.
\end{enumerate}
\end{lemma}

\begin{proof}
\sidenote{19}
Parts $(i)$--\revisedtext{$(iii)$} are elementary consequences of the definition \eqref{T} and the Sokhotski-Plemelj formula. 
For part $(iv)$, one geometrically expands %the product term and 
the factor $(z-\lambda)^{-1}$ for large $\lambda$, and uses the estimate $\| \kappa \|_{L^1(\R)} \lesssim \| \rho \|_{H^{2,2}(\R)}$ to bound the remainder in the integral term for $\lambda$ bounded away from the contour of integration. 
The proof of part $(v)$ can be found in Appendix A of \cite{LPS16}.
\end{proof}

We now define a new unknown function $\nk{1}$ using 
%our partial transmission coefficient 
\sidenote{20}\revisedtext{the function $\delta(\lam)$:}
\begin{equation}\label{n1}
\nk{1}(\lambda) = \nn(\lambda) \delta(\lambda)^{-\sigma_3}.
\end{equation}
We claim that $\nk{1}$ satisfies the following RHP.

\begin{RHP}\label{rhp.n1}
Find an analytic row vector-valued function $\nk{1}: \C \setminus (\R  \cup \poles) \to \C^2$ with the following properties:
\begin{enumerate}
	\item[(i)] $\nk{1}(\lambda) = (1,0) + \bigo{\lambda^{-1}}$ as $\lambda \to \infty$. 
	\item[(ii)] For $\lambda \in \R$, the boundary values
	$\nk[\pm]{1}$ satisfy the jump relation 
	$\nk[+]{1}(\lambda) = \nk[-]{1}(\lambda) \vk{1}(\lambda)$ where
	\begin{equation}\label{V1}
		\vk{1} (\lambda) = \begin{cases}
			\triu{ \rho(\lambda) \delta(\lambda)^{2} e^{2it \theta} }
			\tril{ -\eps \lambda \bar{\rho}(\lambda) \delta(\lambda)^{-2} e^{-2it \theta} }	
			& \lambda \in \posint, \bigskip \\
			\tril{ \frac{ -\eps \lambda \bar{\rho}(\lambda) \delta_-(\lambda)^{-2}}
			{1 -\eps \lambda |\rho(\lambda)|^2} e^{-2it \theta} }
			\triu{ \frac{ \rho(\lambda) \delta_+(\lambda)^{2}}
			{1 -\eps \lambda |\rho(\lambda)|^2} e^{2it \theta}}
			& \lambda \in \negint.
		\end{cases}
	\end{equation}
	\item[(iii)]  $\nk{1}(\lambda)$ has simple poles at each $p \in \poles$: 
	%$ \displaystyle
	\[
		\Res_{\lambda = p} \nk{1}(\lambda) 
		= \lim_{\lambda \to p} \nk{1}(\lambda) \vk{1}(p),
	\]
	%$
	where for each $\lam_k \in \poles^+$
	\begin{equation} \label{n1 residue matrices}		
		\vk{1}(\lambda_k) %&
		= \tril[0]{ \lambda_k C_k \delta(\lambda_k)^{-2} e^{-2it \theta} }, \, \,
		\vk{1}(\bar{\lambda}_k) %&
		= \triu[0]{ \eps \overline{C}_k \delta(\overline{\lambda}_k)^{2} e^{2it \theta} } ~.
	\end{equation}
\end{enumerate}
\end{RHP}

\begin{proposition}
%\label{prop:M1}
Suppose that $\nn$ satisfies  { RHP~\ref{RHP2.row}}. 
Then $\nk{1}$ defined by \eqref{n1} satisfies  RHP~\ref{rhp.n1}.
\end{proposition}

\begin{proof}
The fact that $\nk{1}$ is analytic in $\C \setminus (\R \cup \poles)$, and approaches $(1,0)$ as $\lambda \to \infty$ follows directly from its definition, Lemma~\ref{lem:T}, and the analytic properties of $\nn$. % given in RHP~\ref{RHP2}. 
The relation 
$\vk{1}(\lambda) = \delta_-^{\sig}(\lambda) \lb e^{it\theta \ad \sigma_3} v(\lambda) \rb\delta_+^{\sig}(\lambda)$, the standard factorizations of $v$
%\[
%	v= \striu{\rho e^{2it \theta} } \stril{-\eps \lambda \rho^* e^{-2it \theta} }
%	=\stril{ \frac{-\eps \lambda \rho^* e^{-2it \theta} }{1-\eps \lambda |\rho|^2} }
%	 \sdiag{1-\eps \lambda |\rho|^2}{(1-\eps \lambda |\rho|^2)^{-1} }
%	 \striu{ \frac{\rho e^{2it \theta} }{1-\eps \lambda |\rho|^2} }
%\]
\begin{align*}
	v 	& = \triu{\,\,\,\, \rho \,\,\,\, } \tril{-\eps \lambda \rhobar \,\,\,\,} \\
	 	&=\tril{ \frac{-\eps \lambda \rhobar  }{1-\eps \lambda |\rho|^2} }
	 			\diag{1-\eps \lambda |\rho|^2}{(1-\eps \lambda |\rho|^2)^{-1} }
	 		\triu{ \frac{\rho }{1-\eps \lambda |\rho|^2} },
\end{align*}
and the jump relation \eqref{Tjump} satisfied by $\delta(z)$ on $\negint$ allow us to write $\vk{1}$ as in \eqref{V1}.
%\begin{equation*}
%	\vk{1}(\lambda)= \begin{cases}
%	\triu{ \rho(z) \delta^2(\lambda) e^{2it \theta} } 
%	\tril{-\eps \lambda \rho(\lambda) \delta^{-2}(\lambda)  e^{-2it \theta} }  
%	&  \lambda \in \posint  \\
%	\tril{ \frac{ -\eps \lambda \bar{\rho}(z) \delta_-^{-2}(\lambda) }{ 1-\eps \lambda \rho \bar{\rho}} e^{-2it \theta} } 
%	\triu{ \frac{ \rho(z) \delta_+^2(\lambda) }{1-\eps \lambda \rho \bar{\rho} } e^{2it \theta} }
%	& \lambda \in \negint.
%	\end{cases}
%\end{equation*}
Concerning the residues \eqref{n1 residue matrices}, as $\delta(\lambda)$ is analytic near each $p \in \Lambda$,
%%%%%%%%%%%%%%%%%%%%%%%%%%%%%
%%
%%	PAP changed this to align because of margin 
%%	troubles
%%
%%%%%%%%%%%%%%%%%%%%%%%%%%%%%
\begin{align*}
	\Res_{\lam=p} \nk{1} 
%	= \res_{\lam_k} \nn(\lam) \delta(\lam_k)^{-\sig}
	&= \lim_{\lam \to p} \nn(\lam) v(p) \delta(p)^{-\sig}	
	= \lim_{\lam \to p} \nk{1}(\lam) \delta(p)^{\sig} v(p) \delta(p)^{-\sig}	\\
%	= \lim_{\lam \to p} \nn(\lam) \vk{1}(p) \delta(p)^{-\sig}.
	& =\lim_{\lam\to p} \nk{1}(\lam) \vk{1}(p)
\end{align*}
with $v^{(1)}$ defined in \eqref{n1 residue matrices}.
%	\qedhere
%\]
\end{proof}

%%%%%%%%%%%%%%%%%%%%%%%%%%%%%%%%%%%%%%%
%
%		EXTENSIONS OF JUMP FACTORIZATION
%		File dnls-imrn-arXiv-extensions.tex
%		Created 4.7.2017
%
%%%%%%%%%%%%%%%%%%%%%%%%%%%%%%%%%%%%%%%

\subsection{\texorpdfstring{$\dbar$}{DBAR}-extensions of jump factorization}
\label{sec:extensions}

We now introduce a transformation which uses the factorization \eqref{V1} to deform the jump matrix $\vk{1}$,  replacing it with new jumps along contours in the complex plane which are near identity. 
%The phase function \eqref{phase.lambda} has a single (quadratic) critical point at $\xi = -x/4t$. 
Let
\begin{equation*}
%\label{Sigma_k}
	\begin{gathered}
	\Sk{2} = \Sigma_1 \cup \Sigma_2 \cup \Sigma_3 \cup \Sigma_4 \\
	\Sigma_k = \xi + e^{\frac{i\pi}{4}(2+(2k-3)\sgnt) }\, \R_+ , \quad k = 1,2,3,4,
	\end{gathered}
\end{equation*}
with each ray oriented with increasing (resp. decreasing) real part for $\sgnt = +1$ (resp. $\sgnt = -1$). The function $e^{2it \theta}$ is exponentially increasing along $\Sigma_1$ and $\Sigma_3$ and decreasing along $\Sigma_2$ and $\Sigma_4$, while the reverse is true of $e^{-2it \theta}$.
Let  $\Omega_k,\, k=1,\dots,6$, denote the six connected components of $\C \setminus \lp \R \bigcup_{k=1}^4 \Sigma_k \rp$, starting with sector $\Omega_1$ between $\posint$ and $\Sigma_1$ and numbered consecutively continuing counterclockwise (resp. clockwise) if $\sgnt = +1$ (resp. $\sgnt = -1$) as shown in Figure~\ref{fig:n2def}.

In order to deform the contour $\R$ to the contour $\Sk{2}$, we introduce a new unknown $\nk{2}$ 
%obtained from $\nk{1}$ as
%
\begin{align}
\label{n2 def}
	\nk{2}(\lam) = \nk{1} (\lam)\mathcal{R}^{(2)}(\lam).
\end{align}
In each of the sectors $\Omega_k$, $k=1,3,4,6$, which meet the real axis, the condition that 
%the new unknown 
$\nk{2}$ has no jump on the real axis determines the boundary values of $\mathcal{R}^{(2)}$  through the factorization of $\vk{1}$ in \eqref{V1}. 
These factorizations involve the reflection coefficient $\rho$ which does not extend analytically to the complex plane. 
To extend $\mathcal{R}^{(2)}$ off the real axis, we use the method of 
\cite{BJM16,CJ16,DM08} which introduces non-analytic extensions. 
The new unknown $\nk{2}$ will satisfy  a mixed $\dbar$-RHP.
The only condition on the extension is that we have some mild control on $\dbar \mathcal{R}^{(2)}$ sufficient to ensure that the $\dbar$-contribution to the long-time asymptotics of $q(x,t)$ is negligible. 
This is the content of Lemma~\ref{lem:extensions} below. 
We have considerable freedom in choosing the extension. We use this freedom to ensure that: $1)$ the new jumps on $\Sk{2}$ match a well known model RHP;
\sidenote{18, 22}  $2)$ in a small neighborhood of each pole in $\Lambda$\revisedtext{,} $\mathcal{R}^{(2)}(\lam) =I$---this ensures that the residues are unaffected by the transformation. 
We choose $\mathcal{R}^{(2)}$ as shown in Figure~\ref{fig:n2def}, where the functions $R_1,R_3,R_4,R_6$ satisfy
%%%%%%%%%%%%%%%%%%%%%%%%%%%%%%%%%%%%%%%%
%
%  R2 conditions - DBAR extension of the jumps
%
%%%%%%%%%%%%%%%%%%%%%%%%%%%%%%%%%%%%%%%%
%
% local spacing command
\newcommand{\temp}{\frac{ -\eps \xi \rho^*(\xi)}{1 - \eps \xi |\rho(\xi)|^2} 
		\delta_0(\xi,\sgnt)^{-2} (\sgnt(\lam-\xi))^{-2i \sgnt \kappa(\xi)} 
		(1-\indicator(z))}
\sidenote{23}
\begin{subequations}\label{R_k}
\begin{align}
	\label{R1}
	R_1(\lam) &= \begin{dcases}
		 \spacing{-\eps \lambda \overline{\rho(\lam)} \delta(\lam)^{-2} }{\temp}
			& \lam \in \posint \\
		-\eps \xi \overline{\rho(\xi)} \delta_0(\xi, \sgnt)^{-2} (\revisedtext{\sgnt \cdot}(\lam-\xi))^{-2i \sgnt \kappa(\xi)}
		(1-\indicator(\lam)) 		
			& \lam \in  \Sigma_1
	\end{dcases} \bigskip \\
	\label{R3}
	R_3(\lam) &= \begin{dcases}
		\spacing{\frac{ \rho(\lam)}{1 - \eps \lambda |\rho(z)|^2} \delta_+(\lam)^{2}}{\temp} 
			& \lam \in \negint \\
		\frac{ \rho(\xi)}{1 - \eps \lam |\rho(\xi)|^2} 
		\delta_0(\xi, \sgnt)^{2} (\revisedtext{\sgnt \cdot}(\lam-\xi))^{2i \sgnt \kappa(\xi)} 
		(1-\indicator(\lam)) 
			& \lam \in  \Sigma_2
	\end{dcases} \bigskip \\
	\label{R4}
	R_4(\lam) &= \begin{dcases}
		\spacing{\frac{ -\eps \lambda \overline{\rho(\lam)}}{1 -\eps \lam |\rho(\lam)|^2} \delta_-(\lam)^{-2} }
		{\temp}
			& \lam \in \negint \\
		\frac{ -\eps \xi \overline{\rho(\xi)}}{1 - \eps \xi |\rho(\xi)|^2} 
		\delta_0(\xi, \sgnt)^{-2} (\revisedtext{\sgnt \cdot}(\lam-\xi))^{-2i \sgnt \kappa(\xi)} (1-\indicator(\lam)) 
			& \lam \in  \Sigma_3
	\end{dcases} \bigskip \\
	\label{R6}
	R_6(\lam) &= \begin{dcases}
		\spacing{\rho(\lam) \delta(\lam)^{2} }{\temp}
			& \lam \in \posint \\
		\rho(\xi) \delta_0(\xi, \sgnt)^{2} (\revisedtext{\sgnt \cdot}(\lam-\xi))^{2i \sgnt\kappa(\xi)}(1-\indicator(\lam)) 
			& \lam \in  \Sigma_4.
	\end{dcases} 
\end{align}
\end{subequations}
Here $\indicator$ is a $C_0^\infty(\C,[0,1])$ cutoff function supported on a neighborhood of each point of the discrete spectrum such that 
\begin{equation*}
%\label{chi prop}
	\indicator(\lam) = \begin{cases}
		1 & \dist(\lam, \poles) < \poledist/3 \\
		0 & \dist(\lam, \poles) > 2\poledist/3
	\end{cases}
\end{equation*} 
where $\poledist$, defined by \eqref{distances}, is sufficiently small to ensure that the disks of support intersect neither each other nor the real axis.
\begin{figure}[t]
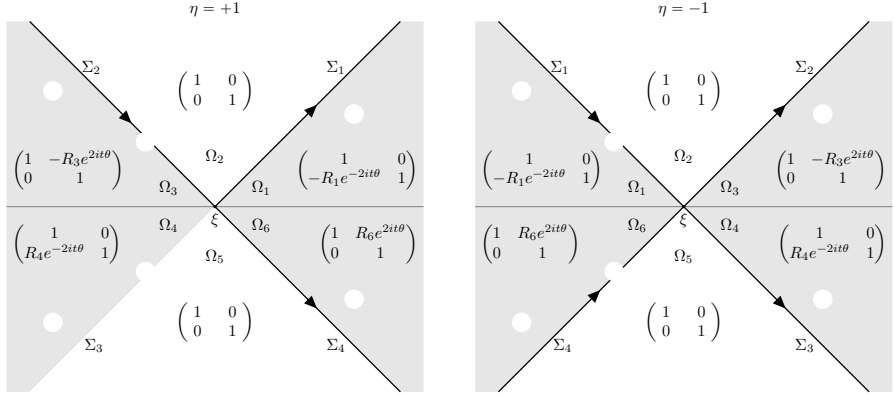

\hspace*{\stretch{1}}
\posDBARcontours
\hspace*{\stretch{1}}
\negDBARcontours
\hspace*{\stretch{1}}
\caption{Depicted here are the contour $\Sk{2} = \bigcup_{k=1}^4 \Sigma_{k}$ and regions $\Omega_k$ $k=1,\dots,6$ defining the transformation $\nk{2} = \nk{1} \mathcal{R}^{(2)}$. 
The labeling of the regions depends on $\sgnt$.
The non-analytic matrix $\mathcal{R}^{(2)}$ is given in each region $\Omega_k$. 
The support of  the $\bar{\partial}$-derivatives, $\Wk{2} = \dbar \mathcal{R}^{(2)}$ is shaded in gray.
\label{fig:n2def}
}
\end{figure}

The following lemma and its proof are almost identical to 
\cite[Proposition 2.1]{DM08}
or 
\cite[Lemma 4.1]{LPS16}.
It establishes the existence of the functions $R_k$ above and gives estimates that are needed to control the contribution of the solution of the $\dbar$-problem (Section~\ref{sec:dbar}) to the large time behavior of $q(x,t)$. To state it, we introduce the factors  
\begin{align*}
	& p_1(\lam) =   {-\eps \overline{\rho(\lam)}} ~,
	&& p_3(\lam) = \frac{\rho(\lam)}{1-\eps \lam |\rho(\lam)|^2}  ~,\\
	& p_6(\lam) = \rho(\lam) ~,
	&& p_4(\lam) =   { \frac{-\eps \overline{\rho(\lam)}}{1-\eps \lam |\rho(\lam)|^2} } ~.
\end{align*}
%%
%% PP introduced new notation to make formulas fit 
%%
We let
$$ P_k(\lam) = \lam^{m_k} p_k'(\Re \lam) $$
where $m_k=1$ for $k=1,4$ and $m_k=0$ for $k=3,6$.

\begin{lemma}\label{lem:extensions}
Suppose that $\rho \in H^{2,2}(\R)$ and that $c:= \inf_{\lam \in \R} (1 - \eps \lam | \rho(\lam)|^2) > 0$ strictly. Then there exist functions $R_k$ on $\Omega_k$, $k=1,3,4,6$ satisfying \eqref{R_k}, such that
\sidenote{24}
\sidenote{25}
\begin{gather*}
	\left| \dbar R_k \right| \lesssim 
	\begin{cases}
		{ \left| \dbar \indicator(\lam) \right| }+ 
		%%|   {\lam^{m_k}\,} p_k'(\Re \lam) | + \log |\lam - \xi|^{-1} 
		|P_k(\lam)| 
		%+ \log |\lam - \xi|^{-1} 
		\revisedtext{-\log|\lam -\xi|}
		& \revisedtext{\lam} \in \Omega_k, \quad |\revisedtext{\lam}-\xi| \leq 1 \\
		{ \left| \dbar \indicator(\lam) \right| }+
		%% |   {\lam^{m_k}\,} p_k'(\Re \lam) | 
		|P_k(\lam)|
		+ |\lam - \xi|^{-1} 
		& \revisedtext{\lam} \in \Omega_k, \quad |\revisedtext{\lam}-\xi| > 1 
	\end{cases} 
\shortintertext{and}
\revisedtext{
	\begin{aligned}	
	\dbar R_k(\lambda) &\equiv 0 
	  &&\text{if } \lam \in \Omega_2 \cup \Omega_5
	  \text{ or } \dist(\lam, \Lambda) \leq \poledist/3 \\
	\dbar R_k(\lam)  &= \bigO{\lam} 
	  &&\text{as } \lam \to 0 \in \Omega_1 \cup \Omega_4			
	\end{aligned}
}				
\end{gather*}
Here 
%% PP took out because of new definition for P_k above
%%{$m_k = 1$ for $k=1,4$ and $m_k=0$ for $k=3,6$ and} 
the implied constants are uniform for $\xi \in \R$ and $\rho$ in a fixed bounded subset of $H^{2,2}(\R)$ with $1- \eps \lam |\rho(\lam)|^2 \geq c > 0$ for a fixed constant $c$.
\end{lemma}

This lemma has the following immediate corollary:
\begin{corollary}\label{cor:R2.bd}
Let $\lam - \xi = u + i v$ with $u,v \in \R$. Then under the assumptions of Lemma~\ref{lem:extensions}  for $k=1,3,4,6$,
%%%%%%%%%%%%%%%%%%%%%%%%%%%%%
%%	
%%	PAP put $\lam \in \Omega_k$ here and removed from
%$	displayed formulas  to avoid overflow
%%
%%%%%%%%%%%%%%%%%%%%%%%%%%%%%
and $\lam \in \Omega_k$, we have
\begin{gather*}
	\left| \dbar \mathcal{R}^{(2)}(\lam;\xi) \right| 
	\lesssim 
	\begin{cases}
		\lp    
				\left| \dbar \indicator(\lam) \right| 
				%%%+ \left|   {\lam^{m_k}\,} p_k'(\Re \lam) \right| } 
				+ \left|P_k(\lam)\right| 
				%+ \log \frac{1}{|z-\xi|} 
				\revisedtext{-\log |z-\xi|}
		\rp 
		e^{-8t |u| |v|}
		& 
		%\lam \in \Omega_k,\ 
		|\lam - \xi| \leq 1 \\
		\\
		\lp   
				 \left| \dbar \indicator(\lam) \right| 
				%%+ \left|   {\lam^{m_k}\,} p_k'(\Re \lam) \right| } 
				+ \left|  P_k(\lam) \right| 
				%+ \frac{1}{\sqrt{1+|\lam - \xi|^2}}
				\revisedtext{\left( 1+ |\lam - \xi|^2\right)^{-1/2}}
		\rp 
		e^{-8t |u| |v|}
		& %\lam \in \Omega_k,\ 
		|\lam - \xi| > 1,
	\end{cases}
\shortintertext{and}
	\revisedtext{
	\begin{aligned}
		\dbar \mathcal{R}^{(2)}(\lam,\xi) 
				& \equiv 0 
				&&\text{if } \lam \in \Omega_2 \cup \Omega_5
				\text{ or } \dist(\lam, \Lambda) \leq \poledist/3, \\
		\dbar \mathcal{R}^{(2)}(\lam,\xi) 	
	  			& = \bigO{\lam}  
				&&\text{as } \lam \to 0 \in \Omega_1 \cup \Omega_4.
	\end{aligned}
	}
\end{gather*}
Here $m_k$ is as defined in Lemma~\ref{lem:extensions}, and all the implied constants are uniform for $\xi \in \R$ and $|t|>1$.
\end{corollary}

\revisedtext{
\begin{remark}
\sidenote{25}
The estimates of the $\dbar$-derivatives at the origin appearing above are used later in the proof of Proposition~\ref{prop:n3 at 0} which is needed to compute an asymptotic expansion of the inverse gauge transform $u(x,t) = \mathcal{G}^{-1}(q)(x,t)$.
\end{remark}
}
\begin{proof}[Proof of Lemma~\ref{lem:extensions}]
We give the construction for $R_1$. Define $f_1(\lam)$ on $\Omega_1$ by\sidenote{26}
\begin{gather*}
	f_1(\lam) =   {\xi} p_1(\xi) \delta_0(\xi, \sgnt)^{-2} 
	( \revisedtext{\eta \cdot (\lam-\xi)} )^{-2i \eta \kappa(\xi) \sig} \delta(\lam)^{2}
\shortintertext{and let}
	R_1(\lam) = \lb f_1(\lam) + \lp   {\lam} p_1(\Re \lam)  
	- f_1(\lam) \rp \sgnt \cos (2\phi) \rb 
	\delta(\lam)^{-2}(1- \indicator(\lam)), 
\end{gather*}
where $\phi = \arg(\lam - \xi)$. It is easy to see that $R_1$, as constructed, satisfies the boundary conditions in \eqref{R1} and that $\dbar R_1(\lam) = 0$ for $\dist(\lam, \Lambda) < \poledist/3$. Writing $\lam - \xi = r e^{i \phi}$ we have
\begin{gather*}
	\dbar = \frac{1}{2} \lp \pd{}{\Re \lam} + i \pd{}{\Im \lam} \rp 
	= \frac{e^{i\phi}}{2} \lp \pd{}{r} + \frac{i}{r} \pd{}{\phi} \rp,
\shortintertext{and}
%\begin{multlined}[.9\textwidth]
\begin{align*}
	\dbar R_1(\lam) &= 
	{-}\lb 
			f_1(\lam) +   
				{\sgnt} \lp   {\lam} p_1(\Re \lam) - f_1(\lam) \rp 
				\cos (2\phi) 
		\rb \
		\delta(\lam)^{-2} \dbar\indicator(\lam) \\
	& + \eta \lb 
					\frac{  {\lam}}{2} p_1'(\Re \lam) \cos(2\phi) 
					- \frac{i e^{i\phi}}{|z-\xi|} (p_1(\Re \lam) 
						- f_1(\lam)) \sin(2\phi) 
			\rb  \\
	& \times \delta^{-2}(\lam) (1-\indicator(\lam)).
\end{align*}
\end{gather*}
\sidenote{27}
Clearly, \sidenote{35}\revisedtext{$\dbar R_1(\lam) = \bigO{\lam}$} as $\lam \to 0$; it follows from Lemma~\ref{lem:T}\revisedtext{, (i) and (v) that}
%(ii. and v.) that
\[
\left| \dbar R_{  {1} } \right| \lesssim_\rho 
	\begin{cases}
		  { | \dbar \indicator(\lam) | } + |   {\lam} p_1'(\Re \lam) | 
		+ \log |\lam - \xi|^{-1}, 
		&  |z-\xi| \leq 1 \\
		  { | \dbar \indicator(\lam) | } + |   {\lam} p_1'(\Re \lam) | 
		+ |\lam - \xi|^{-1}, 
		&  |z-\xi| > 1 
	\end{cases} 
\]
where the implied constants depend on $\inf_\R (1-\eps \lam |\rho(\lam)|^2)$, $\oldnorm{\rho}_{H^{2,2}(\R)} $, and $\poles$. The constructions of $R_3,R_4$ and $R_6$ are similar.
\end{proof}

The new unknown $\nk{2}$ satisfies a mixed $\dbar$-RHP. We compute the new jumps on $\Sk{2}$ using the formula
\[
	\vk{2} = \nk[-]{1}^{-1}  \nk[+]{1}
	= \lp \mathcal{R}^{(2)}_-\rp^{-1} \vk{1} \mathcal{R}^{(2)}_+
\]
where the \sidenote{28}\revisedtext{  subscripts $+/-$} refer to the left/right side of the contour with respect to its orientation. Away from $\Sk{2}$, remembering that $\nk{1}$ is analytic in $\C \setminus (\R \cup \poles)$, we have
\[
	\dbar \nk{2} = \nk{1} \dbar \mathcal{R}^{(2)} = \nk{2} \lp \mathcal{R}^{(2)}\rp^{-1} \dbar \mathcal{R}^{(2)}
	= \nk{2} \dbar \mathcal{R}^{(2)}.
\]
where the last step follows from the nilpotency of $\dbar \mathcal{R}^{(2)}$.
 
%%%%%%%%%%%%%%%%%%%%%%%%%%%%%%%%%%%
%
%  DBAR-RHP problem for n2
%
%%%%%%%%%%%%%%%%%%%%%%%%%%%%%%%%%%%
\sidenote{29, 55}
\begin{DBARHP}\label{rhp.n2}
Find a row vector valued-function 
$$\nk{2}: \C \setminus (\Sk{2} \cup \poles) \to \C^2$$ with the following properties:
\begin{enumerate}
\item[(i)] $\nk{2}(\lam)$ has sectionally continuous first partial derivatives in $\C \setminus (\Sk{2} \cup \poles)$ and continuous boundary values $\nk[\pm]{2}(\lambda)$ on $\Sk{2}$. 
\item[(ii)] $\nk{2}(\lam) = (1,0) + \bigo{\lam^{-1}}$ as $ \lam \to \infty$.
\item[(iii)] For $\lam \in \Sk{2}$, the boundary values satisfy the jump relation 
%% PAP changed to displayed equation for space
	$$\nk[+]{2}(\lam) = \nk[-]{2}(\lam) \vk{2}(\lam),$$ where
\sidenote{30}
\begin{equation}\label{V2}
	\begin{gathered}
		\vk{2}(\lam) = I + (1-\indicator(\lambda)) h(\lam), \\
		\\
		h(\lam) = \begin{cases}
			\tril[0]{ -\eps \xi \bar{\rho}(\xi) \delta_0(\xi,\sgnt)^{-2} 
			(\revisedtext{\sgnt \cdot}(\lam-\xi))^{-2i \sgnt \kappa(\xi)} e^{2it \theta} } 
		 		& \revisedtext{\lam} \in  \Sigma_1 \bigskip \\
    		\triu[0]{ \frac{\rho(\xi) \delta_0(\xi,\sgnt)^{2} }{1- \eps \xi |\rho(\xi)|^2}  
    		(\revisedtext{\sgnt \cdot}(\revisedtext{\lam}-\xi))^{2i \sgnt \kappa(\xi)} e^{-2it \theta} } 
    			& \revisedtext{\lam} \in  \Sigma_2 \bigskip\\
    		\tril[0]{ \frac{-\eps \xi \bar{\rho}(\xi) \delta_0^{-2}(\xi,\sgnt) } {1- \eps \xi |\rho(\xi)|^2} 
    		(\revisedtext{\sgnt \cdot}(\revisedtext{\lam}-\xi))^{-2i \sgnt \kappa(\xi)} e^{2it \theta} } 
    			& \revisedtext{\lam} \in  \Sigma_3 \bigskip \\
    		\triu[0]{ \rho(\xi) \delta_0(\xi,\sgnt)^{2}  
    		(\revisedtext{\sgnt \cdot}(\revisedtext{\lam}-\xi))^{2i  \sgnt \kappa(\xi)} e^{-2it \theta} } 
				& \revisedtext{\lam} \in  \Sigma_4.
		\end{cases}
	\end{gathered}
\end{equation}

\item[(iv)] For $\lam \in \C \setminus \Sk{2}$ we have
\[
	\dbar \nk{2} (\lam) = \nk{2}(\lam) \dbar \mathcal{R}^{(2)}(\lam)
\]	
where
\begin{equation*}
%\label{W2}
	\dbar \mathcal{R}^{(2)}(\lam) = \begin{cases}
		\tril[0]{-\dbar R_1(\lam) e^{-2it\theta} } & \lam \in \Omega_1 \medskip \\
		\triu[0]{-\dbar R_3(\lam) e^{2it\theta} } & \lam \in \Omega_3 \medskip \\
		\tril[0]{\dbar R_4(\lam) e^{-2it\theta} } & \lam \in \Omega_4 \medskip \\
		\triu[0]{\dbar R_6(\lam) e^{2it\theta} } & \lam \in \Omega_6 \medskip \\
		\qquad \quad \mathbf{0} & \textrm{elsewhere}.
	\end{cases}
\end{equation*}	

\item[(v)] $\nk{2}(\lam)$ has simple poles at each point $p \in \poles$. 
	\begin{equation}\label{n2 residue}
		\res_{\lam = p} \nk{2}(\lam) 
		= \lim_{\lam \to p} \nk{2}(\lam) \vk{1}(p) 
	\end{equation}
	where $\vk{1}(p)$ is as given in \eqref{n1 residue matrices}.
\end{enumerate}
\end{DBARHP}
%%%%%%%%%%%%%%%%%%%%%%%%%%%%%%%%%%%%%%%
%
% END of DBAR-RHP PROBLEM
%
%%%%%%%%%%%%%%%%%%%%%%%%%%%%%%%%%%%%%%%

%\input{model-dbar}   		%% 	Model problems and remaining d-bar problem

%%%%%%%%%%%%%%%%%%%%%%%%%%%%%%%%%%%%%%%
%
%		MODEL RHP PROBLEMS
%		File dnls-imrn-arXiv-model.tex
%		Created 4.7.2017
%
%%%%%%%%%%%%%%%%%%%%%%%%%%%%%%%%%%%%%%%

\section{ Decomposition into a RH model problem and a pure \texorpdfstring{$\dbar$}{DBAR}-problem}
\label{sec:models}
The next step in our analysis is to construct the solution $\Nrhp$ of a \textit{matrix-valued} Riemann-Hilbert problem such that the transformation
\begin{equation}\label{n3 def}
	\nk{3}(\lam) = \nk{2}(\lam) (\Nrhp(\lam))^{-1}
\end{equation}
results in a pure $\dbar$-problem, \ie, the new unknown $\nk{3}$ is continuous; it has no jumps or poles. We arrive at the problem for $\Nrhp$ by essentially ignoring the $\dbar$ component of \revisedtext{$\dbar$/Riemann-Hilbert Problem~\ref{rhp.n2}.}

%%%%%%%%%%%%%%%%%%%%%%%%%%%%%%%%%%%
%
%  RHP problem for Nrhp
%
%%%%%%%%%%%%%%%%%%%%%%%%%%%%%%%%%%%
\begin{RHP}\label{rhp.Nrhp}
Find a $2\times2$ matrix valued function 
$\Nrhp: \C \setminus (\Sk{2} \cup \poles) \to SL_2(\C)$ with the following properties:
\begin{enumerate}
\item[(i)] $\Nrhp$ satisfies the symmetry relations 
	$ \Nrhp[22](\lam) = \overline{\Nrhp[11](\lambar)}$
	and 
	$\Nrhp[21](\lam) = \eps \lam \overline{\Nrhp[12](\lambar)}$.
%	\[	\Nrhp(\lam) = \begin{pmatrix}
%			\overline{\Nrhp[22](\lambar)} & \eps \lam^{-1} \overline{\Nrhp[21](\lambar)} \\
%			\eps \lam \overline{\Nrhp[12](\lambar)} &\overline{\Nrhp[11](\lambar)}
%		\end{pmatrix} 
%	\]
%	
\item[(ii)] $\Nrhp(\lam) = \stril{\alpha} + \bigo{\lam^{-1}}$ as $ \lam \to \infty$, for 
a constant $\alpha$ determined by the symmetry condition above.
\item[(iii)] For $\lam \in \Sk{2}$, the boundary values satisfy the jump relation 
	$\Nrhp[+](\lam) = \Nrhp[-](\lam) \vk{2}(\lam)$, where $\vk{2}$ is given by \eqref{V2}.
\item[(iv)] $\Nrhp(\lam)$ has simple poles at each $p \in \poles$ with
%% PAP changed to displayed equation
	$$
	\res_{\lam = p} \Nrhp(\lam) 
	= \lim_{\lam \to p} \Nrhp(\lam) \vk{\mathsc{rhp}}(p).
	$$
	For each $\lam_k \in \poles_+$ 
	\begin{equation}\label{Nrhp residue}
	\begin{aligned}
	\vk{\mathsc{rhp}}(\lam_k) &= 
	  \tril[0]{\lam_k C_k \delta^{-2}(\lam_k) e^{-2it\theta} },
	\\[5pt]
	\vk{\mathsc{rhp}}(\lambar_k) &=
	  \triu[0]{\eps \overline{C_k} \delta^{2}(\lambar_k) e^{2it\theta} }.
	  \end{aligned}
	\end{equation}
\end{enumerate}
\end{RHP}

%%%%%%%%%%%%%%%%%%%%%%%%%%%%%%%%%%%%%%%%%%
%
% END of RHP for Nrhp
%
%%%%%%%%%%%%%%%%%%%%%%%%%%%%%%%%%%%%%%%%%%

The next step will be solving the following $\dbar$-problem.

\begin{DBAR}
\label{dbar.n3}
Given $x,t \in \R$ and $\rho \in H^{2,2}(\R)$ with $\inf_{\R} (1 - \eps \lam|\rho(\lam)|^2) > 0$, find a continuous, row vector-valued function $\nk{3}(\lam)$ with the following properties: 
\sidenote{31}
\begin{enumerate}
	\item[(i)] $\nk{3}(\lam) \to \begin{pmatrix} 1 & 0 \end{pmatrix}$
		as $|\lam | \to \infty$.
	\item[(ii)] $\dbar \nk{3}(\lam) = \nk{3}(\lam) \Wk{3}(\lam)$, where
	\begin{equation}\label{W3def}
		\Wk{3}(\lam) = \Nrhp (\lam)  \dbar \mathcal{R}^{(2)}(\lam) 
		%(\Nrhp)^{-1}(\lam).
		\revisedtext{\left( \Nrhp(\lam) \right)^{-1}}.
	\end{equation}	
\end{enumerate}
\end{DBAR}

\begin{lemma}
%\label{lem:n2.to.n3} 
Suppose that $\nk{2}$ solves $\dbar$-RHP~\ref{rhp.n2}. Given a solution $\Nrhp$ of RHP~\ref{rhp.Nrhp},
the function $\nk{3}$ defined by \eqref{n3 def} satisfies $\dbar$-Problem~\ref{dbar.n3} 
\sidenote{32}
%below.
\revisedtext{above.}
\end{lemma}

\begin{proof}
Given solutions $\nk{2}$ and $\Nrhp$ of $\dbar$-RHP~\ref{rhp.n2} and RHP~\ref{rhp.Nrhp} respectively, the normalization condition for $\nk{3}$ is immediate. As $\Nrhp$ is holomorphic in $\C \backslash \Sk{2}$, the $\dbar$-derivative of $\nk{3}$ satisfies
\sidenote{31}
\begin{equation*}
%\label{W3}
	\dbar \nk{3} = \dbar \nk{2} \revisedtext{\left(\Nrhp\right)^{-1}} = \lb \nk{2} 
%	\Wk{2} 
\bar{\partial}\mathcal{R}^{(2)} 
	\rb \revisedtext{\left(\Nrhp\right)^{-1}} 
	= \nk{3} \lb \Nrhp %\Wk{2} 
 \bar{\partial}\mathcal{R}^{(2)} 
	\revisedtext{\left(\Nrhp\right)^{-1}} \rb.
\end{equation*}
For $\lam \in \Sk{2}$ the computation 
\begin{equation*}
	\begin{aligned}[t]
	\nk[+]{3}(\lam) 
		&= \nk[-]{3}(\lam) \Nrhp[-](\lam) \vk{2}(\lam) 
		%\Nrhp[+](\lam)^{-1}  
		\revisedtext{\left(\Nrhp[+](\lam) \right)^{-1}}
		 \\
		&= \nk[-]{3}(\lam) \Nrhp[-](\lam) \vk{2}(\lam) \lb \vk{2}(\lam)^{-1}
		% \Nrhp[-](\lam)^{-1} 
		\revisedtext{\left(\Nrhp[-](\lam)\right)^{-1}}
		\rb 
		= \nk[-]{3}(\lam)
	\end{aligned}
\end{equation*}
shows that $\nk{3}$ has no jumps and is everywhere continuous. Another direct calculation shows that 
$\nk{3}$ has removable singularities at each pole in $\poles$: for instance if $p \in \poles$ and $\vk{1}(p)$ is the nilpotent residue matrix in \eqref{n1 residue matrices} then 
using \eqref{n2 residue} and \eqref{Nrhp residue} we have the Laurent expansions in $z=\lam-p$ 
\begin{align*}
	\nk{2}(\lam) = a(p) \lb \frac{ \vk{1}(p)}{z} + I \rb + \bigo{z} \qquad % \lam_k)}	\\
	\Nrhp(\lam) = A(p) \lb \frac{ \vk{1}(p)}{z} + I \rb + \bigo{z} % \lam_k)}
\end{align*}
where $a(p)$ and $A(p)$ are the constant row vector and matrix in their respective expansions. As $\Nrhp \in SL_2(\C)$, $(\Nrhp)^{-1} = \sigma_2 (\Nrhp)^\intercal \sigma_2$, and it follows that 
%\[
%	(\Nrhp)^{-1} = \sigma_2 (\Nrhp)^\intercal \sigma_2
%	= \lb \frac{ -\vk{1}(p)}{\lam - p} + I \rb \sigma_2 A(p)^\intercal \sigma_2 + \bigo{(\lam -   {  p} )}.  % \lam_k)}.
%\]
%It follows that 
\begin{align*}
	\nk{3}(\lam) 	&= \nk{2}(\lam) \Nrhp(\lam)^{-1} \\
						&= \left\{ a(p) \lb \frac{ \vk{1}(p)}{z} + I \rb + \bigo{z} \right\}  %\\ 
							%	&\quad \times 
								\left\{ \lb \frac{ -\vk{1}(p)}{z} + I \rb 
								\sigma_2 A(p)^\intercal \sigma_2 + \bigo{z}  \right\} \\
						&=\revisedtext{ \bigo{1},}
\end{align*}
\sidenote{33}
where the last equality follows from the fact that $\vk{1}(p)^2 = 0$.
\end{proof}

The remainder of this section is dedicated to proving the following proposition
\begin{proposition}
%\label{prop:Nrhp.est}
Given $\rho \in H^{2,2}(\R)$ with $c:= \inf_{\lam \in \R} (1 - \eps \lam | \rho(\lam)|^2) >0$ strictly, then there exists $T> 0$ such that for $|t| > T$, there exists a unique solution $\Nrhp(\lam)$ of RHP~\ref{rhp.Nrhp} satisfying 
\[
	\oldnorm{\Nrhp(\lam)}_{L^\infty(\C \backslash B_\poles)} \lesssim 1
\] 
where $B_\poles$ is any open neighborhood of $\poles$ and the implied constants are uniform in $x$ and $|t|>T$; they depend on $B_\poles$ and $\rho$.
\end{proposition}

To prove the existence of $\Nrhp$, we will first construct two explicit models: one which exactly solves the pure soliton problem obtained by ignoring the jump conditions, and a second which uses parabolic cylinder functions to build a matrix whose jumps exactly match those of $\nk{2}$ in a neighborhood of the critical point $\xi$. Using our models we prove that $\Nrhp$ exists and extract its behavior for large $t$.

\subsection{The outer model: the soliton component}
\label{sec:outer model}

The matrix $\Nrhp$ is meromorphic away from the contour $\Sk{2}$ on which its boundary values satisfy the jump relation $\Nrhp[+](\lam) = \Nrhp[-](\lam) \vk{2}(\lam)$. It is clear from \eqref{V2} that 
\begin{equation}\label{V2 bound outside} 
	\left| \vk{2}(\lam) - I \right| \lesssim 
	e^{-2\sqrt{2} t |\lam-\xi|^2},
\end{equation}
where the implied constant depends upon $\poledist$ and $\inf_{\lam \in \R}(1-\eps \lam |\rho(\lam)|^2)$. It follows that outside a fixed neighborhood of $\xi$ we introduce only exponentially small error (in  $t$) by completely ignoring the jump condition on $\Nrhp$. This results in the following outer model problem

\begin{RHP}
\sidenote{34}
\label{outermodel} 
\revisedtext{For any fixed $x,t \in \R$,} 
%let $\Nout: (\C\setminus\Lambda) \to SL_2(\C)$ be 
%an analytic function 
\revisedtext{find an analytic function $\Nout: (\C\setminus\Lambda) \to SL_2(\C)$}
such that
\begin{enumerate}
	\item[(i)] $\Nout$ satisfies the symmetry relations
	$\Nout[22](\lam) = \overline{\Nout[11](\lambar)}$
	and 
	$\Nout[21](\lam) = \eps \lam \overline{\Nout[12](\lambar)}$.
	\item[(ii)] $\Nout(\lam) = \stril{\alpha} + \bigo{\lam^{-1}}$ as $\lam \to \infty$, where $\alpha$ is determined via the symmetry condition.
	\item[(iii)] $\Nout$ has a simple pole at each point in $\poles$ satisfying the residue relations in \eqref{Nrhp residue} with $\Nout$ replacing $\Nrhp$.
	
\end{enumerate}
\end{RHP}

The essential fact we need concerning $\Nout$ is as follows.

\begin{proposition}
\label{outer.soliton}
The unique solution $\Nout$ of 
%Problem~\ref{outermodel} is
\revisedtext{RHP~\ref{outermodel} is given by}
$$
	\Nout(\lam) =  % \Nsol{\emptyset}(\lam \, \vert\, \calD_\xi),
	\ {\mathcal N}^{\mathrm sol}(\lam \, \vert\, \calD_\xi), 
$$
where % $\Nsol{\emptyset}$
{${\mathcal N}^{\mathrm sol} $}
 is the solution of RHP~\ref{RHP2} corresponding to the reflectionless scattering data $\mathcal{D}_\xi = \{ (\lam_k, 
{    \widetilde{C_k}}) \}_{k=1}^N$ generated by the $N$-soliton solution $\qsol(x,t;\mathcal{D}_\xi) $ of \eqref{DNLS2}. Here, $\{\lam_k\}_{k=1}^N$ are the points generated by our original initial data \eqref{data} and the modified connection coefficients are given by 
\[
	%\mathcal{C}_k
 {  \widetilde{C_k}	} = C_k \exp\lp \frac{i}{\pi} \int_{\negint}
	\log(1 - \eps z |\rho(z)|^2) \frac{dz}{z - \lam_k} \rp.
\]
%% PAP changed equation to two lines
As $\lam \to \infty$, $\Nout$ admits the expansion
\begin{equation*}
%\label{Nsol.asymp}
	\Nout(\lam) = \tril{ -\frac{\eps \overline{\qsol}(x,t;\mathcal{D}_\xi)}{2i} } + \frac{ \Nout[1] }{\lam} + \bigo{\frac{1}{\lam^2}},
\end{equation*}
where
$$
	2i (\Nout[1] )_{12} = \qsol(x,t;\mathcal{D}_\xi).
$$
\end{proposition}
\begin{proof}
Using formula \eqref{T} for $\delta(\lam)$, 
%Problem~\ref{outermodel}
\sidenote{34}
\revisedtext{RHP}~\ref{outermodel}
 is identical to Problem~\ref{outmodel2a} with $\calD = \calD_\xi$ and $\Delta = \emptyset$. Uniqueness follows from Lemma~\ref{lem:sol.bound}
 \sidenote{35} 
 \revisedtext{in Appendix B.}
The expansion for large $\lam$ follows from the fact that $\Nout$ is meromorphic, and the given off-diagonal coefficients of the leading and first moment terms follow from \eqref{q.lam} and the symmetry condition in %Problem~\ref{outermodel}.
\sidenote{34}
\revisedtext{RHP~\ref{outermodel}.}
\end{proof}

\sidenote{35}
Lemma \ref{lem:sol.bound} of Appendix~\ref{app:solitons} provides the following useful facts. 

\begin{lemma}\label{lem:outer.bound}
Given $\rho \in H^{2,2}(\R)$ and $\{ (\lam_k, C_k) \}_{k=1}^N \subset \C^+ \times \C^\times$ for RHP~\ref{RHP2}, the solution $\Nout$ of 
%Problem~\ref{outermodel} 
\sidenote{34}
\revisedtext{RHP~\ref{outermodel}}
satisfies
\[
	\oldnorm{(\Nout)^{\pm 1}}_{L^\infty(\C \setminus \mathcal{B}_\poles)} = \bigo{1}
\]
where $\mathcal{B}_\poles$ is any open neighborhood of the points in $\Lambda$, and the implied constant is independent of $(x,t) \in \R^2$ and depends on $\rho$ through its $H^{2,2}(\R)$ norm. 
\end{lemma}  

\sidenote{36}
\revisedtext{The dependency of $\Nout$ on $\rho$ appears through the  modified connection coefficients defined in Proposition \ref{outer.soliton} above.}

\subsection{Local model at the {  stationary phase } point}
\label{sec:local model}
In any neighborhood of the critical point $\lam = \xi$ the bound \eqref{V2 bound outside} does not give a uniformly small estimate of 
%the jump $\vk{2}$ 
\sidenote{37}\revisedtext{the difference $\vk{2}-I$}
for large times. It follows that our outer model, which replaced $\vk{2}$ with identity, is not a good approximation of $\Nrhp$ in a neighborhood of $\xi$. We require a new model, $\NPC$, which is an accurate approximation inside a small--but fixed with respect to $|t|$--neighborhood of $\lam = \xi$. Let 
\begin{equation}\label{xi disk}
	\Uxi = \{ \lam \in \C \,:\, |\lam - \xi| \leq \poledist/3 \} ,
\end{equation}
where the radius $\poledist/3$ is chosen such that $(1-\indicator(\lam))\equiv 1$ for $\lam \in \Uxi$; this has the effect of making the jump matrix $\vk{2}$, \eqref{V2}, constant along $\Sigma_k \cap \Uxi$, $k=1,\dots,4$.

Define the time-scaled local coordinate
\begin{gather}\label{zeta}
	\zeta(\lam) = |8t|^{1/2} (\lam - \xi).
\end{gather}
Under this change of variables we have the identifications
\sidenote{38}
\[
	e^{2it \theta}  = {  e^{- i\eta \zeta^2/2}} e^{4it\xi^2},
	\qquad
	(\revisedtext{\eta \cdot} (\lam - \xi))^{2i\eta \kappa(\xi)} 
	= (\eta \zeta)^{2i \eta \kappa(\xi)} e^{-i \eta \kappa(\xi) \log |8 t|}.
\]
Also set,
\begin{equation}
\begin{gathered}\label{PCconst}
	r_\xi := \rho(\xi) \delta_0(\xi,\sgnt)^2 e^{-i \sgnt \kappa(\xi) \log|8t|} e^{4it\xi^2}, 
	\qquad
	s_\xi := - \eps \xi \overline{r_\xi} \\
 	1+r_\xi s_\xi = 1 - \eps \xi | \rho(\xi)|^2
%	s_\xi = -\eps \xi \bar{\rho}(\xi) \delta_0(\xi,\sgnt)^{-2} e^{i \sgnt \kappa(\xi) \log|8t|} e^{-4it\xi^2},
\end{gathered}
\end{equation}

\begin{figure}[htbp]
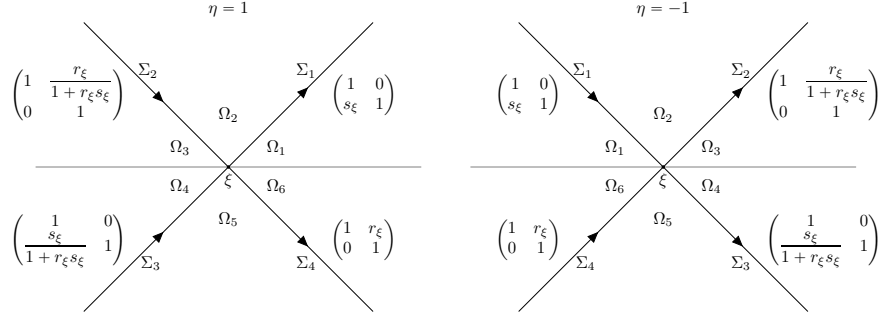

	\centering
	\hspace*{\stretch{1}}
	\FigPCjumps 
	\hspace*{\stretch{1}}
	\FigPCjumpsB
	\hspace*{\stretch{1}}
	\caption{
	The system of contours for the local model problem near $\lam = \xi$. 
	The model jumps are $\vk{\textsc{PC}} = 
	(\sgnt \zeta)^{i \sgnt \kappa(\xi)\sig} e^{-i \sgnt \zeta^2 \sig/4} V 
	(\sgnt \zeta)^{-i \sgnt \kappa(\xi)\sig} e^{i \sgnt \zeta^2 \sig/4}$ 
	where $V$ is given above in terms of the local variable $\zeta$ defined by \eqref{zeta}.
	\label{fig:PCjumps}
	}
\end{figure}

Using the notation just introduced, and extending the constant jumps of $\vk{2}\Big|_{\lam \in \Uxi}$ to infinity along each of the four rays $\Sigma_k,\ k=1,\dots,4$, (see Figure~\ref{fig:PCjumps}) our local model $\NPC$ satisfies  

\begin{RHP}\label{rhp.localmodel}
	Find a $2\times2$ matrix-valued function $\NPC(\lam) = \NPC(\lam;\xi,\sgnt)$, analytic in $\C \setminus \Sk{2}$ with the following properties:
	\begin{enumerate}
		\sidenote{39}
		\item[(i)] $\NPC(\lam;\xi,\sgnt) = 
		%\sdiag{1}{1}
		\revisedtext{I} 
		+ \bigo{\lam^{-1}}$ as $|\lam| \to \infty$.
		\item[(ii)] $\NPC(\lam;\xi,\sgnt)$ has continuous boundary values 
		$\NPC[\pm](\lam; \xi,\sgnt)$ on $\Sk{2}$ which satisfy the jump relation 
		$\NPC[+] = \NPC[-] \vk{\textsc{pc}}$, where
		\sidenote{40}
		\begin{equation} \label{NPC jump}
			\vk{\textsc{pc}}(\lam) = 
			\begin{cases}
				\tril{ s_\xi(\eta \zeta)^{-2i \eta \kappa(\xi)} e^{i \eta \zeta^2/2} } 
					& \revisedtext{\lam} \in \Sigma_1, \medskip \\
				\triu{\frac{r_\xi}{1+ r_\xi s_\xi} 
				(\eta \zeta)^{2i \eta \kappa(\xi)} e^{-i \eta \zeta^2/2} } 
					& \revisedtext{\lam} \in \Sigma_2, \medskip \\
				\tril{\frac{s_\xi}{1+ r_\xi s_\xi} 
				(\eta \zeta)^{-2i \eta \kappa(\xi)} e^{i \eta \zeta^2/2} } 
					& \revisedtext{\lam} \in \Sigma_3, \medskip \\
				\triu{  r_\xi (\eta \zeta)^{2i \eta \kappa(\xi)}  e^{-i \eta \zeta^2/2} } 
					& \revisedtext{\lam} \in \Sigma_4. 
			\end{cases} 
		\end{equation}
	\end{enumerate}
\end{RHP}
\begin{remark}
RHP~\ref{rhp.localmodel} does not possess the symmetry condition shared by RHP~\ref{rhp.Nrhp} and 
%Problem~\ref{outermodel}. 
\sidenote{34}
\revisedtext{RHP~\ref{outermodel}.}
This is because it is a local model and will only be used for bounded values of $\lam$. The normalization is chosen such that the residual error $\error$ defined by \eqref{error def} \sidenote{41}\revisedtext{below} has a near identity jump on the shared boundary between the local and outer models. 
\end{remark}

\revisedtext{
\begin{remark}
\sidenote{42}
The jump matrices for RHP \ref{rhp.localmodel} agree identically with those of RHP \ref{rhp.Nrhp}, and can be solved by special functions, owing to the special choice of $\dbar$-extension of the matrix factors in \eqref{R_k}. By making this choice of extension, one insures that the jumps may be solved exactly `locally,' i.e., in a neighborhood of the critial point, and produces remainders which can be effectively estimated. Even in the original application of the $\dbar$-method to soliton-free NLS, this method produces an improved remainder estimate for large-time asymptotics (see \cite{DM08}). 
\end{remark}
}

This type of model problem is typical in integrable systems whenever there is a phase function, here $\theta$, which has a quadratic critical point along the real line. 
The solution in each of these cases is found by a further reduction of RHP~\ref{rhp.localmodel} to a problem with constant jumps (at the price of nontrivial behavior at infinity) whose solution {  satisfies a} 
%can be expressed in terms of a certain 
differential equation,
 which can be solved using parabolic cylinder functions, $D_a(z)$, whose properties are tabulated in \cite[Chapter 12]{DLMF}. The precise details of the construction for DNLS, which differ only slightly from the construction for KdV or NLS 
{(see Deift-Zhou \cite{DZ93} and Its \cite{Its81})} can be found in \cite{LPS16}; here we give only the necessary details.

\begin{proposition}\label{prop:PCmodel.est}
Fix $\xi$ and let $\kappa = \kappa(\xi)$ be as given in \eqref{T}. 
Then for any choice of constants $r_\xi, s_\xi$ in \eqref{PCconst} such that $1 + r_\xi s_\xi = e^{-2\pi \kappa} \neq 0$, the solution $\NPC(\lam;\xi,\sgnt)$ of RHP~\ref{rhp.localmodel} is given by  
\begin{gather}
\label{local model soln}
	\begin{aligned}
		\NPC(\lam;\xi, +) &= F(\zeta(\lam); s_\xi, r_\xi)  \\
		\NPC(\lam;\xi, -) &= \sigma_2 F(-\zeta(\lam); r_\xi, s_\xi) \sigma_2 
	\end{aligned}
\end{gather}
%\shortintertext{where}	
where
\sidenote{43}
\begin{gather}
\nonumber	
	F(\zeta; s,r) := \Phi_{s,r}(\zeta) \mathcal{P}_{s,r}(\zeta) 
	\zeta^{-i \kappa \sig} e^{i\zeta^2\sig/4} \\[5pt]
\nonumber
	\mathcal{P}_{s,r}(\zeta) = 
	\left\{ 
		\begin{array}{c@{\quad}l@{\hspace{2em} }c@{\quad}l}
			\tril{s_\xi} & \arg \zeta \in \lp 0, \frac{\pi}{4} \rp, &
			\triu{\frac{r_\xi}{1+r_\xi s_\xi}} & \arg \zeta \in \lp \frac{3\pi}{4},\pi \rp, 
		%\smallskip \\ 
		\medskip \\
			\triu{-r_\xi} & \arg \zeta \in \lp -\frac{\pi}{4}, 0 \rp,  &
			\tril{\frac{-s_\xi}{1+r_\xi s_\xi}} & \arg \zeta \in \lp -\pi, -\frac{3\pi}{4} \rp, \\ \\
		\multicolumn{4}{c}{
		%\diag{1}{1} 
		\revisedtext{I}
		\quad |\arg \zeta| \in \lp \frac{\pi}{4},\frac{3\pi}{4} \rp.} 
		\end{array}
	\right.
\end{gather}
and, writing $\kappa$ for $\kappa(\xi)$,
%%%%%%%%%%%%%%%%%%%%%%%%%%%%%
%%
%%		PAP eliminated (s,r) arguments of \beta_{ij}
%%		in a desparate attempt to sleep in the Procrustean
%%		bed of CMP
%%
%%		Also added a bit of space for matrices
%%
%%%%%%%%%%%%%%%%%%%%%%%%%%%%%
\begin{gather*}
%\nonumber	
%\label{SimpliPhi+}
 	\Phi_{s,r} (\zeta)= 
	\begin{pmatrix}
 		{e^{-\frac{3\pi}{4}\kappa} D_{i\kappa}(\zeta e^{-3i\pi/4})} &
		-i \beta_{12}%(s,r)
			e^{\frac{\pi}{4}(\kappa-i)} 
		  D_{-i\kappa-1}(\zeta e^{-\pi i/4}) \\[5pt]
		i \beta_{21}%(s,r) 
			e^{-\frac{3\pi}{4}(\kappa+i)}
		  D_{i\kappa -1}(\zeta e^{-3i\pi/4}) &
		e^{\pi\kappa/4}D_{-i\kappa}(\zeta e^{-i\pi/4})
	\end{pmatrix}
\shortintertext{for $\Im(\zeta)>0$, and for $\Im(\zeta)<0$}
%\nonumber
%\label{SimpliPhi-}
 	\Phi_{s,r} (\zeta)= 
	\begin{pmatrix}
		e^{\pi\kappa/4}D_{i\kappa}(\zeta e^{\pi i/4}) &
		-i \beta_{12}%(s,r) 
			e^{-\frac{3\pi}{4}(\kappa-i)}
				D_{-i\kappa-1}(\zeta e^{3i\pi/4}) \\[5pt]
		i \beta_{21}%(s,r) 
			e^{\frac{\pi}{4}(\kappa+i)}
				D_{i\kappa-1}(\zeta e^{\pi i/4}) &
		e^{-3\pi \kappa/4}D_{-i\kappa}(\zeta e^{3i\pi/4})
	\end{pmatrix}.
\end{gather*}
Here
\begin{gather}
\label{betas}
\beta_{12} :=	\beta_{12}(s,r) = \frac{ \sqrt{2\pi}e^{-\pi \kappa/2} e^{i \pi/4} }{s \Gamma(-i\kappa)}
	\qquad
\beta_{21} :=	\beta_{21}(s,r)
		% = \frac{\kappa}{\beta_{12}(s,r)} 
		= \frac{ -\sqrt{2\pi}e^{-\pi \kappa/2} e^{-i \pi/4} }{r \Gamma(i\kappa)}.
\shortintertext{As $\zeta \to \infty$}
\nonumber
	F(\zeta; s,r) = I + \frac{1}{\zeta} \offdiag{ -i\beta_{12}(s,r)}{i\beta_{21}(s,r)}
	+ \bigo{\zeta^{-2}}.
\end{gather}
\end{proposition}

The essential property of $\NPC$ that we will need later is the asymptotic expansion for large $\zeta$. 
Using \eqref{local model soln},  we have 

\begin{equation} \label{local model expand}
	\NPC(\lam;\xi,\sgnt) = 
	 I + \frac{|8t|^{-1/2}}{\lam - \xi} A(\xi,\sgnt)
	 	+ \bigo{t^{-1}}, \\
	\qquad \lam \in \partial \Uxi,
\end{equation}
where
\begin{equation}\label{Axi}
	A(\xi,\sgnt) = \offdiag{ -iA_{12}(\xi, \sgnt) }{iA_{21}(\xi, \sgnt)}
\end{equation}
{  with 
\begin{align*}
&A_{12} (\xi, +) = \beta_{12} (s_\xi, r_\xi ), \   \  A_{21} (\xi, +) = \beta_{21} (s_\xi, r_\xi ) \nonumber \\
&A_{12} (\xi, -) = -\beta_{21}(r_\xi, s_\xi),  \  A_{21} (\xi, -) = -\beta_{12}(r_\xi, s_\xi)
\end{align*}
satisfies}

\begin{equation}\label{mod beta}
    	|A_{12}(\xi,\sgnt)|^2 =  \frac{\kappa(\xi)}{\xi} 
		\qquad \qquad 
		A_{21}(\xi, \sgnt) = \eps \xi \overline{A_{12}(\xi,\sgnt)}
\end{equation}
\begin{subequations}\label{beta arg}
\begin{align*}
	&\begin{multlined}[.9\textwidth]
	\arg A_{12}(\xi, +) = 
	\frac{\pi}{4} + \arg \Gamma(i \kappa(\xi)) - \arg ( -\eps \xi \overline{\rho(\xi)} ) \\
	+ \frac{1}{\pi} \int_{-\infty}^\xi \log|\xi - \lam|\, \mathrm{d}_\lam \log(1-\eps \lam |\rho(\lam)|^2) 
	- \kappa(\xi) \log|8t| + 4t \xi^2 
	\end{multlined} \\
	&\begin{multlined}[.9\textwidth]
	\arg A_{12}(\xi, -) = 
	\frac{\pi}{4} - \arg \Gamma(i \kappa(\xi)) - \arg (-\eps \xi \overline{\rho(\xi)} ) \\
	+ \frac{1}{\pi} \int_{\xi}^\infty \log|\xi - \lam|\, \mathrm{d}_\lam \log(1-\eps \lam |\rho(\lam)|^2) 
	+ \kappa(\xi) \log|8t| + 4t \xi^2 .
	\end{multlined} 
\end{align*}	
\end{subequations}

For $\eta = +1$, the first line of \eqref{local model expand} and \eqref{mod beta} are proved in \cite{LPS16}; the results for $\eta = -1$ then follow from \eqref{local model soln}. The second equality in \eqref{mod beta} is a consequence of the fact that $\beta_{12} \beta_{21} = \kappa$. Equations \eqref{beta arg} follow simply from \eqref{betas} and \eqref{PCconst}
where we use \eqref{delta0 arg} and integration by parts to express the integral terms.

\sidenote{44}\revisedtext{Later, we will compute an asymptotic expansion for the solution $u(x,t)$ of \eqref{DNLS1}-\eqref{data1} via the inverse gauge transformation $u(x,t) = \mathcal{G}^{-1}(q)(x,t)$. This can by computed in terms of $1/[n_{11}(0;x,t)]^2$, the solution of RHP~\ref{RHP2.row} (cf. Proposition~\ref{prop:u.gauge.expansion} below and \eqref{n11.at.0.a} in particular). 
When $\xi$ is near zero (specifically, when $0 \in \mathcal{U}_{\xi}$) this computation involves the value of the model problem at $\lambda=0$, which from \eqref{zeta} corresponds to $\zeta(0) = -|8t|^{1/2} \xi$ in the model plane}.
Note that, though \eqref{local model soln} is piecewise defined across the real axis, $\NPC$ does not have a jump across the real axis (cf. \eqref{NPC jump}); in the formulas below we have chosen the components of $\Phi_{s,r}$ for which right multiplication by $\mathcal{P}_{s,r}(\zeta)$ has no effect. {The first column  $\NPC[1]$  of $\NPC$ at $\lambda=0$ is given by:}
\begin{multline}\label{NPC1.at.0}
	\NPC[1](0;\xi,\sgnt) = 
	e^{\frac{\pi \kappa(\xi)}{4}} 
	{e^{2it\xi^2 - \frac{i}{2} \sgnt \kappa(\xi)  \log |8t\xi^2|} }
	\diag{1}{ -e^{\frac{i\sgnt \pi}{4}} \sgn(\xi) }	\\
	\times
	\begin{bmatrix}
	 D_{i\sgnt \kappa(\xi)} \lp e^{\frac{i\sgnt \pi}{4}} |8t\xi^2|^{1/2} \rp \\
	 iA_{21}(\xi,\sgnt) D_{i\sgnt \kappa(\xi)-1} 
	   \lp e^{\frac{i\sgnt \pi}{4}} |8t\xi^2|^{1/2} \rp
	\end{bmatrix}
\end{multline}
\begin{lemma}\label{lem:NPC21}
Let $c_1, c_2, c_3$ be strictly positive constants, and suppose that $\rho \in H^{2,2}(\R)$ with $\oldnorm{\rho}_{H^{2,2}(\R)} \leq c_1$, $\inf_{\lam \in \R} (1 - \eps \lam |\rho(\lam)|^2 ) \geq c_2$, and $|\xi| < c_3$. Then as $|t| \to \infty$, 
\[
	|\NPC[21](0;\xi,\sgnt)|  \lesssim t^{-1/2},
\]
where the implied constant is independent of $\xi$ and $\rho$.
\end{lemma}

\begin{proof}
	From \eqref{NPC1.at.0} and \eqref{mod beta} we have, setting  $p := e^{\frac{i\sgnt \pi}{4}} |8t\xi^2|^{1/2} $,
\begin{align*}
	|\NPC[21](0; \xi,\sgnt)|	
	= 
	e^{\frac{\pi \kappa(\xi)}{4} } 
	\left| 
	  A_{21}(\xi,\sgnt) 
	  D_{i\eta \kappa(\xi) - 1} \lp  p \rp 
	\right| 
	= 
	\left| \frac{\kappa(\xi)}{8t\xi} \right|^{1/2} 
	\left| 
	  e^{\frac{\pi \kappa(\xi)}{4} } p 
	  D_{i\eta \kappa(\xi) - 1} \lp   p \rp
	\right|,
\end{align*}
%where we have set $p=|8t\xi^2|^{1/2}$.
Since $\kappa(\xi)/\xi \to \tfrac{\eps}{2\pi} |\rho(0)|^2$ as $\xi \to 0$, %we are done if we can
 {  it is sufficient to } show that the last factor is bounded in $p \geq 0$. For finite $p$ this is trivial, and for large $p$, the asymptotic expansion of $D_{\nu}(z)$ \cite[Eq. 12.9.1]{DLMF} gives 
\[
	\left| 
	  e^{\frac{\pi \kappa(\xi)}{4} } p 
	  D_{i\eta \kappa(\xi) - 1} \lp  p \rp
	\right|
	= \left| e^{-i p^2/4} p^{i \sgnt \kappa(\xi)} \left[1 + \bigo{p^{-2}} \right] \right|
	= 1 + \bigo{p^{-2}}.
%	\qedhere
\]
\end{proof}

We also need the following boundedness property:
\begin{lemma}[{see \cite[Appendix D]{LPS16}}]
\label{lem:PCmodel bound}
Let $c_1$ and $c_2$ be strictly positive constants, and suppose that $\rho \in H^{2,2}(\R)$ with $\oldnorm{\rho}_{H^{2,2}(\R)} \leq c_1$ and $\inf_{\lam \in \R} (1 - \eps \lam |\rho(\lam)|^2 ) \geq c_2$.
Then, 	
\begin{gather*}
	\oldnorm{ \NPC(\, \cdot\, ; \xi,\sgnt) }_\infty \lesssim 1 \qquad
	\oldnorm{ \NPC(\, \cdot\, ; \xi,\sgnt)^{-1} }_\infty \lesssim 1,
\end{gather*}
where the implied constants are uniform in $\xi$ and $|t|>1$ and depends only on $c_1$ and $c_2$.
\end{lemma}

\subsection{Existence theory for the RH model problem}
\label{sec:RHP.exist}
In this section, we prove that the solution $\Nrhp$ of our model problem, RHP~\ref{rhp.Nrhp} exists, by constructing it from the outer and local models introduced previously. We will show that for large times, the error solves a small norm Riemann-Hilbert problem which we can expand asymptotically.

Write the solution $\Nrhp$ of RHP~\ref{rhp.Nrhp} in the form 
\begin{equation}\label{error def}
	\Nrhp(\lam) = 
	\begin{cases}
		\error(\lam) \Nout(\lam) & |\lam-\xi| \notin \Uxi \\
		\error(\lam) \Nout(\lam) \NPC(\lam) & |\lam-\xi| \in \Uxi \\
	\end{cases}
\end{equation}
where  {  $\Uxi$ is defined in \eqref{xi disk},} $\Nout$, the solution of %Problem~\ref{outermodel}, 
\sidenote{34}
\revisedtext{RHP~\ref{outermodel},}
and $\NPC$, the solution of RHP~\ref{rhp.localmodel}, are both  bounded functions of $(x,t)$ {  having determinant equal to $1$}. This relation implicitly defines a transformation to a new unknown $\error$ which satisfies a new RH problem. In order to state it let 
\begin{equation*}
%\label{error contour}
	\Ske = \partial \Uxi \cup (\Sigma_2 \setminus \Uxi)
\end{equation*}
where the circle $\partial \Uxi$ is oriented counterclockwise. 

\begin{RHP}
\label{rhp.E}
Find a $2 \times 2$ matrix value function $\error$ analytic in $\C \backslash \Ske$ with the following properties:
\begin{enumerate}
	\sidenote{45}
	\item[(i)] For $\revisedtext{\lam} \in \C \setminus \calU_\xi$, $\error$ satisfies the symmetry relations 
	$
	\error_{22}(\lam) = \overline{\error_{11}(\lambar)}
	$ and
	$
	\error_{21}(\lam) = \eps \lam \overline{\error_{12}(\lambar)}.
	$
	\item[(ii)] $\error (\lam) = \stril{{\overline{q}_\error}} + \bigo{\lam^{-1}}$ as $|\lam| \to \infty$, for a constant ${\overline{q}_\error}$ determined by the symmetry condition above. 
	\item[(iii)] For $\lam \in \Ske$, the boundary values $\error_\pm$ satisfy the jump relation $\error_+(\lam) = \error_-(\lam) \vke(\lam)$ where
\sidenote{45}
\begin{equation}\label{error jump}
	\vke(\lam) = 
	\begin{cases}
		\Nout(\lam) \vk{2}(\lam) \Nout(\lam)^{-1} & \lam \in \Sigma_2 \setminus \Uxi \\
		\Nout(\lam) \NPC(\lam)^{-1}  \Nout(\lam)^{-1} &  \revisedtext{\lam} \in \partial \Uxi.
	\end{cases}
\end{equation}

\end{enumerate}
\end{RHP}

The jump matrix $\vke$ is uniformly near identity for large times; it follows from \eqref{V2 bound outside}, \eqref{local model expand} and Lemma~\ref{lem:outer.bound}, that 
\begin{gather}\label{error jump bound0}
	\left| \vke(\lam) - I \right| \lesssim 
	\begin{cases}
		|t|^{-1/2} & \lam \in \partial \Uxi \\
		e^{-2\sqrt{2} |t (\lam-\xi)^2|} & \lam \in \Sigma_2 \setminus \Uxi
	\end{cases}
\shortintertext{and} %it follows that}
	\label{error jump bound}
	\oldnorm{\vke - I}_{L^{2,k}(\R) \cap L^\infty(\R)} \lesssim |t|^{-1/2}, \quad k \in {   \N}.
\end{gather}
There is a well known existence and uniqueness theorem for RHPs with near identity jump matrices \cite{DIZ93,TO16,Zhou89}. 
Let $C_\error$ denote the Cauchy integral operator
\begin{equation*}
	C_\error f  = C^-(f(\vke-I)),
\end{equation*}
where $C^-$ is the usual Cauchy projection operator on $\Ske$: 
\[
	C^- f(\lam) = \lim_{\lam \to \Ske[-]} \frac{1}{2\pi i} \int_{\Ske} \frac{ f(z)}{z-\lam} dz.
\]
The essential fact needed for the small-norm theory is that $C_\error$ is a small norm operator,
\begin{equation}\label{error.op.bound}
	\oldnorm{ C_\error }_{L^2(\Ske) \rarr L^2(\Ske)}
	= \bigo{ \| \vke - I \|_\infty } = \bigo{|t|^{-1/2}}.
\end{equation}

\begin{lemma}
\label{lem:Err}
Suppose that $\rho \in H^{2,2}(\R)$ and $c:= \inf_{\lam \in \R} (1 - \lam | \rho(\lam)|^2) > 0$ strictly. 
Then, for sufficiently large times $|t|>0$, there exists a unique solution $\error(\lam;x,t)$ of RHP~\ref{rhp.E} with the property that 
\[
	\oldnorm{\error - {\begin{pmatrix} 1 &&& 0 \\\overline{q}_\error &&& 1 \end{pmatrix}}  }_{L^\infty(\C) } \lesssim |t|^{-1/2}.
\]
Moreover, as $\lam \to \infty$
\[
	\error(\lam) =  {\begin{pmatrix} 1 &&& 0 \\ \overline{q}_\error &&& 1 \end{pmatrix}}
	  + \lam^{-1} \error_{1} + \bigo{ \lam^{-2} }
\]
where ${{q}_\error} := \eps \lp \error_{1} \rp_{12}$ and
\begin{equation}\label{E1.12}
	2i \lp \error_{1} \rp_{12} = 
{ 	\frac{1}{|2t|^{1/2} }  }
	\lb A_{12}(\xi, \sgnt) \Nout[11](\xi)^2 
		+ A_{21}(\xi, \sgnt) \Nout[12](\xi)^2 
	\rb + \bigo{|t|^{-1}}, 
\end{equation}
Here, $\Nout$ is the solution of 
%the Problem~\ref{outermodel} 
\sidenote{34}
\revisedtext{RHP~\ref{outermodel}}
described in Lemma~\ref{outer.soliton}
while $A_{12}$ and $A_{21}$ are given by \eqref{Axi}-\eqref{beta arg}. 
\end{lemma}
 
\begin{proof}
Due to the nonstandard normalization we will construct the solution $\error$ row-by-row.
We begin by considering the first row, which we denote $\vect{e}_1 = \begin{pmatrix} \error_{11} & \error_{12} \end{pmatrix}$, which is canonically normalized.

By standard results in the theory of Cauchy integral operators \cite{DZ03}, $\vect{e}_1$ must satisfy
\begin{equation}\label{row1}
	\vect{e}_1(\lam) = (1 \ 0 )
	+ \frac{1}{2\pi i} \int_{\Ske} \frac{ ((1 \ 0)  + \vect{\mu}_1(z) )(\vke(z)-I) }{z-\lam} dz
\end{equation}
where $\vect{\mu}_1 \in L^2(\Ske)$ is the unique row vector solution of 
\begin{equation*}
%\label{mu}
	(1 - C_\error) \vect \mu_1 = C_\error (1 \ 0)
\end{equation*}
The existence and uniqueness of $\vect \mu_1$ follows immediately from \eqref{error.op.bound}
which establishes the existence of $(1- C_\error)^{-1}$, and allows one to construct $\vect \mu_1$ by Neumann series, moreover, we have 
\begin{equation}\label{mu bound}
	\oldnorm{\vect \mu_1 }_{ {  L^2(\Ske) }} \lesssim  
	\frac{ \oldnorm{ C_\error }_{L^2(\Ske) \rarr L^2(\Ske)} }{1- \oldnorm{ C_\error }_{L^2(\Ske) \rarr L^2(\Ske)} }
	\lesssim |t|^{-1/2}.
\end{equation}
Fix a small constant $d$ and suppose that $\inf_{z \in \Ske} |\lam - z| > d$, then 
\[
	\left| \vect{e}_1 - (1\ 0) \right| \leq  
	\frac{d^{-1}}{2\pi} \lp
		\oldnorm{\vke - I}_{{  L^1}} + \oldnorm{\vect \mu_1}_{{  L^2}} \oldnorm{\vke -I}_{{  L^2}} \rp \lesssim |t|^{-1/2}.
\]
To get $L^\infty$ control for $\lam$ approaching $\Ske$ we observe that the jumps on the contours $\Ske$ are locally analytic, and so \sidenote{46}\revisedtext{the contour $\Ske$ can be freely deformed to a new contour $\widetilde{\Sigma}^{(\error)}$, with different points of self-intersection,} by a bounded invertible transformation $\vect{e}_1 \mapsto \widetilde{\vect{e}}_1$. The previous argument then goes through to show that $\left| \widetilde{ \vect{e}}_1 - (1\ 0) \right|$ is bounded on $\Ske$ which then gives a similar bound on $\vect{e}_1$ as the transformation itself is bounded. 

To build the second row $\vect{e}_2 = \begin{pmatrix} \error_{21} & \error_{22} \end{pmatrix}$, we begin by using the symmetry condition to compute ${\overline{q}_\error}$. Since $\error_{21}(\lam) = \eps \lam \overline{\error_{12}(\overline{\lam})}$ for all large $\lam$, we use \eqref{row1} and take the limit as $\lam \to \infty$ to find
\[
	{q_\error} = \frac{\eps}{2 \pi i } \lp \int_{\Ske} \lb (1\ 0) + \vect{e}_1(z) \rb \lb \vke{z} - I \rb_2 dz \rp,
\]
where the subscript $2$ on the second factor of the integrand denotes the second column of the matrix. 
Finally, using \eqref{error jump bound} and \eqref{mu bound} we have the bound
\begin{equation}\label{alpha bound}
	|{\overline{q}_\error}| \lesssim |t|^{-1/2}.
\end{equation}

Now that $\overline{q}_\error$ is well defined, we construct the second row as 
\begin{equation*}
%\label{row2}
	\vect{e}_2(\lam) = ({\overline{q}_\error} \ 1 )
	+ \frac{1}{2\pi i} \int_{\Ske} \frac{ (({\overline{q}_\error} \ 1)  + \vect{\mu}_2(z) )(\vke(z)-I) }{z-\lam} dz
\end{equation*}
where
\begin{equation*}
%\label{mu2}
	(1- C_\error) \vect \mu_2 = C_\error ({\overline{q}_\error} \ 1).
\end{equation*}
Then repeating the arguments above we have that $\oldnorm{ \vect{\mu}_2}_{L^2{\Ske} } \lesssim  |t|^{-1/2}$ and also that 
$\oldnorm{\vect{e}_2 - ({\overline{q}_\error} \ 1)}_{L^\infty(\C)} \lesssim |t|^{-1/2}$.

Define the matrices 
$\error = \twovec{ \vect{e}_1 }{\vect{e}_2 }$ and $\vect{\mu} = \twovec{ \vect{\mu}_1 }{\vect{ \mu}_2 }$. Then, for large $\lam$ write 
$\error_0 = 
{\begin{pmatrix} 1 &0 \\ \overline{q}_\error &1
\end{pmatrix}}$
and
\begin{gather*}
	\error(\lam) = \error_0 + \lam^{-1} \error_{1} + \lam^{-2} \mathcal{S}(\lam) \\
	\error_{1} = \frac{-1}{2\pi i} \int\limits_{\Ske} (\error_0 + \vect \mu(z))(\vke(z) - I) dz, 
	%\qquad
\\
	\mathcal{S}(\lam) = \frac{\lam}{2\pi i} 
		\int\limits_{\Ske} \frac{(\error_0 + \vect \mu(z)) z (\vke(z) - I)}{(z-\lam)} dz .
\end{gather*}
Using \eqref{error jump bound} and \eqref{mu bound}, as $\lam \to \infty$ 
the residual $\mathcal{S}$ satisfies 
\[
	|\mathcal{S} | \lesssim \oldnorm{V_E-I }_{L^{2,2}(\R)} \lesssim |t|^{-1/2},
\]
while using \eqref{alpha bound} we have
%%%%%%%%%%%%%%%%%%%%%%%%%%%%%
%%	PAP changed to two lines
%%%%%%%%%%%%%%%%%%%%%%%%%%%%%
\begin{equation}\label{E1}
	\begin{aligned}
		\error_{1} 
		&= -\frac{1}{2\pi i} \oint_{\partial \Uxi} (\vke(z) -I) dz + \bigo{|t|^{-1}} \\
		&= |8t|^{-1/2} \Nout(\xi) {A(\xi,\sgnt) } \Nout(\xi)^{-1} + \bigo{|t|^{-1}}.
	\end{aligned}
\end{equation}
The last equality in \eqref{E1} follows from a residue calculation using \eqref{local model expand}. A direct calculation using the fact that $\det \Nout = 1$ then gives \eqref{E1.12}.
\end{proof}

Combining Lemmas~\ref{lem:outer.bound}, \ref{lem:PCmodel bound}, and \ref{lem:Err}, it follows from \eqref{error def} that

\begin{proposition}\label{prop:Nrhp.bound}
Let $c_1$ and $c_2$ be strictly positive constants, and suppose that $\rho \in H^{2,2}(\R)$ with $\oldnorm{\rho}_{H^{2,2}(\R)} \leq c_1$ and $\inf_{\lam \in \R} (1 - \eps \lam |\rho(\lam)|^2 ) \geq c_2$.
Then, 	
\begin{gather*}
	\oldnorm{ \Nrhp }_\infty \lesssim 1 \qquad
	\oldnorm{ (\Nrhp)^{-1} }_\infty \lesssim 1,
\end{gather*}
where the implied constants are uniform in $\xi$ and $|t|>1$ and depend only on $c_1$ and $c_2$.
\end{proposition}

When estimating  the gauge factor for  the solution $u$ of the DNLS equation (Section  \ref{subsec:sol-u}),  we need the following result that provides the large-time 
behavior of the error term $\error$ at $z=0$ : 
\begin{proposition}
\label{prop:E}
Suppose that $\rho \in H^{2,2}(\R)$ and $c:= \inf_{\lam \in \R} (1 - \eps \lam |\rho(\lam)|^2) > 0$ strictly. Then, as $|t| \to \infty$ the unique solution of RHP~\ref{rhp.E} described by Lemma~\ref{lem:Err} satisfies
%%%%%%%%%%%%%%%%%%%%%%%%%%%%%
%%
%%		PAP changed to align and added a line
%%
%%%%%%%%%%%%%%%%%%%%%%%%%%%%%
\begin{align}
\label{e11.at0.out}	
	\error_{11}(0) 
	&= 1
	- \frac{i\eps}{|2t|^{1/2}} 
	\Re \lb A_{12} \, \Nout[11](\xi) \overline{\Nout[12](\xi)} \rb 	
	+ \bigo{|t|^{-1}}
\shortintertext{
whenever $0 \neq \Uxi$; when $0 \in \Uxi$ it satisfies
}
\label{e11.at0.in}	
	\error_{11}(0) 
	&= 1
	- \frac{i\eps}{|8t|^{1/2}} 
	\lb
		2\Re \left( A_{12} \Nout[11](\xi) \overline{\Nout[12](\xi)} \right)
		-\overline{A_{12}} \Nout[12](0)  \overline{\Nout[11](0)}
	\rb \\
	\nonumber
	& \quad
	+ \bigo{|t|^{-1}} \\
\label{e12.at0.in}	
	\error_{12}(0) 
	 & = \bigo{|t|^{-1/2}}
\end{align}
where $A_{12}=A_{12}(\xi,\sgnt)$.
\end{proposition}

\begin{proof}
	Write $\vect{e}_1 = (\error_{11} \ \error_{12})$ for the first row of $\error$. Then starting from \eqref{row1} we use \eqref{error jump} and \eqref{local model expand} together with the bounds \eqref{error jump bound0}, \eqref{error jump bound} and \eqref{mu bound} to write
\[
	\vect{e}_1(0) = (1\ 0) - \frac{(1\ 0)}{|8t|^{1/2} 2\pi i} \int_{\partial \Uxi} 
	\frac{ \Nout(z) A(\xi,\sgnt) \Nout(z)^{-1}}{z(z-\xi)} dz + \bigo{|t|^{-1}}.
\]	
A residue calculation, using the symmetry condition of 
%Problem~\ref{outermodel} 
\sidenote{34}
\revisedtext{RHP~\ref{outermodel}}
and \eqref{mod beta} to simplify the result, completes the proof. 
\end{proof}

\subsection{The remaining \texorpdfstring{$\dbar$}{DBAR} problem}
\label{sec:dbar}

We are ready to consider the pure $\dbar$-Problem~\ref{dbar.n3} for $\nk{3}$. 
The next proposition describes  its large-time asymptotics.

\begin{proposition}
\label{prop:N3.est}
Suppose that $\rho \in H^{2,2}(\R)$ with $c \coloneqq \inf_{\lam \in \R} \left(1- \lam |\rho( \lam)|^2 \right) >0$ strictly.
Then, for sufficiently large time $|t|>0$, there exists a unique solution $\nk{3}(\lam;x,t)$ for $\dbar$-Problem~\ref{dbar.n3} with the property that 
\begin{equation}
\label{N3.exp}
\nk{3}(\lam;x,t) = I + \frac{1}{\lam}  \nk[1]{3}(x,t) + \littleo[\xi,t]{\frac{1}{\lam}}
\end{equation}
for $\lam=i y$ with $y \to +\infty$ where

\begin{equation}
\label{N31.est}
\left| \nk[1]{3}(x,t) \right| \lesssim |t|^{-3/4}
\end{equation}
 where the implied constant in \eqref{N31.est} is independent of $\xi$ and $t$ and uniform for 
$\rho$ in a bounded subset of $H^{2,2}(\R)$ with $\inf_{\lam \in \R} (1-\lam|\rho(\lam)|^2) \geq c>0$ for a fixed $c>0$.
\end{proposition}

\begin{proposition}\label{prop:n3 at 0}
Given the same assumptions as Proposition~\ref{prop:N3.est}  and for sufficiently large times $|t|>0$, the unique solution $\nk{3}(\lam;x,t)$ of $\dbar$-Problem~\ref{dbar.n3} satisfies
\begin{equation*}
	\nk[11]{3}(0;x,t) = 1 + \bigo{|t|^{-3/4}}
\end{equation*}
where the implied constant is independent of $\xi$ and $t$ and is uniform for $\rho$ in a bounded subset of $H^{2,2}(\R)$ with $\inf_{\lam \in \R} (1-\lam|\rho(\lam)|^2) \geq c>0$ for a fixed $c>0$. 	
\end{proposition}

\begin{remark}
The remainder estimate in \eqref{N3.exp} need not be (and is not) uniform in $\xi$ and $t$; what matters for the proof of Theorem \ref{thm:long-time} is that the implied constant in the estimate \eqref{N31.est} for $\nk[1]{3}(x,t)$ is independent of $\xi$ and $t$.
\end{remark}

To prove Proposition~\ref{prop:N3.est} we recast $\dbar$-Problem~\ref{dbar.n3} as a Fredholm-type integral equation using the solid Cauchy transform
\[
	(Pf)(\lam) = \frac{1}{\pi} \int_\C \frac{1}{\lam - z} f(z) \, dm(z) 
\]
where $dm$ denotes Lebesgue measure on $\C$.  The following lemma is standard.

\begin{lemma}\label{lem:DNLS.dbar.int}
A continuous, bounded row vector-valued function $\nk{3}(\lam;x,t)$ solves $\dbar$~Problem~\ref{dbar.n3} if and only if
\begin{equation}
\label{DNLS.dbar.int}
\nk{3}(\lam;x,t) = (1,0) + \frac{1}{\pi} \int_\C \frac{1}{\lam - z} \nk{3}(z;x,t) \Wk{3}(z;x,t) \, dm(z).
\end{equation}
\end{lemma}

\begin{proof}[Proof of Proposition \ref{prop:N3.est}, given Lemmas \ref{lemma:KW}--\ref{lemma:N31.est}]
As in \cite{BJM16}  and \cite{DM08}, we   first show that, for large times, the integral operator $K_W$ 
defined by
\begin{equation*} %\label{KW}
	\left( K_W f \right)(\lam) = \frac{1}{\pi} \int_\C \frac{1}{\lam-z} f(z) \Wk{3}(z) \, dm(z)
\end{equation*}
(suppressing the parameters $x$ and $t$) 
obeys the estimate
\begin{equation} \label{dbar.int.est1}
	\norm[L^\infty \to L^\infty] {K_W}\lesssim |t|^{-1/4} 
\end{equation}
where the implied constants depend only on $\norm[H^{2,2}]{\rho}$ and 
$$c : = \inf_{\lam \in \R} \left( 1- \eps \lam|\rho(\lam)|^2 \right)$$
and, in particular, are independent of $\xi$ and $t$.   
This is the object of Lemma \ref{lemma:KW}.
In particular, this shows that the solution formula
\begin{equation}
\label{N3.sol}
\nk{3} = (I-K_W)^{-1} (1,0) 
\end{equation}
makes sense and defines an $L^\infty$ solution of \eqref{DNLS.dbar.int} bounded uniformly in $ \xi \in \R$ and $\rho$ in a bounded subset of $H^{2,2}(\R)$ with $c>0$.  

We then prove that the solution $\nk{3}(\lam;x,t)$ has a large-$\lam$ asymptotic expansion of the form \eqref{N3.exp} where $\lam \to \infty$ along the \emph{positive imaginary axis} (Lemma \ref{lemma:N3.exp}). Note that, for such $\lam$, we can bound $|\lam-z|$ below by a constant times $|\lam|+|z|$.
The remainder need not be bounded uniformly in $\xi$. Finally, we  prove  estimate \eqref{N31.est}.
\end{proof}

\begin{proof}[Proof of Proposition~\ref{prop:n3 at 0}, given Lemmas \ref{lemma:KW}--\ref{lemma:N31.est}]
From \eqref{DNLS.dbar.int} we have 
\[
	\nk[11]{3}(0) = 1 - \frac{1}{\pi} \int_\C  
	\frac{\nk[11]{3}(z)\Wk[11]{3}(z)+ \nk[12]{3}(z)\Wk[21]{3}(z)}{z} dm(z).
\] 
Computing $\Wk[11]{3}$ {and $\Wk[21]{3}$} using \eqref{W3def} and the symmetry in 
%Problem~\ref{outermodel}, 
\sidenote{34}
\revisedtext{RHP~\ref{outermodel},}
recalling that $\dbar \mathcal{R}^{(2)}$ has zeros on its diagonal, gives
\begin{align*}
	&\frac{\Wk[11]{3}(z)}{z}  = 
	\Nout[12](z) \overline{\Nout[11](\bar z)} \frac{ \dbar \mathcal{R}_{21}^{(2)}(z)}{z}
   +\eps \Nout[11](z) \overline{\Nout[12](\bar z)} \dbar \mathcal{R}_{12}^{(2)}(z). \\
&{
	\frac{\Wk[21]{3}(z)}{z} =
	\overline{\Nout[11](\bar z)}^2 \frac{\dbar \mathcal{R}_{21}^{(2)}(z)}{z}
   -z \overline{\Nout[12](\bar z)}^2 \dbar \mathcal{R}_{12}^{(2)}(z) 
}
\end{align*}
{Equations \eqref{dbar.int.est1} and \eqref{N3.sol} imply that $| \nk{3}(z) | \lesssim 1$. Using} Lemma~\ref{lem:outer.bound} then gives
\[
%	\left| \frac{1}{\pi} \int_\C \frac{\Wk[11]{3}(z)}{z} dm(z) \right| \lesssim
	{\left| \nk{3}(0) - 1 \right| \lesssim }
		\int_\C \left| \frac{ \dbar \mathcal{R}_{21}^{(2)}(z)}{z} \right| 
		+ \left|\dbar \mathcal{R}_{12}^{(2)}(z) \right| dm(z)  = \bigo{|t|^{-3/4}}.
\]
Where the last equality uses Corollary~\ref{cor:R2.bd} to control the size of each term in the integrand, allowing identical estimates as those used to bound $\int |\Wk{3}(z)| dm(z)$ in 
%Lemma ~\ref{N3.exp}
{  Proposition ~\ref{prop:N3.est}} to establish the result. 
\end{proof}
%}

Estimates \eqref{N3.exp}, \eqref{N31.est}, and \eqref{dbar.int.est1} rest on the  bounds stated in three  lemmas.

\begin{lemma} \label{lemma:KW}
Suppose that $\rho \in H^{2,2}(\R)$ and 
$$c := \inf_{\lam \in \R} \left(1- \eps \lam |\rho(\lam)|^2 \right) >0$$ strictly.
Then, the estimate \eqref{dbar.int.est1} holds, where the implied constants depend on $\norm[H^{2,2}]{\rho}$ and $c$.
\end{lemma}

\begin{proof}
To prove \eqref{dbar.int.est1}, first note that
\begin{align*}
 	\norm[\infty]{K_W f} & \leq \norm[\infty]{f} \int_\C \frac{1}{|\lam-z|}|\Wk{3}(z)| \, dm(z) 
\end{align*}
where, using Proposition~\ref{prop:Nrhp.bound}, 
\[
	|\Wk{3}(z)| \leq 
	\norm[\infty]{\Nrhp} \norm[\infty]{(\Nrhp)^{-1}} \left| \dbar \mathcal{R}^{(2)}\right|
	\lesssim \left| \dbar \mathcal{R}^{(2)}\right|.
\]
We will prove the estimate for $\eta = \sgnt = +1$ and $z \in \Omega_1$ since estimates for $\eta = -1$ and $\Omega_3$, $\Omega_4$, and $\Omega_6$ are similar.  
Setting $\lam =\alpha+i\beta$ and $z-\xi=u+iv$, the region $\Omega_1$ corresponds to 
\begin{equation}\label{omega.1}
	\Omega_1= \left\{ (\xi+u,v): v \geq 0, \, v \leq u < \infty\right\}.
\end{equation}
We then have from Corollary~\ref{cor:R2.bd} that
\[
	\int_{\Omega_1}  \frac{1}{|\lam-z|} |\Wk{3}(z)| \, dm(z)  \lesssim  I_1 + I_2 + I_3 + I_4
\]
where 
%%%%%%%%%%%%%%%%%%%%%%%%%%%%%
%%
%%	PAP changed to four lines
%%
%%%%%%%%%%%%%%%%%%%%%%%%%%%%%
\begin{align*}
	I_1 &= 
	 \int_0^\infty \int_v^\infty \frac{1}{|\lam-z|} |p_1'(u)| e^{-8tuv} \, du \, dv 
	\\[5pt]
	I_2 &= 
	 \int_0^1 \int_v^1 \frac{1}{|\lam-z|} \left| \log(u^2 + v^2) \right| e^{-8tuv} \,du\,dv
	\\[5pt]
	I_3 &= 
	 \int_0^\infty \int_v^\infty \frac{1}{|\lam-z|} \frac{1}{1+ |z-\xi|} e^{-8tuv} \, du \, dv
	\\[5pt] 
	I_4 &= 
	\int_0^\infty \int_v^\infty \frac{1}{|\lam-z|} |\indicator(z)| e^{-8tuv} \, du \, dv. 
\end{align*}
We recall from \cite[proof of Proposition C.1]{BJM16} the bound
\sidenote{47}
\begin{equation*} 
%\label{BJM16.bd1}
\revisedtext{\norm[L^2(v,\infty)]{\frac{1}{\lam - z}} = \Big(  \int_{v+\xi}^\infty \frac{du}{(u-\alpha)^2+(v-\beta)^2}           \Big)^{1/2} }
	\leq \frac{\pi^{1/2}}{|v-\beta|^{1/2}}
\end{equation*}
where $z= \xi + u+iv$ and $\lam =\alpha + i \beta$. Using this bound and Schwarz's inequality on the $u$-integration we may bound $I_1$  by constants times
\[
	(1+\norm[2]{p_1'})  \int_0^\infty \frac{1}{|v-\beta|^{1/2}} e^{-tv^2} \, dv \lesssim t^{-1/4}
\]
(see for example \cite[proof of Proposition C.1]{BJM16} for the estimate)
For $I_2$, we remark that $ |\log(u^2+v^2)| \lesssim 1+ |\log(u^2)|$ and that $1+ |\log(u^2)|$ is square-integrable on $[0,1]$. We can then argue as before to conclude that
$ I_2 \lesssim t^{-1/4}$. Similarly, the inequality 
\[ 
	\frac{1}{1+|z-\xi|} \leq \frac{1}{1+u}
\] 
and the finite support of $\indicator$ shows that we can bound $I_3$ and $I_4$ in a similar way.

It now follows that
\[ 
	\int_{\Omega_1} \frac{1}{|\lam-z|} |\Wk{3}(z)| \, dm(z) \lesssim t^{-1/4} 
\]	
which, together with similar estimates for the integrations over $\Omega_3$, $\Omega_4$, and $\Omega_6$,  proves \eqref{dbar.int.est1}.
\end{proof}

\begin{lemma}
\label{lemma:N3.exp}
For $\lam=iy$, as $y \to +\infty$, the expansion \eqref{N3.exp} holds with 
\begin{equation}
\label{N3.1}
\nk[1]{3}(x,t) = \frac{1}{\pi} \int_{\C} \nk{3}(z;x,t) \Wk{3}(z;x,t) \, dm(z) . 
\end{equation}
\end{lemma}

\begin{proof}
We write \eqref{DNLS.dbar.int} as 
\[
	\nk{3}(\lam;x,t) = (1,0) + \frac{1}{\lam} \ \nk[1]{3}(x,t) 
	+ \frac{1}{\pi \lam} \int_{\C} \frac{z}{\lam-z} \nk{3}(z;x,t) \Wk{3}(z;x,t) \, dm(z)
\]
where $\nk[1]{3}$ is given by \eqref{N3.1}. If $\lam =iy$ and $z \in \Omega_1 \cup \Omega_3 \cup \Omega_4 \cup \Omega_6$, it is easy to see that $|\lam|/|\lam-z|$ is bounded above by a fixed constant independent of $\lam$, while $|\nk{3}(z;x,t)| \lesssim 1$ by the remarks following \eqref{N3.sol}. If we can show that 
$$\int_\C |\Wk{3}(z;x,t)| \, dm(z) < \infty, $$ it will follow from the Dominated Convergence Theorem that 
\[
	\lim_{y \to \infty} \int_\C \frac{z}{iy-z} \nk{3}(z;x,t) \Wk{3}(z;x,t) \, dm(z) = 0 
\] 
which implies the required asymptotic estimate. We will estimate the integral $\displaystyle \int_{\Omega_1} |\Wk{3}(z)| \, dm(z)$ since the other estimates are similar.  
Using Corollary~\ref{cor:R2.bd} and \eqref{omega.1}, we may then estimate
\[
\int_{\Omega_1} |W^{(3)}(z;x,t)| \, dm(z)	\lesssim  I_1+I_2+I_3+I_4
\]
where for $\eta = \sgn t =1$, 
\begin{align*}
&I_1 = \int_0^\infty \, \int_v^\infty \left| p_1'(\xi+u) \right| e^{-8tuv} \, du \, dv
&&I_2 =	\int_0^1 \int_v^1 \left| \log(u^2 + v^2) \right| e^{-8tuv} \, du \, dv \\
&I_3 = \int_0^\infty \, \int_v^\infty \frac{1}{\sqrt{1+u^2+v^2}} e^{-8tuv} \, du \, dv 
&&I_4 = \int_0^\infty \int_v^\infty  |\indicator(z)| e^{-8tuv} \, du \, dv. 
\end{align*}
For $\eta = -1$, the integrations domains are reflected in the $u$-variable, and the following estimates are altered in an obvious manner. 
To estimate $I_1$, we use the Schwarz inequality on the $u$-integration to obtain
\[
	I_1 \leq 	\norm[2]{p_1'} \frac{1}{4\sqrt{t}} \int_0^\infty \frac{1}{\sqrt{v}}e^{-  8 tv^2} \, dv
		=		\norm[2]{p_1'} \frac{\Gamma(1/4)}{ 8^{5/4}t^{3/4}}.
\]
Since $\log(u^2+v^2) \leq \log(2u^2)$ for $v \leq u \leq 1$, we may similarly bound
\[
	I_2  \leq 	\norm[L^2(0,1)]{\log(2u^2)} \frac{\Gamma(1/4)}{8^{5/4}t^{3/4}}.
\]
To estimate $I_3$, we note that $1+u^2+v^2 \geq  1+u^2$ and $(1+u^2)^{-1/2} \in L^2(\R^+)$, so we may 
conclude that
\[
	I_3 \leq \norm[2]{(1+u^2)^{-1/2}} \frac{\Gamma(1/4)}{ 8^{5/4}t^{3/4}}.
\]
Finally, as $\indicator$ is finitely supported, by similar bounds
\[
	I_4 \leq C_\poles \frac{\Gamma(1/4)}{ 8^{5/4}t^{3/4}}.
\]
where the constant $C_\poles$ depends only on the discrete spectrum $\poles$. 
These estimates together show that
\begin{equation} \label{W.L1.est}
	\int_{\Omega_1} |\Wk{3}(z;x,t)| \, dm(z)	\lesssim t^{-3/4}
\end{equation}
and that the implied constant depends only on $\norm[{H^{2,2}}]{\rho}$ and $\poles$.  
In particular, the integral \eqref{W.L1.est} is bounded uniformly as $t \to \infty$.
\end{proof}

The estimate \eqref{W.L1.est} is also strong enough to prove % \eqref{dbar.int.est3}.
\eqref{N31.est}:

\begin{lemma} \label{lemma:N31.est}
The estimate \eqref{N31.est} %\eqref{dbar.int.est3}
holds with constants uniform in $\rho$ in a bounded subset of $H^{2,2}(\R)$ and $\inf_{\lam \in \R} \left(1-\eps \lam |\rho(\lam)|^2 \right) > 0$ strictly.
\end{lemma}

\begin{proof}
From the representation formula \eqref{N3.1}, Lemma \ref{lemma:KW}, and the remarks following, we have
\[ 
	\left|\nk[1]{3}(x,t) \right| \lesssim \int_\C |\Wk{3}(z;x,t)| \, dm(z). 
\]
In the proof of Lemma \ref{lemma:N3.exp}, we bounded this integral by $t^{-3/4}$ modulo constants with the required uniformities.
\end{proof}

%\input{largetime}	   		%% 	Large-time asymptotics

%%%%%%%%%%%%%%%%%%%%%%%%%%%%%%%%%%%%%%%
%
%		LARGE-TIME ASYMPTOTICS
%		Large-Time Asymptotics
%		File dnls-imrn-arXiv-largetime.tex
%		Created 4.7.2017
%
%%%%%%%%%%%%%%%%%%%%%%%%%%%%%%%%%%%%%%%

\newcommand{\ssstyle}{\scriptscriptstyle}

\section{Large-time asymptotics for solutions of DNLS} 
\label{sec:largetime}

We now gather the estimates on the RHPs considered previously to prove Theorems \ref{thm:long-time} and \ref{thm:long-time-gauge} which provide precise asymptotic descriptions of the large-time behavior of the solutions $q(x,t)$ and $u(x,t)$ 
of \eqref{DNLS2} and \eqref{DNLS1} within any given space-time cone $\mathcal{S}$.
The leading\revisedtext{-}order soliton component of each expansion 
%arise 
\sidenote{48}\revisedtext{arises}
from our outer model $\Nout$. 
We begin with the following result, which describes the soliton components in $\Nout$ in terms of only those solitons whose speeds are `visible' from $\mathcal{S}$. 

Recall the notation \eqref{distances}, \eqref{S.cone}-\eqref{posint} used in Theorem~\ref{thm:long-time}-\ref{thm:long-time-gauge} and for any real interval $I$ let
\begin{equation}\label{nu0}
	\nu_0(I)  = \nu_0 = \min_{ \lam \in \poles \setminus \poles(I)} 
	\dist(\Re \lam, I) .
\end{equation}
Additionally, for any multi-soliton $\qsol(x,t;\calD)$, where $\calD$ is the associated reflectionless scattering data, we use \eqref{u.gauge} 
\sidenote{49}\revisedtext{below}
to define
\begin{equation}\label{usol.1}
	u_{\mathrm{sol}}(x,t;\calD)  :=
	\qsol(x,t;\calD) 
	\exp \lp
		  -i\eps \int_x^{\infty} |\qsol(y,t;\calD)|^2 dy
	\rp
\end{equation}
\begin{proposition}
\label{prop:sol separation}
Fix $x_1 \leq x_2$ and $v_1 \leq v_2$.
Take $\mathcal{S}$, $I$, and $\calD_I$ as in Theorem~\ref{thm:long-time}; take $\Nout$ and $\calD_\xi$ as in Proposition~\ref{outer.soliton}. Then as $|t| \to \infty$ with $(x,t) \in \mathcal{S}(v_1,v_2,x_1,x_2)$ we have
\begin{equation}\label{Nout.to.Nsol}
   \Nout(\lam) =   \mathcal{N}^{\mathrm{sol} }  
   (\lam \,\vert\, \calD_I)
   \ \prod_{\mathclap{\Re \lam_k \in \negint \setminus I }}
   \quad \  \lp \frac{\lam - \lambar_k}{\lam-\lam_k} \rp^{\sig} 
	+ \bigo{ e^{-4\poledist \nu_0 |t|} }
\end{equation}
where ${\mathcal N}^{\mathrm sol}(\lam \,\vert\, \calD_I)$
 is the solution of RHP~\ref{RHP2} corresponding to the reflectionless scattering data $\calD_I$.
In particular, 
\begin{align}
	\label{N12.out.inf}
	&2i (\Nout[1])_{12} := \lim_{\lam \to \infty} 2i\lam \Nout[12](\lam;x,t) 
%	   = \qsol(x,t;\calD_\xi) 
	   = \qsol(x,t;\calD_I) + \bigo{e^{-4\poledist \nu_0 |t| }}  \\
%    &\begin{aligned}[t]
%	   \Nout[11](0;x,t)^{-2}
%	   &= \exp \Big( -i\eps \int_{x}^\infty | \qsol(y,t;\calD_\xi)|^2 dy \Big) \\
%	   &= \exp \Big( -i\eps \int_{x}^\infty | \qsol(y,t;\calD_I)|^2 dy  
%	   + 4i \sum_{\mathclap{\Re \lam_k \in \negint \setminus I}} \arg \lam_k \Big)
%	   \lp 1 + e^{-4\poledist \nu_0 |t| } \rp. 
%    \end{aligned}
   \label{N.out.zero}
   &\begin{multlined}[.9\textwidth][t]
     \Nout(0;x,t) = 
     \exp \bigg( \frac{i \eps \sig}{2} \int_x^\infty |\qsol(y,t;\calD_I)|^2 dy \bigg) \\
     \times 
     \begin{pmatrix} 
       1 & -\int_{x}^\infty \usol(y,t;\calD_I) dy \\ 
       0 & 1 
     \end{pmatrix}
     \exp \bigg( 
       -2i\sig \sum_{\mathclap{\scriptscriptstyle \Re \lam_k \in \negint\setminus I}} \arg \lam_k
     \bigg) 
     + \bigo{ e^{-4\poledist \nu_0 |t|} }.
   \end{multlined}
\end{align}
\end{proposition}

\begin{proof}
\sidenote{50}Apply Lemma~\ref{lem:Nsol.asymp} and \revisedtext{Corollary~\ref{cor:Nsol.asymp} of Appendix B} to $\Nout$--as described by Proposition~\ref{outer.soliton}--and observe that $\calD_I = \widehat{\calD_\xi}$. 
\end{proof}

\subsection{Large-time asymptotic for solution \texorpdfstring{$q$}{q}  of  (1.3) }
 \label{subsec:sol-q}

%The solution $q$ is recovered through the reconstruction formula
%\eqref{q.lam}
%\begin{equation} \label{recons-q}
%q(x,t)  = \lim_{z\to \infty}  2iz n_{12}(x,t,z),
%\end{equation}
%where $z \rarr \infty$ in any direction not tangent to the contour $\R$ and
%$\nn$  satisfies RHP~\ref{RHP2}. 
Inverting the sequence of transformations \eqref{n1}, \eqref{n2 def}, \eqref{n3 def} we construct the solution $\nn$  as 
\begin{align} \label{n.soln}
	\nn(\lam) 
	= \nk{3}(\lam) \Nrhp(\lam) \mathcal{R}^{(2)}(\lam)^{-1}  \delta(\lam)^{\sig}.
\end{align}

It follows from \eqref{Texpand} and Corollary~\ref{cor:R2.bd} 
that as $\lam \to \infty$ non-tangentially to the real axis, 
\begin{align*}
	 \delta(\lambda)^{\sig} = I + \lp \delta_1/\lam \rp \sigma_3 + \bigo{ \lam^{-2} },
	\qquad
%	\delta_1(\xi,\eta) =  i \int_{\negint} \kappa(z) dz 
	\mathcal{R}^{(2)} = I + \bigo{ e^{-c|t|} }. 
\end{align*}
From 	\eqref{error def}, we have 
\begin{equation*}
\Nrhp(\lam) = \error(\lam)\Nout(\lam) 
\end{equation*}
and  the large-$\lambda$ behavior of  $\error(\lam)$ and $\Nout(\lam)$ are given in Proposition~\ref{outer.soliton} and Lemma~\ref{lem:Err} as 
\begin{gather*}
    \mathcal{E}(\lam) = 
    \mathcal{E}_0 + \lam^{-1} \mathcal{E}_1 + \bigo{ \lam^{-2} }, \quad 
    \mathcal{E}_0 = \tril{ \bar q_\mathcal{E} },
    \\
    \Nout(\lam) = \Nout[0] + \lam^{-1} \Nout[1] + \bigo{ \lam^{-2} }, \quad 
    \Nout[0] = \tril{  \eps \overline{ (\Nout[1])_{12} } }.
\end{gather*}

\begin{proof}[Proof of Theorem~\ref{thm:long-time}]
Inserting the above expansions into \eqref{n.soln} and using Proposition~\ref{prop:N3.est}
the reconstruction formula \eqref{q.lam}  gives
\begin{gather*}
	q(x,y) = 2i (\Nout[1])_{12} + 2i (\error_1)_{12} + \bigo{|t|^{-3/4}}.
\end{gather*}
Using \eqref{N12.out.inf} for the first term on the right-hand-side and \eqref{E1.12} to identify $2i (\error_1)_{12}$ with the correction factor $f(x,t)$ in Theorem~\ref{thm:long-time} gives \eqref{q.long-time}. 
\end{proof}

%% PAP Commented Out
\subsection{Large-time asymptotic for solution \texorpdfstring{$u$}{u}  of  (1.1) }
 \label{subsec:sol-u}

To prove Theorem~\ref{thm:long-time-gauge} we construct the solution $u(x,t)$ of the DNLS equation \eqref{DNLS1} with initial data $u_0$ by means of the inverse gauge transformation
\begin{equation}\label{u.gauge}
%	u(x,t) = \calG^{-1}(q)(x,t) 
%	:= q(x,t) \exp \left( i \epsilon \int_{-\infty}^x |q(y,t)|^2 dy \right).
	u(x,t) = \calG^{-1}(q)(x,t) 
	:= q(x,t) \exp \left( -i \varepsilon \int_x^{\infty} |q(y,t)|^2 dy \right)
\end{equation}
As we have the large-time behavior of $q(x,t)$ in hand, to compute the large-time behavior of $u(x,t)$ it will suffice to evaluate the large-time asymptotics of the expression
\begin{equation}\label{u.gauge.fact}
%	\exp \left( i \epsilon \int_{-\infty}^x |q(y,t)|^2 dy \right).
	\exp \left( -i \varepsilon \int_x^{\infty} |q(y,t)|^2 dy \right).
\end{equation}
%We will prove
\begin{proposition}\label{prop:u.gauge.expansion}
	Suppose that $q_0 \in H^{2,2}(\R)$ and that $q(x,t)$ solves \eqref{DNLS2} with initial data $q_0$.
	Let $\{ \rho, \{(\lam_k, c_k)\}_{k=1}^N \}$ be the scattering data associated to $q_0$. {Fix $\xi = -x/(4t)$ and $M>0$}. 
Fix real constants $v_1 \leq v_2$ and $x_1 \leq x_2$. Define {$\mathcal{S}$}, $I$ and $\mathcal{D}_I$ as described in Theorem~\ref{thm:long-time} and take $\alpha_0$ as in Theorem~\ref{thm:long-time-gauge}. Then as $|t| \to \infty$ {with $(x,t) \in \mathcal{S}(v_1,v_2,x_1,x_2)$:} \\
%inside the cone 
%\[
%	x_1 + v_1 t \leq x \leq x_2 + v_2 t	
%\]

\noindent
For {$\xi \geq M|t|^{1/8}$}, we have
\begin{align}
\label{u.gauge.out}
	\exp & \lp -i \eps \int_x^{\infty}  |q(y,t)|^2 dy \rp =  \\ 
	\nonumber
	& \lb
	1
	+ \frac{2i\eps}{|2t|^{1/2}} 
	\real \lb 
	  A_{12}(\xi,\sgnt)  { {\mathcal N}_{11}^{\mathrm sol}}  
	  (\xi\,\vert\, \calD_I) 
	  \overline{   { {\mathcal N}_{12}^{\mathrm sol}}      
	  (\xi\,\vert\, \calD_I)} 
	  e^{i\phi(\xi)}
	\rb
	+ \bigo{|t|^{-3/4}}
	\rb \\
	\nonumber
	&
	\times \exp 
	\lp 
		-i \eps \int_x^{\infty} |\qsol(y,t;\mathcal{D}_I)|^2 dy 
	\rp
	e^{i \alpha_0(\xi, \sgnt)},
\end{align}
while for {$\xi \leq M|t|^{-1/8}$},
\begin{multline}\label{u.gauge.in}
%   \exp \lp i \eps \int_{-\infty}^x |q(y,t)|^2 dy \rp =    % original gauge
	\exp \lp -i \eps \int_x^{\infty} |q(y,t)|^2 dy \rp =    % new gauge
	%% PP
	\\
	F(\xi,t,\sgnt)
	\Bigg[
	  1  
	  + \frac{i\eps}{|2t|^{1/2}} 
	  \bigg\{
		2\real \lb 
		  A_{12}(\xi,\sgnt) { {\mathcal N}_{11}^{\mathrm sol}}  % \Nsol[11]{\emptyset}
		  (\xi;x,t \, \vert \, \calD_I)
		  \overline{ { {\mathcal N}_{12}^{\mathrm sol}}  %\Nsol[12]{\emptyset}
		  (\xi;x,t \, \vert \, \calD_I) }
		  e^{i \phi(\xi) }
		\rb  
%		+ \overline{\Nsol[12](\xi)} A_{12}(\xi,\sgnt) \Nsol[11](\xi)  
	    \\
	    + \overline{A_{12}(\xi,\sgnt)} 
	    \lp 1 - G(\xi,t,\sgnt) \rp 
	    \exp \Big( 
	      4i\sum\limits_{\mathclap{\Re \lam_k \in \negint \setminus I}} \arg \lam_k 
	    \Big)
	    \int_x^\infty \usol(y,t;\mathcal{D}_I) dy
	  \bigg\} 
	  +\bigo{|t|^{-3/4}} 
	\Bigg] 
	\\
	\times \exp
	\lp -i \eps \int_x^{\infty} |\qsol(y,t;\mathcal{D}_I)|^2 dy \rp 
	e^{i \alpha_0(\xi,\sgnt) }
\end{multline}
where, with $p := e^{\frac{i\sgnt \pi}{4}} |8t\xi^2|^{1/2}$ 
\begin{equation}\label{FandG}
%\begin{aligned}
	F(\xi,t,\sgnt) =  \lb e^{p^2/4} p^{-i \sgnt \kappa(\xi)}
	   D_{i\eta \kappa(\xi)} \lp p \rp \rb^{-2}, 
	\qquad
	G(\xi,t,\sgnt) = 
	  \frac{ p D_{i\eta \kappa(\xi) - 1} \lp p \rp}
	  {D_{i\eta \kappa(\xi)} \lp p \rp},
%\end{aligned}
\end{equation}
and $\usol$ is defined by \eqref{usol.1}.
\end{proposition}

\begin{proof}[Proof of Proposition~\ref{prop:u.gauge.expansion}]
Using \eqref{N1.near.0}-\eqref{N1.at.0}, the gauge factor \eqref{u.gauge.fact} can be expressed in terms of spectral functions. Starting from \eqref{n.soln}, observe that 
$\mathcal{R}^{(2)}(0) = \striu{*}$ 
is upper triangular, see  Figure~\ref{fig:n2def} and equation~\eqref{R_k};
similarly, the symmetry in condition 1 of RHP~\ref{rhp.Nrhp} guarantees that $\Nrhp[21](0) = 0$. So we have 
\begin{multline}
\label{n11.at.0.a}
	\exp \lp -i \eps \int_{x}^\infty |q(y,t)|^2 dy \rp
	= \nn_{11}(0;x,t)^{-2}
	=  [ \nk[11]{3}(0) \Nrhp[11](0) \delta(0) ]^{-2} \\
	= \Nrhp[11](0)^{-2} \exp \bigg( 
		{\frac{i}{\pi}} \int_{\negint} 
		\frac{ \log( 1 - \eps \lam |\rho(\lam)^2|)}{\lam} d\lam 
	  \bigg) + \bigo{|t|^{-3/4}},
\end{multline}
where we've used \eqref{T} and Proposition~\ref{prop:n3 at 0} in the last equality.

The value of $\Nrhp[11](0)$ depends on the location of the point $\xi$ in the spectral plane. \\

If $|\xi| > \poledist/3$ then $0 \not\in \Uxi$ (cf. \eqref{xi disk}), so it follows from \eqref{error def} and \eqref{N.out.zero} that  
\begin{align}
	\Nrhp[11](0)^{-2} 
		&= \error_{11}(0)^{-2} 
		\exp \bigg(  
		  {-i\eps} \int_x^\infty |\qsol(y,t,\mathcal{D}_I)|^2 dy 
		  +4i \sum_{\mathclap{\Re \lam_k \in \negint\setminus I}} \arg \lam_k
        \bigg) 
 \label{nrhp.at.0.a}
        \\
 \nonumber
 		&\quad
        + \bigo{ e^{-4\poledist \nu_0 |t|} }. 
\end{align}
Plugging  \eqref{nrhp.at.0.a} into \eqref{n11.at.0.a} 
and using \eqref{e11.at0.out} and \eqref{Nout.to.Nsol} to evaluate $\error_{11}(0)$ gives \eqref{u.gauge.out}. 

If $|\xi| < \poledist/3$ then $0 \in \Uxi$. We expand \eqref{error def}, using Lemma~\ref{lem:NPC21} and \eqref{e12.at0.in} to drop terms of order $t^{-1}$: 
\begin{equation*}
%\label{Nrhp.at0.in.a}
	\Nrhp[11](0) 
		 =  \error_{11}(0)  \NPC[11](0) \Nout[11](0)
		 \lp  1 + \frac{\Nout[12](0)}{\Nout[11](0)} \frac{\NPC[21](0)}{\NPC[11](0)} \rp
		 + \bigo{|t|^{-1}}. 
\end{equation*}
Using \eqref{N.out.zero} this becomes
\begin{multline}
	\Nrhp[11](0)^{-2} =
	\error_{11}(0)^{-2} \NPC[11](0)^{-2} 
	\exp \bigg( 
	  {-i\eps} \int_x^\infty |\qsol(y,t,\mathcal{D}_I)|^2 dy
	  +4i \sum_{\ssstyle \mathclap{\Re \lam_k \in \negint \setminus I}} \arg \lam_k  
	\bigg) 
\label{nrhp.at.0.b}
	\\
	\times \Bigg[ 1  - \frac{\NPC[21](0)}{\NPC[11](0)} 
	\exp \big( 
	  4i \sum_{\ssstyle \mathclap{\Re \lam_k \in \negint \setminus I}} \arg \lam_k 
	\big) 
	\int_{x}^\infty u_{\mathrm{sol}}(y,t;D_I) dy
	+ \bigo{|t|^{-1}}
	\Bigg]^{-2} .
\end{multline}

Expanding $\error_{11}(0)^{-2}$ using \eqref{e11.at0.in} and Proposition~\ref{prop:sol separation}
\begin{multline*}
\error_{11}(0)^{-2} = 
	  1  
	  + \frac{i\eps}{|2t|^{1/2}} 
	  \bigg\{ 2\real 
	  \lb 
	    A_{12}(\xi,\sgnt) \Nsol[11]{\emptyset}(\xi\,\vert\, \calD_I) 
	    \overline{\Nsol[12]{\emptyset}(\xi\,\vert\, \calD_I)}
	    e^{i \phi(\xi)} 
	  \rb
	 \\
	 + \overline{A_{12}(\xi,\sgnt)} 
	 \exp \big( 
	   4i\sum_{\mathclap{\ssstyle \Re\lam_k \in \negint \setminus I}} \arg\lam_k 
	 \big) 
	 \int_x^\infty \usol(y,t;\mathcal{D}_I) dy
	\bigg\} 
	+\bigo{|t|^{-1}}. 
\end{multline*}

By introducing the notation $p := e^{\frac{i\sgnt \pi}{4}} |8t\xi^2|^{1/2} $, we use \eqref{NPC1.at.0} and the symmetry $A_{21} = \eps \xi \overline{A_{12}}$ (see \eqref{mod beta}) to write
\begin{equation}
\label{NPC.at.0.p}
\begin{gathered}  
	\NPC[11](0)^{-2} = 
	  \lb e^{p^2/4} p^{-i\eta \kappa(\xi)} D_{i \sgnt \kappa(\xi)}(p) \rb^{-2} 
	  = F(\xi,t,\sgnt) \\
	\frac{\NPC[21](0)}{\NPC[11](0)} 
%	  = - 2 i A_{21} e^{-i\eta \pi/4} \sgn(\xi) 
%     \frac{  D_{i\eta \kappa(\xi) - 1} \lp p \rp} {D_{i\eta \kappa(\xi)} \lp p \rp} 
     = -i\eps |\xi|  e^{i\eta \pi/4} \overline{A_{12}} 
        \frac{  D_{i\eta \kappa(\xi) - 1} (p)}{D_{i\eta \kappa(\xi)} (p)} 
     =  -\overline{A_{12}} i\eps |8 t|^{-1/2}  
        p \frac{  D_{i\eta \kappa(\xi) - 1} (p)}{D_{i\eta \kappa(\xi)} (p)}.
\end{gathered}
\end{equation}
%By expanding $\error_{11}(0)^{-2}$ using \eqref{e11.at0.in} and evaluating both $\NPC[11](0)$ and $\NPC[21](0)/\NPC[11](0)$ using \eqref{NPC1.at.0}, we arrive at 
Combining \eqref{nrhp.at.0.b}-\eqref{NPC.at.0.p} with \eqref{n11.at.0.a} gives \eqref{u.gauge.in}.

Finally, we observe \cite[Eq.~12.9.1]{DLMF} that inserting the expansion $$D_\nu(p) = e^{-p^2/4} p^{\nu} \lb 1 + \bigo{p^{-2}} \rb$$ into \eqref{FandG} gives
\[
	F(p) = 1 + \bigo{|t|^{-1}\xi^{-2}}
	\qquad 
	G(p) = 1 + \bigo{|t|^{-1}\xi^{-2}}
\]	
so that the inner expansion \eqref{u.gauge.in} for $|\xi| \leq \poledist/3$ agrees with the outer expansion \eqref{u.gauge.out} for $|\xi| > M |t|^{-1/8}$.

\end{proof}

\section*{Acknowledgements} 

\revisedtext{The authors thank the referee for a very careful reading of the manuscript and a number of suggestions which helped us make the revised version more readable.}

PAP and JL would like to thank the University of Toronto and the Fields Institute for hospitality during part of the time that this work was done. Conversely, RJ and CS would like to thank the department of Mathematics at the University of Kentucky for their hospitality during  part of the time that this work was done.  

PAP was supported in part by Simons Foundation Research and Travel Grant 359431, and CS was supported in part by Grant 46179-13 from the Natural Sciences and Engineering Research Council of Canada.

\appendix			   			

%\input{appA}						%%	Weak Plancherel formula

%%%%%%%%%%%%%%%%%%%%%%%%%%%%%%
%
%		appA.tex - Weak Plancherel Formula
%		7.23.2017
%
%%%%%%%%%%%%%%%%%%%%%%%%%%%%%%

\section{The Weak Plancherel Formula}
\label{app:Weak}

We establish 
relations between the transmission coefficients  $\balpha$ and $\alpha$ and the scattering data $ \left( \rho, \{\lambda_k\}_{k=1}^N \right)$. {Recall that $\Lam^+ = \{ \lam_k \}_{k=1}^N \subset \C^+$.}

\begin{lemma}
The following relations 
\begin{equation}
\label{transmission+}
\balpha(\lambda)=\prod_{k=1}^N \frac{\lambda-\lambda_k}{\lambda-\overline{\lambda}_k}  \exp \left(  
-\int_{-\infty}^{+\infty}\frac{\log(1-\eps\xi |\rho(\xi)|^2)}{\xi-\lambda}\frac{d\xi}{2\pi i}
\right)
\end{equation}
\begin{equation}
\label{transmission-}
\alpha(\lambda)=\prod_{k=1}^N \frac{\lambda-\overline{\lambda}_k}{\lambda-{\lambda}_k}   \exp \left(  
\int_{-\infty}^{+\infty}\frac{\log(1-\eps\xi |\rho(\xi)|^2)}{\xi-\lambda}\frac{d\xi}{2\pi i}\right)
\end{equation}
hold.
\end{lemma}

\begin{proof}
The functions $\balpha(\lambda)$ and $\alpha(\lambda)$ have  simple zeros 
{in $\Lam^+$ and $\overline{\Lam^+}$ respectively}.
Defining
\begin{equation}
\label{trace}
\bgamma(\lambda)=\prod_{k=1}^N\frac{\lambda-\overline{\lambda}_k}{\lambda-\lambda_k}\balpha(\lambda), \,\,  ~~
\gamma(\lambda)=\prod_{k=1}^N\frac{\lambda-\lambda_k}{\lambda-\overline{\lambda}_k}\alpha(\lambda),
\end{equation}
$\bgamma(\lambda)$  is analytic in the upper half plane where it has no zeros,  while $\gamma$ is analytic in the lower half plane where it has no zeros. Also $\bgamma$ and $\gamma$ $\rightarrow$ 1 as $|\lambda|\rightarrow\infty$ in the respective half planes.
Therefore,
\begin{align*}
\log \bgamma(\lambda)=\int_{-\infty}^{+\infty}\frac{\log\bgamma(\xi)}{\xi-\lambda}\frac{d\xi}{2\pi i}\, ,
 \quad\quad 
 \int_{-\infty}^{+\infty}\frac{\log\gamma(\xi)}{\xi-\lambda}\frac{d\xi}{2\pi i}=0 \quad\Im(\lambda)>0 \\
\log \gamma(\lambda)=-\int_{-\infty}^{+\infty}\frac{\log\gamma(\xi)}{\xi-\lambda}\frac{d\xi}{2\pi i}\, ,
 \quad\quad 
 \int_{-\infty}^{+\infty}\frac{\log\bgamma(\xi)}{\xi-\lambda}\frac{d\xi}{2\pi i}=0 \quad\Im(\lambda)<0. 
 \end{align*}
Using \eqref{trace}, as well as the identity $\balpha(\xi)\alpha(\xi)=\bgamma(\xi)\gamma(\xi) = \left(1-\xi|\rho(\xi)|^2\right)^{-1}$, we deduce
\begin{align*}
%\label{trace-balpha}
&\log\balpha(\lambda)=\sum_{k=1}^N\log\left(\frac{\lambda-\lambda_k}{\lambda-\overline{\lambda}_k}   \right)-\int_{-\infty}^{+\infty}\frac{\log(1-\eps\xi |\rho(\xi)|^2)}{\xi-\lambda}\frac{d\xi}{2\pi i},\quad \Im(\lambda)>0, \\
%\label{trace-alpha}
&\log\alpha(\lambda)=\sum_{k=1}^N\log\left(\frac{\lambda-\overline{\lambda}_k}{\lambda-{\lambda}_k}   \right)+\int_{-\infty}^{+\infty}\frac{\log(1-\eps\xi|\rho(\xi)|^2)}{\xi-\lambda}\frac{d\xi}{2\pi i},\quad \Im(\lambda)>0.
\end{align*}
from which the identities \eqref{transmission+} and \eqref{transmission-} are obtained.
\end{proof}
\sidenote{51}\revisedtext{The next lemma can be seen as a weak version of a Plancherel identity for the scattering transform. One should compare it to the following identity for the AKNS system associated the defocussing cubic NLS equation:
$$ -\int_{\R} \log(1-|r(k)|^2) \, dk = \pi \norm[L^2]{q}^2 $$
where $q$ is the potential and $r$ is its scattering transform (see, for example, the discussion in \cite{MTT03}, Appendix A).  For small data this reduces to the linear Plancherel formula. 
In our ``weak Plancherel formula,'' we only obtain equality modulo $2\pi$ because of the exponentials. 
}
\begin{lemma}\label{lem.Plancherel}
Suppose that $q(x,t)$ is the solution of \eqref{DNLS2} for initial data $q_0 \in H^{2,2}(\R)$ and let $\{ \rho, \{(\lam_k, c_k)\}_{k=1}^N \}$ be the scattering data associated to $q_0$. Then the identity
\begin{equation*} 
%\label{Plancherel}
	\exp \lb i \eps \int_\R |q(y,t)|^2 \rb = 
		\exp 
		\lb 
			-4i \lp \sum_{k=1}^N \arg \lam_k \rp 
			- \frac{i}{\pi} \int_\R 
			\frac{ \log(1 - \eps \lam | \rho(\lam)|^2) } {\lam} d\lam 
		\rb
\end{equation*}
holds.	
\end{lemma}

\begin{proof}
We can express the transmission coefficient \eqref{transmission-} in terms of the normalized Jost function matrices $N^{\pm}(x,t;\lam)$ defined by
%%%%%%%%%%%%%%%%%%%%%%%%
%   
%		Some surgery begins here
%
%%%%%%%%%%%%%%%%%%%%%%%%
\sidenote{18}
\begin{equation}
\label{direct.n}
\begin{aligned}
\frac{dN^\pm}{dx}	&=	-i\lam [\sigma_3, N^\pm] + Q_\lam N^\pm -\frac{i}{2}\sigma_3 Q^2 N^\pm\\
\lim_{x \rarr \pm \infty} N^\pm(x,\lam)	&=	\revisedtext{I}%\twomat{1}{0}{0}{1}
\end{aligned}
\end{equation}
(see \eqref{LS}) 
where $q(x)$ is replaced by $q(x,t)$ in the definitions of $Q$ and $Q_\lam$.
Recall that
\begin{equation}
\label{direct.n.jc}
N^+(x,\lam) = N^-(x,\lam) 
		e^{-i\lam x \ad(\sigma_3)} T(\lambda).
\end{equation}
where $T$
%%%%%%%%%%%%%%%%%%%%%%%%
%
%		Surgery ends here
%
%%%%%%%%%%%%%%%%%%%%%%%%
Using the definition \eqref{Jost.T} of $T$ and  \eqref{direct.n.jc} and taking the limit as $x \to -\infty$ gives 
\begin{equation}\label{alpha.jost}
	\alpha(\lam) = \lim_{x \to -\infty} \NN_{11}^+(x,t;\lambda). 
\end{equation}	

Consider \eqref{direct.n}  for $\lam \approx 0$
\sidenote{18}
\begin{equation*}
	\begin{aligned}
	&\od{\NN^+}{x} = 
	  \begin{pmatrix}
	    -\tfrac{i\eps}{2} |q(x,t)|^2 & q(x,t) \\ 
	    0 & \tfrac{i\eps}{2} |q(x,t)|^2
      \end{pmatrix} \NN^+
	  + \lambda 
	  \left[ 
	    \begin{pmatrix} -i & 0 \\ \eps \overline{q(x,t)} & i \end{pmatrix} N^+     
	    + N^+ \begin{pmatrix} i & 0 \\ 0 & -i \end{pmatrix} 
	  \right]
	\\
	&\lim_{x\to \infty} \NN^+(x,t;\lam) = \revisedtext{I}%\begin{pmatrix} 1 &&& 0 \\ 0 &&& 1 \end{pmatrix}.
	\end{aligned} .
\end{equation*}
As $\lam = 0$ is a regular point of this system of equations, one can easily show that 
\begin{equation}\label{N1.near.0}
   \NN^+(x,t;\lam) = 
   e^{ \frac{i \eps \sig}{2} \int_x^\infty |q(y,t)|^2 dy }
   \begin{pmatrix} 
     1 &&& -\dint_{x}^\infty q(y,t) e^{-i\eps \int_{y}^\infty |q(w,t)|^2 dw} dy \\ 
     0 &&& 1 
   \end{pmatrix} + \bigo{\lam}. 
\end{equation}
Writing $\NN_1^+$ for the first column of the Jost function we have
\begin{equation}\label{N1.at.0}
	 \NN^+_1(x,t;0) = \NN_{1-}(0;x,t) = \NN_1(x,t,0),
\end{equation}
where we have used the fact that the Jost function $\NN_{1}^+(x,t;0)$
gives the first column of the Beals-Coifman solution $\NN_-(0;x,t)$ of RHP~\ref{RHP2} (see, for example, 
\cite[Section 4.1]{Liu17} for discussion). We can drop the minus-boundary value because the jump relation in Problem \ref{RHP2}(iii) gives ${\NN_{11}}_+(0;x,t) = {\NN_{11}}_-(0;x,t)$ so that $\NN_{11}(\lam;x,t)$ is continuous at the origin.
Evaluating $\alpha(0)^2$ two ways: by 
%combing 
\sidenote{53}\revisedtext{combining}
\eqref{N1.at.0} with \eqref{alpha.jost}; and evaluating \eqref{transmission-} at $\lambda = 0$, gives the result.
\end{proof}

%\input{appC}						%%	Reflectionless solutions
%%%%%%%%%%%%%%%%%%%%%%%%%%%%%%
%
%		File appC.tex		-	7.22.2017
%
%%%%%%%%%%%%%%%%%%%%%%%%%%%%%%

\section{Solutions of RHP~\ref{RHP2} for reflectionless scattering data}
\label{app:solitons}

The bright soliton solutions of \eqref{DNLS2} can be characterized as the potentials $q(x,t)$ for which the associated scattering data are reflectionless: $\lp  \rho \equiv 0, \{ (\lam_k ,C_k) \}_{k=1}^N \rp$, and $(\lam_k , C_k) \in \C^+ \times \C^\times$ for  $k=1,\ldots,N$. 
If $N=1$, with scattering data $(\lam = u+iv, C)$, the single soliton solution of \eqref{DNLS2} is\sidenote{14}
\begin{multline}\label{1sol}
	\mathcal{Q}_{\mathrm{sol}}(x,t) = 
	\mathcal{Q}(x-x_0 + 4 u t\revisedtext{,\lam}) \\
	\times \exp i \left\{  { 4 }|\lam|^2 t -2{u}(x+4ut) - \frac{\eps}{4}
	\int_{-\infty}^{x-x_0+4u t} \mathcal{Q}(\eta\revisedtext{,\lam})^2 d\eta
	- \varphi_0
	\right\}
\end{multline}
where \sidenote{14}
\begin{gather*}
	\mathcal{Q}(y\revisedtext{,\lam}) = 
	\sqrt{\frac{8 v^2}{|\lam| \cosh(4 v y) - \eps u}}, 
\\[5pt]
	x_0 = \frac{1}{4 v} \log \frac{ |\lam| |C|^2 }{4 v^2},
	\qquad
	\varphi_0 = \arg(\lambda) + \arg(C) + \pi/2,
\end{gather*}
which describes a solitary wave with amplitude envelope $\mathcal{Q}$ traveling at speed $c = -4 \Re \lambda$. For $N > 1$, the solution formulae become ungainly, but we expect, generically, that for $|t| \gg 1$, the solution will resemble $N$ independent 1-solitons each traveling at its unique speed $-4 \Re \lam_k$\footnote{The non-generic case occurs when $\Re \lam_j = \Re \lam_k$ for one or more pairs $j \neq k$. In this case the solution possesses localized, quasi-periodic traveling waves known as breather solitons.}. For this reason, these solutions are called $N$-solitons of \eqref{DNLS2}. 

%Given reflectionless data $\left\{ \{ (\lam_k ,C_k) \}_{k=1}^N \right\}$ and sets
%\begin{equation}
%	\Delta \subset \{ 1, 2 ,\dots, N \}, \qquad \nabla = \{ 1, 2 ,\dots, N \} \setminus \Delta
%\end{equation}
%we define the following problem,

%% PAP altered for CMP
%\newcommand{\NsolD}[1][]{{{\mathcal{N}}_{#1}^{\mathrm{sol,\Delta}} }}
\newcommand{\NsolD}[1][]{{{\mathcal{N}}_{#1}^{\,\mathrm{sol,\Delta}} }}
\newcommand{\vD}[1][]{ {{v}_{#1}^{\Delta}} }

\begin{problem}\label{outmodel2a}
Given $(x,t) \in \R^2$ and data $\calD = \{ (\lam_k ,C_k) \}_{k=1}^N \subseteq \C^+ \times \C^\times$ fix $\Delta \subset \{1,\dots,N\}$. Find an analytic function $\NsolD(\,\cdot\,; x,t \, \vert\, \mathcal{D}): (\C \setminus \poles) \to SL_2(\C)$ such that
\begin{enumerate}
	\item[(i)] $\NsolD$ satisfies the symmetry relation 
	$$
		\NsolD(\lam; x,t \, \vert\, \mathcal{D}) = 
		\lam^{-\sig/2} \sigma_\eps^{-1} \overline{\NsolD(\overline{\lam}; x,t \, \vert\, \mathcal{D})} 
		\sigma_\eps \lam^{\sig/2}
	$$
	\item[(ii)] $\NsolD(\lam; x,t \, \vert\, \mathcal{D}) = \tril{\alpha(x,t)} + \bigo{\lam^{-1}}$ as $\lam \to \infty$.
	\item[(iii)] $\NsolD$ has simple poles at each point in $\poles$. For each $\lam_k \in \poles_+$ 
	\begin{gather}
	%\label{out2 residue}
	\nonumber
			\begin{aligned}
				\res_{\lam = \lam_k} \NsolD(\lam; x,t \, \vert\, \mathcal{D}) 
				&= \lim_{\lam \to \lam_k} 
				  \NsolD(\lam; x,t \, \vert\, \mathcal{D})
				  \vD(\lam_k) \smallskip \\
				\res_{\lam = \overline{\lam_k}} \NsolD(\lam; x,t \, \vert\, \mathcal{D}) 
				&= \lim_{\lam \to \overline{\lam_k}} 
				  \NsolD(\lam; x,t \, \vert\, \mathcal{D}) 
				  \vD(\overline{\lam_k})
			\end{aligned}
\shortintertext{where}
\begin{aligned}
\label{out2a residue matrices}
			\vD(\lambda_k) &= 
			\begin{dcases}
				\tril[0]{ \lam_k \gamma_k^\Delta(x,t) } 
					& k \not\in \Delta 	\\		
				\triu[0]{ \lam_k^{-1} \gamma_k^\Delta(x,t) } 
					& k \in \Delta, 	
			\end{dcases} 
\\[5pt]
	        \vD(\overline{\lam_k}) &= 
	        \overline{
	          \lam^{-\sig/2} \sigma_\eps^{-1} \vD(\lam_k) \sigma_\eps \lam^{\sig/2}
	        },
\end{aligned}
\shortintertext{and}
\nonumber
	\gamma_k^\Delta(x,t) = 
	   \begin{cases}
	      C_k \mathfrak{B}^\Delta(\lam_k)^{-2} e^{-2it\theta(\lam_k)}
	      & k \not\in \Delta \medskip \\
		  C_k^{-1} (1/\mathfrak{B}^\Delta)'(\lam_k)^{-2} e^{2it\theta(\overline{\lam_k})} 
		  & k \in \Delta .
	   \end{cases}
\\[5pt]
\nonumber
	\mathfrak{B}^\Delta(\lam) = \prod_{k \in \Delta} 
	   \lp \frac{\lam - \overline{\lam}_k}{ \lam - \lam_k} \rp   
\end{gather}
\end{enumerate}
\end{problem}
\begin{remark}
When the context is clear we will omit the dependence of $\NsolD$ on $x$,$t$, and/or $\calD$ so that $\NsolD(\lam) = \NsolD(\lam,\ \vert\, \calD) = \NsolD(\lam;x,t\, \vert\, \calD)$ are all equivalent representations of the same function.
\end{remark}

When $\Delta = \emptyset$ Problem~\ref{outmodel2a} is identical to RHP~\ref{RHP2} with scattering data $\{ \rho \equiv 0, \calD=\{ (\lam_k, C_k) \}_{k=1}^N \}$. For any other choice of $\Delta$ the relation between these problems is\footnote{When $\Delta = \emptyset$ we set $\mathfrak{B}^\emptyset(\lam) = I$ for consistency.}
\begin{equation}\label{Blaschke}	
	\mathcal{N}^{\mathrm{sol},\emptyset}(\lam) =  
	\NsolD(\lam) \mathfrak{B}^\Delta(\lam)^{\sig}.
\end{equation}
Each choice of $\Delta$ represents a different normalization of the RHP. The idea, which is common in the literature \cite{BJM16,BKMM,DKKZ}, is to choose normalizations which prepare the problem for asymptotic analysis. 

For convenience, order the spectrum (possibly non-uniquely) such that 
\begin{equation*}
	\Re \lam_1 \leq \Re \lam_2 \leq \dots \leq \Re \lam_N,
\end{equation*} 
and let $\Re \lam_{N+1} = -\Re \lam_{0} = +\infty$. 
For $j=0,1,\dots,N$, define the sets
\begin{equation*}
	\begin{gathered}
	S_j^\pm = \{ (x,t) \in \R^2 \,:\, \xi \in [\Re z_j, \Re z_{j+1}), \ \pm t \geq 0 \}, \\
	\Delta_j^+ = \{ \ell \in \N \,:\, 1 \leq \ell \leq j \} , \qquad
	\Delta_j^- = \{1, \dots, N\} \setminus \Delta_j^+ 
	\end{gathered}
\end{equation*}	
As $\eta t \to \infty$ with $(x,t) \in S_j^\eta$, the set of $\lambda_k$ which have exponentially growing residue coefficients in $\mathcal{N}^{\textrm{sol},\emptyset}$ are indexed by $\Delta^\eta_j$---the remaining poles have bounded residues. The transformation \eqref{Blaschke} results in a new problem for $\mathcal{N}^{\textrm{sol},\Delta_j^\eta}$ which has, in light of \eqref{phase.lambda}, and \eqref{out2a residue matrices}, residue coefficients $\gamma_k^{\Delta_j^\eta}(x,t)$ which are uniformly bounded as $(x,t)$ vary over $S_j^\eta$:
\begin{equation}\label{gamma.bounds}
	\begin{gathered}
	|\gamma_k^{\Delta_j^\eta}(x,t)|  \leq 
	K_\poles
	\left.
	\begin{cases}  
		e^{-8|t| \Im \lam_k |\Re \lam_k - \xi | }
		& |\xi| \leq \Xi_0 \\
		e^{-2|x| \Im \lam_k |1 - \xi^{-1} \Re \lam_k| }
		& |\xi| \geq \Xi_0 
	\end{cases}
	\right\}
	\leq K_\poles, 
	%\qquad 
	\\
	\forall k, \ (x,t) \in S_j^\eta
	\end{gathered}
\end{equation}
where $\Xi_0>0$ is any fixed constant and $K_\poles$ is a fixed constant depending only on the scattering data.

\begin{lemma}
\label{lem:sol.bound}
\sidenote{54}Given data 
$\revisedtext{\calD = \{ (\lam_k, C_k) \}_{k=1}^N} \subset \C^+ \times \C^\times$ such that $\lam_j \neq \lam_k$ for $j \neq k$ there exists a unique solution of Problem~\ref{outmodel2a} for each $(x,t) \in \R^2$. Moreover, the solution satisfies 
\begin{equation*} 
%\label{bound-nsol}
	\| \lp \mathcal{N}^{\,\mathrm{sol},\Delta} \rp^{-1} \|_{L^\infty(\C \setminus \mathcal{B}_\poles) }
	\lesssim 1
\end{equation*}
where $\mathcal{B}_\poles$ is any open neighborhood of the poles $\Lambda = \{ \lam_k, \overline{\lam_k} \}_{k=1}^N$, and the implied constant depends only on $\mathcal{B}_\poles$ and the  scattering data; it is independent of $x,t$ and $\Delta$. 
\end{lemma}

\begin{proof}
Since $\det \mathcal{N}^{\,\textrm{sol},\Delta}(\lam) = 1$ we only need to consider $\|\mathcal{N}^{\textrm{sol},\Delta}\|$.
The relation \eqref{Blaschke} is bounded and invertible away from $\Lambda$, so it is sufficient to work with whichever choice of $\Delta$ is convenient for any given $(x,t)$. As noted previously, taking $\Delta = \emptyset$, Problem~\ref{outmodel2a} is exactly RHP~\ref{RHP2} for scattering data $\{ \rho \equiv 0, \calD \}$.
Existence and uniqueness in the case $\eps=-1$ follow from Theorem 4.3 of 
%Paper III
\sidenote{55}\revisedtext{\cite{JLPS17b}}
and the equivalence of Problems \ref{RHP2} and \ref{RHP2.row}.
%for which existence and uniqueness was proved in Propositions~\ref{prop:RHP2.exist}-\ref{prop:RHP2.unique}. 
To prove boundedness, observe that the solution $\NsolD$ of Problem~\ref{outmodel2a} is a continuous function of the parameters $\gamma_k^\Delta$. For $(x,t) \in S_j^\eta $, fix $\Delta = \Delta_j^\eta$, then \eqref{gamma.bounds} shows that the parameters $\gamma_k^{\Delta_j^\eta}$ vary over compact sets. This establishes boundedness on each $S_j^\eta$, and since $\bigcup_{j=0}^N (S_j^+\cup S_j^-) = \R^2$, this completes the proof. 
\end{proof}

Let us now consider the asymptotic behavior of soliton solutions which is needed for Theorems~\ref{thm:long-time}-\ref{thm:long-time-gauge}. Recall the notation established by \eqref{distances}, \eqref{S.cone}-\eqref{posint}, and \eqref{nu0}.
\begin{lemma}\label{lem:Nsol.asymp}
Fix reflectionless data $\calD = \{ (\lam_k, C_k) \}_{k=1}^N$; parameters $v_1\leq v_2$, $x_1\leq x_2$, and a cone $\mathcal{S}(v_1,v_2,x_1,x_2)$ as in Theorem~\ref{thm:long-time}. Let $I = [-v_2/4, -v_1/4]$, $\eta = \sgn t$, and take 
\sidenote{12}$\revisedtext{\Lambda^+(I)}$ and $N(I)\leq N$ as in \eqref{LamI}--\eqref{NI}. Then, as $|t| \to \infty$ with $(x,t) \in \mathcal{S}(v_1,v_2,x_1,x_2)$,
\begin{equation*}
	\mathcal{N}^{\mathrm{sol}, \emptyset}(\lam \,\vert\, \mathcal{D}) =
	\lb I + \bigo{e^{-4\poledist \nu_0 |t| }} \rb
	\mathcal{N}^{\mathrm{sol},\emptyset}(\lam \,\vert\, \widehat{\mathcal{D}})
	\ \prod_{\mathclap{\Re \lam_k \in \negint \setminus I }}
	 \quad \  
	  \lp \frac{\lam - \lambar_k}{\lam-\lam_k} \rp^{\sig} 
\end{equation*}
where
\[
	\widehat{\mathcal{D}} = 
	\Bigg\{ 
	  (\lam_k, \widehat{C}_k) \,:\, \lam_k \in \poles(I), \quad 
	  \widehat{C}_k = C_k \ \prod_{\mathclap{\Re \lam_j \in \negint \setminus I }} 
	  \quad \ 
	  \lp \frac{\lam_k - \lam_j}{\lam_k-\lambar_j} \rp^2 
	\Bigg\}.
\]
\end{lemma}

Using \eqref{q.lam} and \eqref{N1.near.0}-\eqref{N1.at.0} applied to $\mathcal{N}^{\textrm{sol},\emptyset}(\lam \, \vert \, \calD)$ we have the immediate corollary
\begin{corollary}
\label{cor:Nsol.asymp} 
	Under the assumption of Lemma~\ref{lem:Nsol.asymp} it follows that as $|t| \to \infty$ with $(x,t) \in \mathcal{S}(v_1,v_2,x_1,x_2)$ we have:
\begin{align*}
	&\lim_{\lam \to \infty} 2i\lam \Nsol[12]{\emptyset}(\lam;x,t \,|\, \calD) 
	   = \qsol(x,t;\calD) 
	   = \qsol(x,t;\widehat{\calD}) \lp 1 + e^{-4\poledist \nu_0 |t| } \rp \\
    &\begin{aligned}[t]
	  \Nsol[11]{\emptyset}(0;x,t \,|\, \calD)^{-2}
	  &= \exp \Big( -i\eps \int_{x}^\infty | \qsol(y,t;\calD)|^2 dy \Big) \\
	  &= \exp \Big( -i\eps \int_{x}^\infty | \qsol(y,t; \widehat{\calD})|^2 dy  
	  + 4i \sum_{\mathclap{\Re \lam_k \in \negint \setminus I}} \arg \lam_k \Big)
%% PAP put new line here
\\
	 &\quad \times  \lp 1 + e^{-4\poledist \nu_0 |t| } \rp. 
   \end{aligned}  
\end{align*}
\end{corollary}

\begin{proof}[Proof of Lemma~\ref{lem:Nsol.asymp}]
Let 
$
	\Delta_I^+ = \{ k \,:\, \Re \lam_k < -v_2/4 \} 
$
and
$
	\Delta_I^- = \{ k \,:\, \Re \lam_k > -v_1/4 \}. 
$
Consider Problem~\ref{outmodel2a} with $\Delta = \Delta_I^\eta$. For $\lam_j \in \poles \setminus \poles(I)$ and $(x,t) \in \mathcal{S}$ the residue coefficients \eqref{out2a residue matrices} satisfy
\begin{equation}\label{asymp.residue.bound}
	\| v^{\Delta_I^\pm}(\lam_j) \| 
	\sim e^{-8t \Im \lam_j \Re(\lam_j-\xi)}
	= \bigo{e^{-4\poledist \nu_0 |t|}}, \qquad 
	% \lam_j \in \poles \setminus \poles(I), \quad (x,t) \in \mathcal{S},  \quad
	\pm t \to +\infty.
\end{equation}
Introduce small disks $D_k$ around each $\lam_k \in \poles^+ \setminus \poles(I)$ whose radii are chosen sufficiently small that they are non-overlapping and do not intersect the real axis. Orient their boundaries, $\partial D_k$, positively.
We trade these residues for near-identity jumps by the change of variables
\begin{equation}\label{sol.interpolation}
	\widetilde{\mathcal{N}}^{\, \mathrm{sol},\Delta_I^\eta}(\lam \,\vert\, \calD) =
	\begin{cases}
	\mathcal{N}^{\, \mathrm{sol},\Delta_I^\eta}(\lam \,\vert\, \calD) 
	\lp I - \frac{ v^{\Delta_I^\eta}(\lam_k) }{\lam - \lam_k} \rp	
	& z \in D_k  \\
	\mathcal{N}^{\, \mathrm{sol},\Delta_I^\eta}(\lam \,\vert\, \calD) 
	\lp I - \frac{ v^{\Delta_I^\eta}(\overline{\lam_k}) }{\lam - \overline{\lam_k}} \rp
	& z \in \overline{D_k} \\
	\mathcal{N}^{\, \mathrm{sol},\Delta_I^\eta}(\lam \,\vert\, \calD) .
	& \text{elsewhere}	
	\end{cases}
\end{equation}
The new unknown $\widetilde{\mathcal{N}}^{\, \mathrm{sol},\Delta_I^\eta}(\lam \,\vert\, \calD)$ has jumps across each disk boundary 
$$
	\widetilde{\mathcal{N}}^{\, \mathrm{sol},\Delta_I^\eta}_+(\lam \,\vert\, \calD) =
	\widetilde{\mathcal{N}}^{\, \mathrm{sol},\Delta_I^\eta}_-(\lam \,\vert\, \calD) 
	\widetilde{v}(\lam)
$$
where $\widetilde{v}(\lam)$ is given on $\partial D_k$ (resp. $\overline{\partial D_k}$) by the  last factor in the first (resp. second) line of \eqref{sol.interpolation}. 
By virtue of \eqref{asymp.residue.bound} these jumps satisfy
\begin{equation}\label{sol.jump.small}
	\|  \widetilde{v}(\lam) - I  \|_{L^{\infty}(\widetilde{\Sigma}) } 
	= \bigo{ e^{-4\poledist \nu_0 |t|}},
	\qquad
	\widetilde{\Sigma} = 
	\bigcup_{\mathclap{\lam_k \in \poles^+ \setminus \poles(I)}} \, 
	( \partial D_k \cup \overline{\partial{D}}_k ).
\end{equation}

Next we observe that $\mathcal{N}^{\, \mathrm{sol},\emptyset}(\lam \,\vert\, \widehat{\calD})$ has the same poles as $\widetilde{\mathcal{N}}^{\, \mathrm{sol},\Delta_I^\eta}_+(\lam \,\vert\, \calD)$ with exactly the same residue conditions. A simple calculation shows that the quantity
\begin{equation}\label{sol.error.def}
	e(\lam) = \widetilde{\mathcal{N}}^{\, \mathrm{sol},\Delta_I^\eta}_+(\lam \,\vert\, \calD)
	\lb \mathcal{N}^{\, \mathrm{sol},\emptyset}(\lam \,\vert\, \widehat{\calD}) \rb^{-1} 
\end{equation}
has no poles with jump $e_+(\lam) = e_-(\lam) v_e(\lam)$. Here
\begin{equation*}
	v_e(\lam) = 
	  \lb \mathcal{N}^{\, \mathrm{sol},\emptyset}(\lam \,\vert\, \widehat{\calD}) \rb 
	  \widetilde{v}(\lam)
	  \lb \mathcal{N}^{\, \mathrm{sol},\emptyset}(\lam \,\vert\, \widehat{\calD}) \rb^{-1}, 
	  \quad
	  \lam \in \widetilde{\Sigma}
\end{equation*}
satisfies an estimate identical to \eqref{sol.jump.small} by  Lemma~\ref{lem:sol.bound} applied to $\mathcal{N}^{\, \mathrm{sol},\emptyset}(\lam \,\vert\, \widehat{\calD})$. Using the theory of small norm RHPs \cite{Zhou89,TO16}, one shows that $e$ exists and that 
$e(\lam) = I + \bigo{ e^{-4\poledist \nu_0 |t|}}$ for all sufficiently large $|t|$. The result then follows from \eqref{sol.interpolation}, \eqref{sol.error.def} and \eqref{Blaschke}.
\end{proof}

%\input{bib}						%%	Bibliography
%%%%%%%%%%%%%%%%%%%%%%%%%%%%%%%%%%%%%%%
%		
%		BIBLIOGRAPHY
%		File dnls-imrn-arXiv-bib.tex
%		Created 4.7.2017
%
%%%%%%%%%%%%%%%%%%%%%%%%%%%%%%%%%%%%%%% 

\end{document}